\documentclass[reqno, 10pt]{amsart}
\RequirePackage[colorlinks,linkcolor=blue,citecolor=blue,urlcolor=blue]{hyperref}
\usepackage[top=3.8cm, bottom=2.3cm, left=3cm, right=3cm]{geometry}

\usepackage{hyperref}
\usepackage[centertags]{amsmath}
\usepackage{amsmath}
\usepackage[dutch,british]{babel}
\usepackage{amsfonts}
\usepackage{amssymb}
\usepackage{amsthm}
\usepackage{multirow}
\usepackage{newlfont}
\usepackage{color}
\usepackage{lipsum}

\usepackage{authblk}

\allowdisplaybreaks

\usepackage{stackengine}
\stackMath

\newcommand\ttilde[2][1]{%
 \def\useanchorwidth{T}%
  \ifnum#1>1%
    \stackon[0pt]{\ttilde[\numexpr#1-1\relax]{#2}}{\hspace{0.2em}\scriptscriptstyle\thicksim}%
  \else%
    \stackon[1pt]{#2}{\hspace{0.2em}\scriptscriptstyle\thicksim}%
  \fi%
}

\usepackage{bbm}
\include{amsthm_sc}

\usepackage{stmaryrd}
	 \SetSymbolFont{stmry}{bold}{U}{stmry}{m}{n}
\usepackage{mathrsfs}

\usepackage{fancyhdr}
\usepackage{fancybox}
\usepackage{setspace}
\usepackage{cleveref}

\theoremstyle{plain} %plain, definition, remark
\newtheorem{theorem}{Theorem}[section]
\newtheorem*{theorem*}{Theorem}
\newtheorem{lemma}[theorem]{Lemma}
\newtheorem{corollary}[theorem]{Corollary}
\newtheorem*{corollary*}{Corollary}
\newtheorem{proposition}[theorem]{Proposition}
\newtheorem*{proposition*}{Proposition}
\newtheorem{definition}[theorem]{Definition}
\newtheorem*{definition*}{Definition}
\newtheorem{assumption}[theorem]{Assumption}

\theoremstyle{definition} %plain, definition, remark

\newtheorem*{example*}{Example}
\newtheorem{remark}[theorem]{Remark}

\newtheorem*{remark*}{Remark}
\newtheorem*{remarks*}{Remarks}

\newcommand{\deq}{\mathrel{\mathop:}=}
\newcommand{\e}[1]{\mathrm{e}^{#1}}
\newcommand{\R} {\mathbb{R}}
\newcommand{\C} {\mathbb{C}}
\newcommand{\N} {\mathbb{N}}

\newcommand{\E} {\mathbb{E}}

\newcommand{\p} {\mathbb{P}}
\newcommand{\bra}[1]{\langle #1 |}

\newcommand{\ket}[1]{| #1 \rangle}

\DeclareMathOperator{\Tr}{Tr}

\DeclareMathOperator{\supp}{supp}

\DeclareMathOperator{\re}{\mathrm{Re}}
\DeclareMathOperator{\im}{\mathrm{Im}}

\newcommand{\caD}{{\mathcal D}}
\newcommand{\caE}{{\mathcal E}}

\newcommand{\caN}{{\mathcal N}}

\newcommand{\opunit}{\text{1}\kern-0.22em\text{l}}

\newcommand{\bsu}{{\boldsymbol u}}

\newcommand{\cf}{{\it cf.}\;}
\newcommand{\eg}{{\it e.g.}\;}
\newcommand{\ie}{{\it i.e.}\;}

\newcommand{\wt}{\widetilde}
\newcommand{\ol}{\overline}
\newcommand{\wh}{\widehat}

\newcommand{\nm}{s}
\newcommand{\nmf}{s^{(4)}}

\newcommand{\msc}{m_\mathrm{sc}}
\newcommand{\bee}{ \begin{equation} }
\newcommand{\eeq	}{ \end{equation} }

\newcommand{\GOE}{\mathrm{GOE}}

\newcommand{\lone}{\mathbbm{1}}

\newcommand{\ben}{\begin{arabicenumerate}}
\newcommand{\een}{\end{arabicenumerate}}

\newcommand{\dd}{\mathrm{d}}
\newcommand{\ii}{\mathrm{i}}

\renewcommand{\P}{\mathbb{P}}

\newcommand{\Phiepsi}{\Phi_\epsilon}

\newcommand{\Ir}{{I}}

\newcommand{\id}{\mathrm{I}}

\numberwithin{equation}{section} %\numberwithin{lemma}{section}
\numberwithin{theorem}{section}

\newcommand\blfootnote[1]{%
  \begingroup
  \renewcommand\thefootnote{}\footnote{#1}%
  \addtocounter{footnote}{-1}%
  \endgroup
}

\begin{document}

\renewcommand{\thefootnote}{\fnsymbol{footnote}}	

\begin{minipage}{0.85\textwidth}
 \vspace{2.2cm}
    \end{minipage}
\begin{center}
\Large\bf
Local law and Tracy--Widom limit for sparse random matrices\blfootnote{

{{ \textsuperscript{$*$}Supported by Samsung Science and Technology Foundation project number SSTF-BA1402-04}\\
{\textsuperscript{$\dagger$}Supported by ERC Advanced Grant RANMAT No.\ 338804}}\\
{{\phantom{\textsuperscript{$\dagger$}}\textit{Keywords}:  Local law, sparse random matrices, Erd\H{o}s--R\'enyi graph}\\
{\phantom{\textsuperscript{$\dagger$}}\textit{AMS Subject Classification (2010)}: 46L54, 60B20}\\
{\phantom{\textsuperscript{$\dagger$}}\textit{Date}: 27th May 2016}\\
}}
\end{center}

\vspace{1.4cm}
\begin{center}
 \begin{minipage}{0.45\textwidth}
\begin{center}
Ji Oon Lee\textsuperscript{$*$}  \\
\footnotesize { KAIST }\\
{\it jioon.lee@kaist.edu}
\end{center}
\end{minipage}
\begin{minipage}{0.45\textwidth}
 \begin{center}
Kevin Schnelli\textsuperscript{$\dagger$}\\
\footnotesize 
{IST Austria}\\
{\it kevin.schnelli@ist.ac.at}
\end{center}
\end{minipage}

\end{center}

\vspace{1.1cm}

\begin{center}
 \begin{minipage}{0.9\textwidth}
\small
\hspace{10pt}
We consider spectral properties and the edge universality of sparse random matrices, the class of random matrices that includes the adjacency matrices of the Erd\H{o}s--R\'enyi graph model $G(N,p)$. We prove a local law for the eigenvalue density up to the spectral edges. Under a suitable condition on the sparsity, we also prove that the rescaled extremal eigenvalues exhibit GOE Tracy--Widom fluctuations if a deterministic shift of the spectral edge due to the sparsity is included. For the adjacency matrix of the Erd\H{o}s--R\'{e}nyi graph this establishes the Tracy--Widom fluctuations of the second largest eigenvalue for $p\gg N^{-2/3}$ with a deterministic shift of order $(Np)^{-1}$. 
 \end{minipage}
\end{center}
 \date{27th May 2016}
 \vspace{7mm}

\thispagestyle{headings}

\section{Introduction} \label{sec:intro}

We consider spectral properties of sparse random matrices.  One of the most prominent examples in the class of sparse random matrices is the (centered) adjacency matrix of the Erd\H{o}s--R\'{e}nyi graph on~$N$ vertices, where an edge is independently included in the graph with a fixed probability $p\equiv p(N)$. Introduced in~\cite{ER,ERG2,Gi}, the Erd\H{o}s--R\'{e}nyi graph model $G(N,p)$ serves as a null model in the theory of random graphs and has numerous applications in many fields including network theory. Information about a random graph can be obtained by investigating its adjacency matrix, especially the properties of its eigenvalues and eigenvectors.

The sparsity of a real symmetric $N$ by $N$ random matrix may be measured by the sparsity parameter $q\equiv q(N)$, with $0\le q\le N^{1/2}$, such that the expected number of non-vanishing entries is $q^2$. For example for the adjacency matrices of the Erd\H{o}s--R\'enyi graph we have $q^2 \simeq Np$, while for standard Wigner matrices we have $q=N^{1/2}$. We call a random matrix sparse if $q$ is much smaller than $N^{1/2}$.

For Wigner matrices, one of the fundamental inputs in the proof of universality results is the local semicircle law~\cite{ESY1,ESY2,ESY3,EYY}, which provides an estimate of the local eigenvalue density down to the optimal scale. The framework built on the local law can also help understanding the spectral properties of sparse random matrices~\cite{EKYY1}. However, in contrast to Wigner matrices, the local eigenvalue density for a sparse random matrix depends on its sparsity. For this reason, the universality of the local eigenvalue statistics for sparse random matrices was proved at first only for $q\ge N^{1/3}$ in~\cite{EKYY1,EKYY2}. Recently, bulk universality was proved in~\cite{HLY} under the much weaker condition $q \geq N^{\epsilon}$, for any $\epsilon > 0$. The main obstacle in the proof of the edge universality is that the local law obtained in~\cite{EKYY1} deteriorates at the edge of the spectrum. 

Our first main results is a local law for sparse random matrices up to the edge. More precisely, we show a local law for the eigenvalue density in the regime $q\ge N^\epsilon$, for arbitrarily small $\epsilon>0$. The main observation is that, although the empirical spectral measure of sparse random matrices converges in the large $N$ limit to the semicircle measure, there exists a deterministic correction term that is not negligible for large but finite $N$. As a result, we establish a local law that compares the empirical spectral measure not with the semicircle law but with its refinement. (See Theorem~\ref{theorem:local} and Corollary~\ref{corollary density of states} for more detail.)

The largest eigenvalue $\mu_1$ of a real symmetric $N$ by $N$ Wigner matrix (whose entries are centered and have variance $1/N$) converges almost surely to two under the finite fourth-moment condition, and $N^{2/3}(\mu_1 -2)$ converges in distribution to the GOE Tracy--Widom law. For sparse random matrices the refinement of the local semicircle law reveals that the eigenvalues at the upper edge of the spectrum fluctuate around a deterministic number larger than two, and the shift is far greater than $N^{-2/3}$, the typical size of the Tracy--Widom fluctuations.

Our second main result is the edge universality that states the limiting law for the fluctuations of the rescaled largest eigenvalues of a (centered) sparse random matrices is given by the Tracy--Widom law when the shift is taken into consideration and if $q\ge N^{1/6+\epsilon}$, where $\epsilon>0$ is arbitrarily small. We expect the exponent one-sixth to be critical. (See Theorem~\ref{thm tw} and the discussion below it for more detail.) For the adjacency matrices of the Erd\H{o}s--R\'{e}nyi graphs, the sparsity conditions corresponds to $p\ge N^{-2/3 + \epsilon}$, for any $\epsilon > 0$, and our result then assures that the rescaled second largest eigenvalue has GOE Tracy--Widom fluctuations; see Corollary~\ref{cor: TW for A}.

In the proof of the local law, we introduce a new method based on a recursive moment estimate for the normalized trace $m$ of the Green function, \ie  we recursively control high moments of $|P(m)|$, for some polynomial $P$, by using lower moments of $|P(m)|$, instead of fully expanding all powers of~$m$; see Section~\ref{sec:outline} for detail. This recursive computation relies on cumulant expansions which were used in the random matrix theory literature many times, especially in the study of linear eigenvalue statistics~\cite{KKP,LP}.

Our proof of the Tracy--Widom limit of the extremal eigenvalues relies on the Green function comparison method~\cite{EYY2,EYY}. However, instead of applying the conventional Lindeberg replacement approach, we use a continuous flow that interpolates between the sparse random matrix and the Gaussian Orthogonal Ensemble (GOE). The main advantage of using a continuous interpolation is that we may estimate the rate of change of $m$ along the flow even if the moments of the entries in the sparse matrix are significantly different from those of the entries in the GOE matrix. The change of $m$ over time is offset by the shift of the edge. A similar idea was used in the proof of edge universality of other random matrix models in~\cite{LS14,LS14b}.

This paper is organized as follows: In Section~\ref{sec:results}, we define the model, present the main results and outline applications to adjacency matrices of the Erd\H{o}s--R\'{e}nyi graph ensemble. In Section~\ref{sec:outline}, we explain the main strategy of our proofs. In Section~\ref{sec:rho}, we prove several properties of the deterministic refinement of Wigner's semicircle law. In Section~\ref{sec:proof main}, we prove the local law using our technical result on the recursive moment estimate, Lemma~\ref{lem:claim}. In Section~\ref{sec:stein}, we prove Lemma~\ref{lem:claim} with technical detail. In Section~\ref{sec:tw}, we prove our second main result on the edge universality using the Green function comparison method.

{\it Notational conventions:}
We use the symbols $O(\,\cdot\,)$ and $o(\,\cdot\,)$ for the standard big-O and little-o notation. The notations $O$, $o$, $\ll$, $\gg$, refer to the limit $N\to \infty$ unless otherwise stated. Here $a\ll b$ means $a=o(b)$. We use~$c$ and~$C$ to denote positive constants that do not depend on $N$, usually with the convention $c\le C$. Their value may change from line to line. We write $a\sim b$, if there is $C\ge1$ such that $C^{-1}|b|\le |a|\le C |b|$. Throughout the paper we denote for $z\in\C^+$ the real part by $E=\re z$ and the imaginary part by $\eta=\im z$. For $a\in\R$, we let $(a)_+=\max (0,a)$, and $(a)_-=  -\min(a,0)$. Finally, we use double brackets to denote index sets, \ie for $n_1,n_2\in\mathbb{R}$, $\llbracket  n_1,n_2\rrbracket:=[n_1,n_2]\cap \mathbb{Z}$.

{\it Acknowledgement:} We thank L\'aszl\'o Erd\H{o}s for useful comments and suggestions. Ji Oon Lee is grateful to the department of mathematics, University of Michigan, Ann Arbor, for their kind hospitality during the academic year 2014--2015.

\section{Definitions and main results} \label{sec:results}

\subsection{Motivating examples}\label{subsection motivating examples}

\subsubsection{Adjacency matrix of Erd\H os--R\'enyi graph}\label{le subsection adjacency matrix ER graph}
One motivation for this work is the study of adjacency matrices of the {\it Erd\H os--R\'enyi random graph model} $G(N,p)$. The off-diagonal entries of the adjacency matrix associated with an Erd\H os--R\'enyi graph are independent, up to the symmetry constraint, Bernoulli random variables with parameter $p$, \ie the entries are equal to $1$ with probability~$p$ and~$0$ with probability $1-p$. The diagonal entries are set to zero, corresponding to the choice that the graph has no self-loops. Rescaling this matrix ensemble so that the bulk eigenvalues typically lie in an order one interval we are led to the following random matrix ensemble. Let $A$ be a real symmetric $N\times N$ matrix whose entries, $A_{ij}$, are independent random variables (up to the symmetry constraint $A_{ij}=A_{ji}$) with distributions
\begin{align}\label{le probability AER}
 \P \Big( A_{ij}=\frac{1}{\sqrt{Np(1-p)}}\Big)=p\,, \quad\qquad \P( A_{ij}=0)= 1-p\,,\qquad\quad\P ( A_{ii}=0)=1\,,\qquad  (i\not=j)\,.
 \end{align}
Note that the matrix $ A$ typically has $N(N-1)p$ non-vanishing entries. For our analysis it is convenient to extract the mean of the entries of $A$ by considering the matrix $\wt A$ whose entries, $\wt A_{ij}$, have distribution
\begin{align*}
 \p \Big( \wt A_{ij} = \frac{1-p}{\sqrt{Np(1-p)}} \Big) = p\,,\qquad \p \Big( \wt A_{ij} = -\frac{p}{\sqrt{Np(1-p)}} \Big) = 1-p\,, \qquad \p\big( \wt A_{ii}=0\big)=1\,,
\end{align*}
with $i\not=j$. A simple computation then reveals that
\begin{align}\label{assumption ER1}
\E \wt A_{ij} = 0\,,\qquad \qquad \E \wt A_{ij}^2 =\frac{1}{ N}\,,
\end{align}
and
\begin{align}\label{assumption ER2}
 \E \wt A_{ij}^k = \frac{(-p)^k (1-p) + (1-p)^k p}{(Np(1-p))^{k/2}} =\frac{1}{ N d^{(k-2)/2}}(1 + O(p))\,,\qquad\qquad (k\ge 3)\,,
\end{align}
with $i\not=j$, where $d\deq pN$ denotes the~{\it expected degree} of a vertex, which we allow to depend on $N$.  As already suggested by~\eqref{assumption ER2}, we will assume that $p\ll 1$.

\subsubsection{Diluted Wigner matrices}\label{sec:diluted w}
Another motivation for this work are {\it diluted Wigner matrices}. Consider the matrix ensemble of real symmetric $N\times N$ matrices of the form
\begin{align}
 D_{ij}=B_{ij}V_{ij} \,,\qquad\qquad (1\le i\le j\le N)\,,
\end{align}
where $(B_{ij}:i\le j)$ and $(V_{ij}:i\le j)$ are two independent families of independent and identically distributed random variables. The random variables $(V_{ij})$ satisfies $\E V_{ij}^2=1$ and $\E V_{ij}^{2k}\le (Ck)^{ck}$, $k\ge 4$, for some constants $c$ and $C$, and their distribution is, for simplicity, often assumed to be symmetric. The random variables $B_{ij}$ are chosen to have a Bernoulli type distribution given by
\begin{align}
 \P  \Big( B_{ij} = \frac{1}{\sqrt{Np}} \Big) =p\,, \qquad \p ( B_{ij} = 0) =1- p\,,\qquad \p(B_{ii}=0)=1\,,
\end{align}
with $i\not=j$. We introduce the {\it sparsity parameter $q$} through 
 \begin{align}\label{le q0}
  p=\frac{q^2}{N}\,,
 \end{align}
with $0<q\le N^{1/2}$. We allow $q$ to depend on $N$. We refer to the random matrix $D=(D_{ij})$ as a {\it diluted Wigner matrix} whenever $q\ll N^{1/2}$. For $q=N^{1/2}$, we recover the usual Wigner ensemble.\newpage

\subsection{Notation}
In this subsection we introduce some of the notation and conventions used.
\subsubsection{Probability estimates}

We first introduce a suitable notion for high-probability estimates.

\begin{definition}[High probability event]
We say that an $N$-dependent event $\Xi \equiv \Xi^{(N)}$ holds with high probability if, for any (large) $D>0$,
\begin{align}
\p\big(\Xi^{(N)}\big) \geq 1 - N^{-D}\,,
\end{align}
for sufficiently large $N\ge N_0(D)$. 
\end{definition}
\begin{definition}[Stochastic domination] \label{def:domination}
Let $X\equiv X^{(N)}$, $Y\equiv Y^{(N)}$ be $N$-dependent non-negative random variables. We say that $Y$ stochastically dominates $X$ if, for all (small) $\epsilon>0$ and (large)~$D>0$,
\begin{align}
\P\big(X^{(N)}>N^{\epsilon} Y^{(N)}\big)\le N^{-D}\,,
\end{align}
for sufficiently large $N\ge N_0(\epsilon,D)$, and we write $X \prec Y$.
 When
$X^{(N)}$ and $Y^{(N)}$ depend on a parameter $u\in U$ (typically an index label or a spectral parameter), then $X(u) \prec Y (u)$, uniformly in $u\in U$, means that the threshold $N_0(\epsilon,D)$ can be chosen independently of $u$. A slightly modified version of stochastic domination appeared first in~\cite{EKY}.
\end{definition}

In Definition~\ref{def:domination} and hereinafter we implicitly choose $\epsilon>0$ strictly smaller than $\phi/10>0$, where $\phi>0$ is the fixed parameter appearing in~\eqref{le phi} below.

The relation $\prec$ is a partial ordering: it is transitive and it satisfies the arithmetic rules of an order relation, {\it e.g.}, if $X_1\prec Y_1$ and $X_2\prec Y_2$ then $X_1+X_2\prec Y_1+Y_2$ and $X_1 X_2\prec Y_1 Y_2$. Furthermore, the following property will be used on a few occasions: If $\Phi(u)\ge N^{-C}$ is deterministic, if $Y(u)$ is a nonnegative random variable satisfying $\E [Y(u)^2]\le N^{C'}$ for all~$u$, and if $Y(u) \prec \Phi(u)$ uniformly in~$u$, then, for any $\epsilon>0$, we have $\E [Y(u)] \le N^\epsilon \Phi(u)$ for $N\ge N_0(\epsilon)$, with a threshold independent of $u$. This can easily be checked since 
$$
\E [Y(u) \lone(Y(u) > N^{\epsilon/2} \Phi)] \leq \left(\E[ Y(u)^2] \right)^{1/2} \big( \p[ Y(u) > N^{\epsilon/2} \Phi] \big)^{1/2} \leq N^{-D}\,,
$$
for any (large) $D > 0$, and $\E [Y(u) \lone(Y(u) \leq N^{\epsilon/2} \Phi(u))] \leq N^{\epsilon/2} \Phi(u)$, hence $\E [Y(u)] \leq N^{\epsilon} \Phi(u)$.

\subsubsection{Stieltjes transform} Given a probability measure $\nu$ on $\R$, we define its Stieltjes transform as the analytic function $m_\nu\,:\,\C^+\rightarrow \C^+$, with $\C^+\deq\{ z=E+\ii\eta\,:\, E\in\R, \eta>0\}$, defined by
\begin{align}
 m_{\nu}(z)\deq\int_\R\frac{\dd\nu(x)}{x-z}\,,\qquad\qquad (z\in\C^+)\,.
\end{align}
Note that $\lim_{\eta\nearrow 0}\ii \eta \, m_{\nu}(\ii\eta)=-1$ since $\nu$ is a probability measure. Conversely, if an analytic function $m\,:\,\C^+\rightarrow \C^+$ satisfies $\lim_{\eta\nearrow 0}\ii \eta \, m(\ii\eta)=-1$, then it is the Stieltjes transform of a probability~measure. 

Choosing $\nu$ to be the standard semicircle law with density $\frac{1}{2\pi}\sqrt{4-x^2}$ on $[-2,2]$, on easily shows that $m_{\nu}$, for simplicity hereinafter denoted by $\msc$, is explicitly given by
\begin{align}\label{le msc}
\msc(z) = \frac{-z + \sqrt{z^2 -4}}{2}\,, \qquad \qquad(z \in \C^+)\,,
\end{align}
where we choose the branch of the square root so that $\msc(z) \in \C^+$, $z\in\C^+$. It directly follows that
\begin{align}\label{le sce for msc}
1 + z \msc(z) + \msc(z)^2 = 0\,,  \qquad \qquad(z \in \C^+)\,.
\end{align}

\subsection{Main results}
In this section we present our main results. We first generalize the matrix ensembles derived from the Erd\H os--R\'enyi graph model and the diluted Wigner matrices in Section~\ref{subsection motivating examples}.

\begin{assumption} \label{assumption H}
Fix any small $\phi>0$. We assume that $H = (H_{ij})$ is a real symmetric $N \times N$ matrix whose diagonal entries are almost surely zero and whose off-diagonal entries are independent, up to the symmetry constraint $H_{ij} = H_{ji}$, identically distributed random variables. We further assume that $(H_{ij})$ satisfy the moment conditions
\begin{align}\label{le moment condition}
\E H_{ij} = 0\,, \qquad \E (H_{ij})^2 = \frac{1-\delta_{ij}}{N}\,, \qquad \E |H_{ij}|^k \leq \frac{(Ck)^{ck}}{N q^{k-2}}\,, \quad\qquad (k\ge 3)\,,
\end{align}
with sparsity parameter $q$ satisfying
\begin{align}\label{le phi}
 N^{\phi}\le q\le N^{1/2}\,.
\end{align}
\end{assumption}

We assume that the diagonal entries satisfy $H_{ii}=0$ a.s., yet this condition can easily be dropped. For the choice $\phi=1/2$ we recover the {\it real symmetric Wigner ensemble} (with vanishing diagonal). For the rescaled adjacency matrix of the Erd\H{o}s--R\'enyi graph, the sparsity parameter~$q$, the edge probability~$p$ and the expected degree of a vertex $d$ are linked by $q^2=pN=d$.

 We denote by $\kappa^{(k)}$ the $k$-th {\it cumulant} of the i.i.d.\ random variables $(H_{ij}:i<j)$. Under Assumption~\ref{assumption H} we have $\kappa^{(1)}=0$, $\kappa^{(2)}=1/N$, and
\begin{align}\label{le cumulant bound}
 |\kappa^{(k)}| \leq \frac{(2Ck)^{2(c+1)k}}{N q^{k-2}}\,, \quad\qquad (k\ge 3)\,.
\end{align}
We further introduce the {\it normalized cumulants}, $s^{(k)}$, by setting
\begin{align}\label{normalized cumulants}
 s^{(1)}\deq 0\,,\qquad\quad s^{(2)}\deq 1\,,\qquad\quad s^{(k)}\deq Nq^{k-2}\kappa^{(k)}\,,\qquad(k\ge 3)\,.
\end{align}
In case $H$ is given by the centered adjacency matrix~$\wt A$ introduced in~Subsection~\ref{le subsection adjacency matrix ER graph}, we have that  $\nm^{(k)}=1+O(d/N)$, $k\ge 3$,  as follows from~\eqref{assumption ER2}.

We start with the local law for the Green function of this matrix ensemble.
\subsubsection{Local law up to the edges for sparse random matrices}
Given a real symmetric matrix $H$ we define its Green function, $G^H$, and the normalized trace of its Green function, $m^H$, by setting
\begin{align}\label{se green functions}
 G^H(z)\deq \frac{1}{H-z\id}\,,\qquad\quad m^H(z)\deq\frac{1}{N}\mathrm{Tr} \,G^H(z)\,,\qquad\qquad (z\in\C^+)\,.
\end{align}
The matrix entries of $G^H(z)$ are denoted by $G^{H}_{ij}(z)$. In the following we often drop the explicit $z$-dependence from the notation for $G^H(z)$ and $m^H(z)$. 

Denoting by $\lambda_1 \geq \lambda_2 \geq \dots \geq \lambda_N$ the ordered eigenvalues of $H$, we note that $m^H$ is the Stieltjes transform of the empirical eigenvalue distributions, $\mu^H$, of $H$ given by
\begin{align}
 \mu^H\deq\frac{1}{N}\sum_{i=1}^N \delta_{\lambda_i}\,.
\end{align}

We further introduce the following domain of the upper-half plane
\begin{align}\label{second domain}
\caE\deq \{ z = E+ \ii \eta \in \C^+ : |E| < 3, \, 0< \eta \le 3 \}\,.
\end{align}

Our first main result is the local law for $m^H$ up to the spectral edges.
\begin{theorem} \label{theorem:local}
Let $H$ satisfy Assumption \ref{assumption H} with $\phi>0$. Then, there exists an algebraic function $\wt m : \C^+ \to \C^+$ and $2<L<3$ such that the following hold:
\begin{enumerate}
\item The function $\wt m$ is the Stieltjes transform of a deterministic symmetric probability measure $\wt \rho$, \ie $\wt m(z)=m_{\wt\rho}(z)$. Moreover, $\supp \wt\rho=[-L,L]$ and $\wt\rho$ is absolutely continuous with respect to Lebesgue measure with a strictly positive density on $(-L,L)$.
\item The function $\wt m\equiv \wt m(z)$ is a solution to the polynomial equation
\begin{align}\begin{split} \label{eq:poly la}
P_{ z}(\wt m)& \deq 1 + z \wt m + \wt m^2 + \frac{{ \nmf }}{q^2}\wt m^4=0\,,\qquad\qquad (z\in \C^+)\,.
\end{split}\end{align}

\item The normalized trace $m^H$ of the Green function of $H$
satisfies
\begin{align}\label{le local law}
|m^H (z) - \wt m (z)| \prec \frac{1}{q^2}+\frac{1}{N\eta}\,,
\end{align}
uniformly on the domain $\caE$, $z=E+\ii\eta$.
\end{enumerate}
\end{theorem}
Some properties of $\wt\rho$ and its Stieltjes transform $\wt m$ are collected in Lemma~\ref{lem:w} below.

The local law~\eqref{le local law} implies estimates on the {\it local density of states} of $H$. For $E_1< E_2$ define 
\begin{align*}
 \frak{n}(E_1,E_2)\deq\frac{1}{N}|\{ i\,:\, E_1<\lambda_i\le E_2\}|\,,\qquad\quad n_{\wt\rho}(E_1,E_2)\deq \int_{E_1}^{E_2}\wt\rho(x)\,\dd x\,.
\end{align*}
\begin{corollary}\label{corollary density of states}
 Suppose that $H$ satisfies Assumption \ref{assumption H} with $\phi>0$. Let $E_1,E_2\in\R$, $E_1<E_2$. Then,
 \begin{align}\label{le estimate on density of states}
  |\frak{n}(E_1,E_2)-n_{\wt\rho}(E_1,E_2)|\prec \frac{E_2-E_1 }{q^2}+ \frac{1}{N}\,.
 \end{align}

\end{corollary}

The proof of Corollary~\ref{corollary density of states} from Theorem~\ref{theorem:local} is a standard application of the Helffer-Sj\"ostrand calculus; see e.g., Section 7.1 of~\cite{EKYY13} for a similar argument.

An interesting effect of the sparsity of the entries of $H$ is that its eigenvalues follow, for large~$N$, the deterministic law $\wt\rho$ that depends on the sparsity parameter $q$. While this law approaches the standard semicircle law $\rho_{sc}$ in the limit $N\rightarrow\infty$, its deterministic refinement to the standard semicircular law for finite $N$ accounts for the non-optimality at the edge of results obtained in~\cite{EKYY1}, \ie when~\eqref{le local law} is compared with~\eqref{le EYYY1 1} below.

\begin{proposition}[Local semicircle law, Theorem 2.8 of \cite{EKYY1}] \label{local semiclrcle law}
Suppose that $H$ satisfies Assumption~\ref{assumption H} with $\phi>0$. Then, the following estimates hold uniformly for $z \in \caE$:
\begin{align}\label{le EYYY1 1}
|m^H(z) - \msc(z)| \prec \min \left\{ \frac{1}{q^2 \sqrt{\kappa+\eta}}, \frac{1}{q} \right\} + \frac{1}{N\eta}\,,
\end{align}
where $\msc$ denote the Stieltjes transform of the standard semicircle law, and
\begin{align}\label{le EYYY1 2}
\max_{1 \leq i, j \leq N} |G^H_{ij}(z) - \delta_{ij} \msc(z)| \prec \frac{1}{q} + \sqrt{\frac{\im \msc(z)}{N\eta}} + \frac{1}{N\eta}\,,
\end{align}
where $\kappa\equiv \kappa(z)\deq |E-2|$, $z=E+\ii\eta$.
\end{proposition}
We remark that the estimate~\eqref{le EYYY1 1} is essentially optimal as long as the spectral parameter $z$ stays away from the spectral edges, \eg for energies in the bulk $E\in[-2+\delta,2-\delta]$, $\delta>0$. For the individual Green function entries, $G_{ij}$, we believe that the estimate~\eqref{le EYYY1 2} is already essentially optimal ($\msc$ therein may be replaced by $\wt m$ without changing the error bound). A consequence of Proposition~\ref{local semiclrcle law} is that all eigenvectors of $H$ are completely delocalized.
\begin{proposition}[Theorem~2.16 and Remark~2.18 in~\cite{EKYY1}]\label{cor: delocalization}
 Suppose that $H$ satisfies Assumption~\ref{assumption H} with $\phi>0$. Denote by $(\bsu_i^H)$ the $\ell^2$-normalized eigenvectors of $H$.~Then,
 \begin{align}
  \max_{1\le i\le N} { \|\bsu_i^H\|}_\infty\prec\frac{1}{\sqrt{N}}\,.
 \end{align}

\end{proposition}
Using~\eqref{le EYYY1 1} as {\it a priori} input it was proved in~\cite{HLY} that the local eigenvalue statistics in the bulk agree with the local statistics of the GOE, for $\phi>0$; see also~\cite{EKYY2} for $\phi>1/3$.
When combined with a high moment estimates of~$H$ (see Lemma~4.3 in~\cite{EKYY1}), the estimate in~\eqref{le EYYY1 1} implies the following bound on the operator norm~of~$H$.

\begin{proposition}[Lemma 4.4 of \cite{EKYY1}] \label{a priori norm bound}
Suppose that $H$ satisfies Assumption \ref{assumption H} with $\phi>0$. Then,
\begin{align}
| \|H\| -2| \prec \frac{1}{q^2}+\frac{1}{N^{2/3}}\,.	
\end{align}
\end{proposition}

The following estimates of the operator norm of $H$ sharpens the estimates of Proposition~\ref{a priori norm bound} by including the deterministic refinement to the semicircle law as expressed by Theorem~\ref{theorem:local}.

\begin{theorem} \label{prop:norm bound}
Suppose that $H$ satisfies Assumption \ref{assumption H} with $\phi>0$. Then,
\begin{align}\label{le estimate on operator norm}
\left|\|H\|-L\right|\prec \frac{1}{q^4} +\frac{1}{N^{2/3}}\,,
\end{align}
where $\pm L$ are the endpoints of the support of the measure $\wt \rho$ given by
\begin{align}\label{le L}
L=2+\frac{\nmf}{q^2}+O(q^{-4})\,.
\end{align}

\end{theorem}

Here and above, we restricted the choice of the sparsity parameter $q$ to the range $N^\phi\le q\le N^{1/2}$ for arbitrary small $\phi>0$. Yet, pushing our estimates and formalism we expect also to cover the range $(\log N)^{A_0\log\log N}\le q\le N^{1/2}$, $A_0\ge 30$, considered in~\cite{EKYY1}. In fact, Khorunzhiy showed for diluted Wigner matrices (\cf Subsection~\ref{sec:diluted w}) that $\|H\|$ converges almost surely to $2$ for $q\gg(\log N)^{1/2}$, while~$\|H\|$ diverges for $q\ll(\log N)^{1/2}$; see Theorem~2.1 and Theorem~2.2 of~\cite{Kho1} for precise statements.

 As noted in Theorem~\ref{prop:norm bound}, the local law allows strong statements on the locations of the extremal eigenvalues of $H$. We next discuss implications for the fluctuations of the rescaled extremal eigenvalues.

\subsubsection{Tracy--Widom limit of the extremal eigenvalues}

Let $W$ be a real symmetric Wigner matrix and denote by $\lambda_1^{W}$ its largest eigenvalue. The {\it edge universality} for Wigner matrices asserts that
\begin{align}\label{le TW}
\lim_{N \to \infty} \p \Big( N^{2/3} (\lambda_1^{W} -2) \leq s \Big) = F_1 (s)\,,
\end{align}
where $F_1$ is the Tracy--Widom distribution function~\cite{TW1,TW2} for the GOE. Statement~\eqref{le TW} holds true for the smallest eigenvalue $\lambda_N^{W}$ as well. We henceforth focus on the largest eigenvalues, the smallest eigenvalues can be dealt with in exactly the same way.

The universality of the Tracy--Widom laws for Wigner matrices was first proved in~\cite{SiSo1,So1} for real symmetric and complex Hermitian ensembles with symmetric distributions. The symmetry assumption on the entries' distribution was partially removed in~\cite{PeSo1,PeSo2}. Edge universality without any symmetry assumption was proved in~\cite{TV2} under the condition that the distribution of the matrix elements has subexponential decay and its first three 
moments match those of the Gaussian distribution, \ie the third moment of the entries vanish. The vanishing third moment condition was removed in~\cite{EYY}. A necessary and sufficient condition on the entries' distribution for the edge universality of Wigner matrices was given in~\cite{LY}.

 Our second main result shows that the fluctuations of the rescaled largest eigenvalue of the sparse matrix ensemble are governed by the Tracy--Widom law, if the sparsity parameter~$q$ satisfies $q\gg N^{1/6}$.

\begin{theorem} \label{thm tw}
Suppose that $H$ satisfies Assumption \ref{assumption H} with $\phi>1/6$. Denote by $\lambda^H_1$ the largest eigenvalue of~$H$. Then,

\begin{align} \label{eq:main}
\lim_{N \to \infty} \p \left( N^{2/3} \big( \lambda^H_1 -L\big)\leq s \right) = F_1 (s)\,,
\end{align}
where $L$ denotes the upper-edge of the deterministic measure $\wt\rho$ given in~\eqref{le L}.
\end{theorem}

The convergence result~\eqref{eq:main} was obtained in Theorem~2.7 of~\cite{EKYY2} under the assumption that the sparsity parameter $q$ satisfies $q\gg N^{1/3}$, \ie $\phi > 1/3$ (and with $2$ replacing $L$). 

In the regime $N^{1/6}\ll q\le N^{1/3}$, the deterministic shift of the upper edge by $L-2=O(q^{-2})$ is essential for~\eqref{eq:main} to hold since then $q^{-2} \ge N^{-2/3}$, the latter being the scale of the Tracy--Widom fluctuations. In other words, to observe the Tracy--Widom fluctuations in the regime $N^{1/6}\ll q\le N^{1/3}$ corrections from the fourth moment of the matrix entries' distribution have to be accounted for. This is in accordance with high order moment computations for diluted Wigner matrices in~\cite{Kho2}.

It is expected that the order of the fluctuations of the largest eigenvalue exceeds $N^{-2/3}$ if $q \ll N^{1/6}$. The heuristic reasoning is that, in this regime, the fluctuations of the eigenvalues in the bulk of the spectrum are much larger than $N^{-2/3}$ and hence affect the fluctuations of the eigenvalues at the edges. Indeed, the linear eigenvalue statistics of sparse random matrices were studied in~\cite{BGM,ShTi}. For sufficiently smooth functions $\varphi$, it was shown there that
\begin{align*}
\frac{q}{\sqrt N} \sum_{i=1}^N \varphi(\lambda_i) - \E \bigg[ \frac{q}{\sqrt N} \sum_{i=1}^N \varphi(\lambda_i) \bigg]
\end{align*}
converges to a centered Gaussian random variable with variance of order one. This suggests that the fluctuations of an individual eigenvalue in the bulk are of order $N^{-1/2} q^{-1}$, which is far greater than the Tracy--Widom scale~$N^{-2/3}$ if $q \ll N^{1/6}$.

\begin{remark} \label{k main}
Theorem \ref{thm tw} can be extended to correlation functions of extreme eigenvalues as follows: For any fixed~$k$, the joint distribution function of the first $k$ rescaled eigenvalues converges to that of the GOE, \ie if we denote by $\lambda_1^{\GOE} \geq \lambda_2^{\GOE} \geq \ldots \geq \lambda_N^{\GOE}$ the eigenvalues of a GOE matrix independent of $H$, then
\begin{align} \label{eq: k main}
&\lim_{N \to \infty} \p \left( N^{2/3} \big( \lambda^H_1 -L\big) \leq s_1 \,,  N^{2/3} \big( \lambda_2^H - L\big) \leq s_2 \,, \ldots, N^{2/3} \big( \lambda_k^H -L\big) \leq s_k \right)\nonumber \\
&\quad= \lim_{N \to \infty} \p \left( N^{2/3} \big( \lambda_1^{\GOE} - 2 \big) \leq s_1 \,, N^{2/3} \big( \lambda_2^{\GOE} - 2 \big) \leq s_2 \,, \ldots, N^{2/3} \big( \lambda_k^{\GOE} - 2 \big) \leq s_k \right)\,.
 \end{align}
\end{remark}

We further mention that all our results also hold for complex Hermitian sparse random matrices with the GUE Tracy--Widom law describing the limiting edge fluctuations.

\subsubsection{Applications to the adjacency matrix of the Erd\H os--R\'enyi graph}
We briefly return to the adjacency matrix~$A$ of the Erd\H os--R\'enyi graph ensemble introduced in Subsection~\ref{le subsection adjacency matrix ER graph}. Since the entries of~$A$ are not centered, the largest eigenvalue $\lambda_1^A$ is an outlier well-separated from the other eigenvalues. Recalling the definition of the matrix~$\wt A$ whose entries are centered, we notice that
\begin{align}\label{le triv shift}
 A=\wt A+f\ket{\boldsymbol{e}}\bra{\boldsymbol{e}}-a \id\,,
\end{align}
with $f\deq q(1-q^2/N)^{-1/2}$, $a\deq f/N$ and $\boldsymbol{e}\deq N^{-1/2} (1,1,\ldots,1)^\mathrm{T}\in\R^N$. (Here, $\ket{\boldsymbol{e}}\bra{\boldsymbol{e}}$ denotes the orthogonal projection onto $\boldsymbol{e}$.)  The expected degree $d$ and the sparsity parameter $q$ are linked by $$d=pN=q^2.$$ Applying a simple rank-one perturbation formula and shifting the spectrum by $a$ we get from Theorem~\ref{theorem:local} the following corollary whose proof we leave aside.

\begin{corollary}\label{cor:local law A}
Fix $\phi>0$. Let $A$ satisfies~\eqref{assumption ER1} and~\eqref{assumption ER2} with expected degree $N^{2\phi}\le d\le N^{1-2\phi}$.  Then
the normalized trace $m^A$ of the Green function of $A$
satisfies
\begin{align}\label{le local law bis}
|m^A (z) - \wt m (z+a)| \prec \frac{1}{d}+\frac{1}{N\eta}\,,
\end{align}
uniformly on the domain $\caE$, where $z=E+\ii\eta$ and $a=\frac{q}{N}{(1-q^2/N)^{-1/2}}$.
\end{corollary}
Let $\lambda^A_1\ge \lambda^A_2\ge \ldots \ge \lambda^A_N$ denote the eigenvalues of $A$. The behavior of the largest eigenvalue $\lambda_1^A$ was fully determined in~\cite{EKYY1}, where it was shown that it has Gaussian fluctuations, \ie
\begin{align}
 \sqrt{\frac{N}{2}}(\lambda_1^A-\E\lambda_1^A)\rightarrow\mathcal{N}(0,1)\,	,
\end{align}
in distribution as $N\to\infty$, with $\E\lambda_1^A=f-a+\frac{1}{f}+O({q^{-3}})$; see Theorem~6.2 in~\cite{EKYY1}.

Combining Theorem~\ref{thm tw} with the reasoning of Section~6 of~\cite{EKYY2}, we have the following corollary on the behavior of the second largest eigenvalue $\lambda_2^A$ of the adjacency matrix $A$.
\begin{corollary}\label{cor: TW for A}
Fix $\phi>1/6$ and $\phi'>0$. Let $A$ satisfies~\eqref{assumption ER1} and~\eqref{assumption ER2} with expected degree $N^{2\phi}\le d\le N^{1-2\phi'}$. Then,
 \begin{align} \label{eq:main A}
\lim_{N \to \infty} \p \left( N^{2/3} \big( \lambda^A_2 - L-a\big)\leq s \right) = F_1 (s)\,,
\end{align}
where $L=2+d^{-1}+O(d^{-2})$ is the upper edge of the measure $\wt \rho$; see~\eqref{le L}.
\end{corollary}
We skip the proof of Corollary~\ref{cor: TW for A} from Theorem~\ref{thm tw}, since it is essentially the same as the proof of Theorem~2.7 in~\cite{EKYY2}, where the result was obtained for $\phi>1/3$, with $L$ replaced by $2$. In analogy with Remark~\ref{k main}, the convergence result in~\eqref{eq:main A} extends in an obvious way to the eigenvalues $\lambda_{k-1}^A,\ldots,\lambda_2^A$, for any fixed $k$. The analogous results apply to the $k$-smallest eigenvalues of~$A$. We leave the details to the interested reader.

\begin{remark} \label{community}
The largest eigenvalues of sparse random matrices, especially the (shifted and normalized) adjacency matrices of the Erd\H os--R\'enyi graphs, can be used to determine the number of clusters in automated community detection algorithms~\cite{BS, Lei} in stochastic block models. Corollary~\ref{cor: TW for A} suggests that the test statistics for such algorithms should reflect the shift of the largest eigenvalues if $N^{-2/3} \ll p \ll N^{-1/3}$, or equivalently, $N^{1/3} \ll d \ll N^{2/3}$. If $p \ll N^{-2/3}$, the test based on the edge universality of random matrices may fail as we have discussed after Theorem~\ref{thm tw}.

In applications, the sparsity should be taken into consideration due to small $N$ even if $p$ is reasonably large. For example, in the Erd\H os--R\'enyi graph with $N \simeq 10^3$, the deterministic shift is noticeable if $p \simeq 0.1$, which is in the vicinity of the parameters used in numerical experiments in~\cite{BS, Lei}.
\end{remark}

\section{Strategy and outline of proofs} \label{sec:outline}
In this section, we outline the strategy of our proofs. We begin with the local law of Theorem~\ref{theorem:local}. 

\subsection{Wigner type matrices}
We start by recalling the approach initiated in~\cite{ESY1,ESY2,ESY3} for Wigner matrices. Using Schur's complement (or the Feshbach formula) and large deviation estimates for quadratic forms by Hanson and Wright~\cite{HW}, one shows that the normalized trace $m^W(z)$ approximately satisfies the equation $1+zm^W(z)+m^W(z)^2\simeq 0$, with high probability, for any~$z$ in some appropriate subdomain of $\mathcal{E}$. Using that $\msc$ satisfies~\eqref{le sce for msc}, a local stability analysis then yields $|m^W(z)-m_{sc}(z)|\prec (N\eta)^{-1/2}$, $z\in\mathcal{E}$. In fact, the same quadratic equation is approximately satisfied by each diagonal element of the resolvent, $G_{ii}^W$, and not only by their average $m^W$. This observation and an extension of the stability analysis to vectors instead of scalars then yields the {\it entry-wise local law}~\cite{EYY2,EYY3,EYY}, $|G_{ij}^W(z)-\delta_{ij} \msc(z)|\prec (N\eta)^{-1/2}$, $z\in\mathcal{E}$. (See Subsection~\ref{Example: Local law for the GOE} for some details of this argument.) Taking the normalized trace of the Green function, one expects further cancellations of fluctuations to improve the bound. Exploring the {\it fluctuation averaging mechanism} for $m^W$ and refining the local stability analysis, one obtains the {\it strong local law} up to the spectral edges~\cite{EYY}, $|m^W(z)-\msc(z)|\prec(N\eta)^{-1}$, $z\in\mathcal{E}$. It was first introduced in~\cite{EYY3} and substantially extended in~\cite{EKY,EKYY13} to generalized Wigner matrices. We refer to~\cite{E,EKYY13} for reviews of this general approach. Parallel results were obtained in~\cite{TV1,TV2}. For more recent developments see~\cite{AEK,BES15b,BKY,CMS,GNT,LS}.
 
The strategy outlined in the preceding paragraph was applied to sparse random matrices in~\cite{EKYY1}. The sparsity of the entries manifests itself in the large deviation estimate for quadratic forms, \eg letting $(H_{ij})$ satisfy~\eqref{le moment condition} and choosing $(B_{ij})$ to be any deterministic $N\times N$ matrix,~Lemma~3.8 of~\cite{EKYY1} assures that
\begin{align}\label{le LDE}
 \Big|\sum_{k,l}H_{ik}B_{kl}H_{li}-\frac{1}{N}\sum_{k=1}^NB_{kk}\Big|\prec\frac{{\max_{k,l}} |B_{kl}|}{q}+\Big(\frac{1}{N^2}\sum_{k,l} |B_{kl}|^2\Big)^{1/2}\,,
\end{align}
for all $i\in\llbracket 1,N\rrbracket$. Using the above ideas the entry-wise local law in~\eqref{le EYYY1 2} was obtained in~\cite{EKYY1}. Exploiting the fluctuation averaging mechanism for the normalized trace of the Green function, an additional power of $q^{-1}$ can be gained, leading to~\eqref{le EYYY1 1} with the deteriorating factor $(\kappa+\eta)^{-1/2}$.

To establish the local law for the normalized trace of the Green function which does not deteriorate at the edges, we propose in this paper a novel recursive moment estimate for the Green function. When applied to the proof of the strong local law for a Wigner matrix, it is estimating $\E|1+zm^W(z)+m^W(z)^2|^D$ by the lower moments $\E|1+zm^W(z)+m^W(z)^2|^{D-l}$, $l\ge 1$. The use of recursive moment estimate has three main advantages over the previous fluctuation averaging arguments: (1) it is more convenient in conjunction with the cumulant expansion in Lemma~\ref{le stein lemma}, (2) it is easier to track the higher order terms involving the fourth and higher moments if needed, and (3) it does not require to fully expand the higher power terms and thus simplifies bookkeeping and combinatorics. The same strategy can also be applied to individual entries of the Green function by establishing a recursive moment estimate for $\E|1+zG^W_{ii}(z)+\msc(z) G^W_{ii}(z)|^D$ and $\E|G_{ij}|^D$ leading to the entry-wise local law.

We illustrate this approach for the simple case of the GOE next.

\subsection{Local law for the GOE}\label{Example: Local law for the GOE}
Choose $W$ to be a GOE matrix.  Since $\msc\equiv \msc(z)$, the Stieltjes transform of the semicircle law, satisfies $1+z\msc+\msc^2=0$, we expect that moments of the polynomial $1+zm(z)+m(z)^2$, with $m\equiv m^W(z)$ the normalized trace of the Green function~$G\equiv G^W(z)$, are small. We introduce the subdomain $\caD$ of $\caE$ by setting
\begin{align} \label{domain}
\caD \deq \{ z = E+ \ii \eta \in \C^+ : |E| < 3, \, N^{-1} < \eta < 3 \}\,.
\end{align}

We are going to derive the following {\it recursive moment estimate} for $m$. For any $D\ge 2$,
\begin{align}
&\E [|1 + zm + m^2|^{2D}]\nonumber \\
&\qquad\leq \E \left[ \frac{\im m}{N\eta} |1 + zm + m^2|^{2D-1} \right] + (4D-2) \E \left[ \frac{\im m}{(N\eta)^2} |z+2m| \cdot |1 + zm + m^2|^{2D-2} \right]\,,\label{example GOE}
\end{align}
for $z\in\C^+$. Fix now $z\in \caD$. Using Young's inequality, the second order Taylor expansion of $m(z)$ around $\msc(z)$ and the {\it a priori} estimate $|m(z)-\msc(z)|\prec 1$, we conclude with Markov's inequality from~\eqref{example GOE} that
\begin{align}\label{le SCE for GOE} 
 \left|\alpha_{\mathrm{sc}}(z) (m(z)-\msc(z))+(m(z)-\msc(z))^2\right|\prec \frac{|m(z)-\msc(z)|}{N\eta}+\frac{\im \msc(z)}{N\eta}+\frac{1}{(N\eta)^2}\,,
\end{align}
where $\alpha_{\mathrm{sc}}(z)\deq z+2\msc(z)$; see Subsection~\ref{subsection proof of prposition prop:local} for a similar computation. An elementary computation reveals that $|\alpha_{\mathrm{sc}}(z)| \sim \im \msc(z)$. Equation~\eqref{le SCE for GOE} is a {\it self-consistent equation} for the quantity $m(z)-\msc(z)$. Its local stability properties up to the edges were examined in the works~\cite{EYY2,EYY3}. From these results and~\eqref{le SCE for GOE} it follows that, for fixed $z\in \caD$, $|m(z)-\msc(z)|\prec 1$, implies $|m(z)-\msc(z)|\prec \frac{1}{N\eta}$. To obtain the local law on $\caD$ one then applies a continuity or bootstrapping argument~\cite{ESY1,EYY2,EYY3} by decreasing the imaginary part of the spectral parameter from $\eta\sim 1$ to $\eta\ge N^{-1}$. Using the monotonicity of the Stieltjes transform, this conclusion is extended to all of $\caE$. This establishes the local law for the~GOE,
\begin{align}\label{le local law for GOE}
 |m(z)-\msc(z)|\prec\frac{1}{N\eta}\,,
\end{align}
uniformly on the domain $\mathcal{E}$. 

Hence, to obtain the strong local law for the GOE, it suffices to establish~\eqref{example GOE} for fixed $z\in\caD$. By the definition of the normalized trace, $m\equiv m(z)$, of the Green function we have
\begin{align} \label{expansion 0}
\E [ |1 + zm + m^2|^{2D} ] = \E \bigg[ \bigg( \frac{1}{N} \sum_{i=1}^N (1 + z G_{ii}) + m^2 \bigg) (1 + zm + m^2)^{D-1} (\ol {1 + zm + m^2})^D \bigg]\,.
\end{align}
We expand the diagonal Green function entry $G_{ii}\equiv G^W_{ii}$ using the following identity:
\begin{align} \label{green identity}
1 + z G_{ii} = \sum_{k=1}^N W_{ik} G_{ki}\,,
\end{align}
which follows directly from the defining relation $(W-z\id)G = \id$. To some extent~\eqref{green identity} replaces the conventional Schur complement formula. We then obtain
\begin{align} \label{expansion wigner}
\E [ |1 + zm + m^2|^{2D} ] = \E \bigg[ \bigg( \frac{1}{N} \sum_{i, k} W_{ik} G_{ki} + m^2 \bigg) (1 + zm + m^2)^{D-1} (\ol {1 + zm + m^2})^D \bigg]\,.
\end{align}
Using that the matrix entries $W_{ik}$ are Gaussians random variables, integration by parts shows that
\begin{align}\label{le SL}
\E_{ik} [W_{ik} F(W_{ik})] = \frac{1+\delta_{ik}}{N} \E_{ik} [ \partial_{ik} F(W_{ik})]\,,
\end{align}
for differentiable functions $F\,:\,\R\to\C$, where $\partial_{ik} \equiv \partial/(\partial W_{ik})$. Here we used that $\E W_{ij}=0$ and $\E W^2_{ij}=(1+\delta_{ij})/N$ for the GOE. Identity~\eqref{le SL} is often called {\it Stein's lemma} in the statistics literature. Combining~\eqref{expansion wigner} and~\eqref{le SL} we obtain
\begin{align} \begin{split} \label{expansion wigner 2}
\E [ |1 + zm + m^2|^{2D} ]&= \frac{1}{N^2} \E \bigg[ \sum_{i, k}(1+\delta_{ik}) \partial_{ik} \bigg( G_{ki} (1 + zm + m^2)^{D-1} (\ol {1 + zm + m^2})^D \bigg) \bigg] \\
&\qquad + \E \big[ m^2 (1 + zm + m^2)^{D-1} (\ol {1 + zm + m^2})^D \big]\,.
\end{split} \end{align}
We next expand and estimate the first term on the right side of \eqref{expansion wigner 2}. It is easy to see that
\begin{align} \begin{split}\label{expansion wigner 3}
&(1+\delta_{ik})\partial_{ik} \left( G_{ki} (1 + zm + m^2)^{D-1} (\ol {1 + zm + m^2})^D \right) \\
&= -G_{ii} G_{kk} (1 + zm + m^2)^{D-1} (\ol {1 + zm + m^2})^D - G_{ki} G_{ki} (1 + zm + m^2)^{D-1} (\ol {1 + zm + m^2})^D \\
&\quad - \frac{2(D-1)}{N} G_{ki} (z+2m) \sum_{j=1}^N G_{jk} G_{ij} (1 + zm + m^2)^{D-2} (\ol {1 + zm + m^2})^D \\
&\quad - \frac{2D}{N} G_{ki} (\ol z + 2\ol m) \sum_{j=1}^N \ol{G_{jk}} \ol{G_{ij}} (1 + zm + m^2)^{D-1} (\ol {1 + zm + m^2})^{D-1}\,.
\end{split} \end{align}
After averaging over the indices $i$ and $k$, the first term on the right side of~\eqref{expansion wigner 3} becomes
\begin{align*}
-\frac{1}{N^2} \E \bigg[ \sum_{i, k} G_{ii} G_{kk} (1 + zm + m^2)^{D-1} (\ol {1 + zm + m^2})^D \bigg] = - \E[ m^2 (1 + zm + m^2)^{D-1} (\ol {1 + zm + m^2})^D]\,,
\end{align*}
which exactly cancels with the second term on the right side of~\eqref{expansion wigner 2}. The second term on the right side of~\eqref{expansion wigner 3} can be estimated as
\begin{align*} \begin{split}
&\bigg|\E \bigg[ \frac{1}{N^2} \sum_{i, k} G_{ki} G_{ki} (1 + zm + m^2)^{D-1} (\ol {1 + zm + m^2})^D \bigg] \bigg|\\
&\qquad\qquad\leq \E \bigg[ \frac{1}{N^2} \sum_{i, k} |G_{ki}|^2 |1 + zm + m^2|^{2D-1} \bigg] = \E \bigg[ \frac{\im m}{N\eta} |1 + zm + m^2|^{2D-1} \bigg]\,,
\end{split} \end{align*}
where we used the identity
\begin{align}\label{ward identity}
\frac{1}{N}\sum_{i=1}^N|G_{ik}(z)|^2=\frac{\im G_{kk}(z)}{N\eta}\,,
\end{align}
which we refer to as the {\it Ward identity} below. It follows from the spectral decomposition of $W$.

For the third term on the right side of~\eqref{expansion wigner 3} we have that
\begin{align} \begin{split} \label{expansion wigner 3 third}
&\bigg|\E \bigg[ \frac{1}{N^3} \sum_{i, j, k} G_{ki} G_{jk} G_{ij} (z+2m) (1 + zm + m^2)^{D-2} (\ol {1 + zm + m^2})^D \bigg]\bigg| \\
&\leq \E \bigg[ \frac{|\Tr G^3|}{N^3} |z+2m| \cdot |1 + zm + m^2|^{2D-2} \bigg] \leq \E \bigg[ \frac{\im m}{(N\eta)^2} |z+2m| \cdot |1 + zm + m^2|^{2D-2} \bigg]\,,
\end{split} \end{align}
where we used that
\begin{align}\label{sum rule 1}
\frac{1}{N^3}|\Tr G^3| \leq \frac{1}{N^3}\sum_{\alpha=1}^N \frac{1}{|\lambda_{\alpha} - z|^3} \leq \frac{1}{N^2\eta^2} \frac{1}{N}\sum_{\alpha=1}^N \frac{\eta}{|\lambda_{\alpha} - z|^2} = \frac{\im m}{N^2\eta^2}.
\end{align}
The fourth term on the right side of~\eqref{expansion wigner 3} can be estimated in a similar manner since
\begin{align}\label{sum rule 2}
\frac{1}{N^3}\bigg| \sum_{i, j, k} G_{ki}  \ol{G_{ij}}\ol{G_{jk}} \bigg| =\frac{1}{N^3} \left| \Tr G \ol G^2 \right| \leq \frac{1}{N^3}\sum_{\alpha=1}^N \frac{1}{|\lambda_{\alpha} - z|^3}\le\frac{\im m}{N^2\eta^2}\,.
\end{align}
Returning to~\eqref{expansion wigner 2}, we hence find, for $z\in\C^+$, that
\begin{align*} \begin{split}
&\E [|1 + zm + m^2|^{2D}] \\
&\qquad\leq \E \left[ \frac{\im m}{N\eta} |1 + zm + m^2|^{2D-1} \right] + (4D-2) \E \left[ \frac{\im m}{(N\eta)^2} |z+2m| \cdot |1 + zm + m^2|^{2D-2} \right]\,,
\end{split} \end{align*}
which is the recursive moment estimate for the GOE stated in~\eqref{example GOE}. 

\begin{remark}\label{remark on entrywise local law of GOE}
 The above presented method can also be used to obtain the entry-wise local law for the Green function of the GOE. Assuming the local law for $m^W$ has been obtained, one may establish a recursive moment estimate for $1+zG^W_{ii}(z)+\msc(z) G^W_{ii}(z)$ to derive
 \begin{align}\label{diagonal entrywise law for GOE}
 \big|G^W_{ii}(z)-\msc(z)\big|\prec\sqrt{\frac{\im \msc(z)}{N \eta}}+\frac{1}{N\eta}\,,
 \end{align}
 uniformly in $i\in\llbracket 1,N\rrbracket$ and $z\in\mathcal{E}$. (One may also consider high moments of $1+z G^W_{ii}(z)+m^W(z)W_{ii}(z)$ to arrive at the same conclusion.) We leave the details to the reader. Yet, for later illustrative purposes in Section~\ref{sec:stein} and Section~\ref{sec:tw}, we sketch the derivation of recursive moment estimate for the off-diagonal Green function entries $G_{ij}\equiv G^W_{ij}(z)$, $i\not=j$. Let $z\in\caD$. Using the relation $(W-z\id)G = \id$, we~get
\begin{align*}
\E\Big[ z|G_{ij}|^{2D}\Big] = z\E\Big[ G_{ij} G_{ij}^{D-1} \ol{G_{ij}^D}\Big] = \sum_{k=1}^N\E\Big[ W_{ik} G_{kj} G_{ij}^{D-1} \ol{G_{ij}^D}\Big]=\sum_{k=1}^N \frac{1+\delta_{ik}}{N} \E \left[ \partial_{ik} \left( G_{kj} G_{ij}^{D-1} \ol{G_{ij}^D} \right) \right]\,,
\end{align*}
where we used Stein's lemma~in~\eqref{le SL} in the last step. Upon computing the derivative we get, for $i\not=j$,
\begin{align} \begin{split} \label{GOE local off-diagonal}
\E \left[ z|G_{ij}|^{2D} \right] &=
-\frac{1}{N} \sum_{k=1}^N \E \left[ G_{kk} G_{ij} G_{ij}^{D-1} \ol{G_{ij}^D} \right] -\frac{1}{N} \sum_{k=1}^N \E \left[ G_{ki} G_{kj} G_{ij}^{D-1} \ol{G_{ij}^D} \right] \\
& \qquad -\frac{D-1}{N} \sum_{k=1}^N \E \left[ G_{kj}^2 G_{ii} G_{ij}^{D-2} \ol{G_{ij}^D} \right] -\frac{D-1}{N} \sum_{k=1}^N \E \left[ G_{kj} G_{ik} G_{ij}^{D-1} \ol{G_{ij}^D} \right] \\
& \qquad -\frac{D}{N} \sum_{k=1}^N \E \left[ |G_{kj}|^2 \ol{G_{ii}} G_{ij}^{D-1} \ol{G_{ij}^{D-1}} \right] -\frac{D}{N} \sum_{k=1}^N \E \left[ G_{kj} \ol{G_{ik}} G_{ij}^{D-1} \ol{G_{ij}^D} \right] \,.
\end{split} \end{align}
The first term in the right side of~\eqref{GOE local off-diagonal} equals $-\E \big[m|G_{ij}|^{2D} \big]$. Using the Ward identity~\eqref{ward identity} we~have
\begin{align*}
 \frac{1}{N}\sum_{k=1}^N |G_{ki}G_{kj}|\le \frac{1}{2N} \sum_{k=1}^N |G_{ki}|^2+\frac{1}{2N} \sum_{k=1}^N |G_{kj}|^2=\frac{\im G_{ii}+\im G_{jj}}{2N\eta}\,.
\end{align*}
Thus using~\eqref{diagonal entrywise law for GOE}, we get the bound
\begin{align*}
 \frac{1}{N}\sum_{k=1}^N |G_{ki}G_{kj}|\prec \frac{\im \msc}{N\eta}+\frac{1}{(N\eta)^2}\,,
\end{align*}
uniformly on $\mathcal{E}$. We can now easily bound the right side of~\eqref{GOE local off-diagonal}, for example, any given $\epsilon > 0$,
\begin{align*}
\bigg| \frac{D}{N} \sum_{k=1}^N \E \left[ |G_{kj}|^2 \ol{G_{ii}} G_{ij}^{D-1} \ol{G_{ij}^{D-1}} \right] \bigg| \leq N^\epsilon \E \left[ \bigg(\frac{\im \msc}{N\eta}+\frac{1}{(N\eta)^2} \bigg) \big|G_{ij}\big|^{2D-2} \right]\,,
\end{align*}
for $N$ sufficiently large. Thus we get from~\eqref{GOE local off-diagonal} that, for $i\not=j$,
 \begin{align}\begin{split}\label{GOE local off-diagonal 2}
  |z+\msc|\,\E \big[|G_{ij}|^{2D}\big]&\le   \E \Big[ |m-\msc| \big|G_{ij}\big|^{2D} \Big] + N^{\epsilon}  \E \bigg[\bigg(\frac{\im \msc}{N\eta}+\frac{1}{(N\eta)^2} \bigg)|G_{ij}|^{2D-2} \bigg]\\ &\qquad\qquad+  N^{\epsilon}  \E \bigg[  \bigg(\frac{\im \msc}{N\eta}+\frac{1}{(N\eta)^{2}} \bigg)|G_{ij}|^{2D-1} \bigg]\,,
 \end{split}\end{align}
uniformly on $\mathcal{E}$, for $N$ sufficiently large. Since $|z+\msc| > c$ on $\mathcal{E}$, for some $N$-independent constant $c>0$, we find from~\eqref{GOE local off-diagonal 2} and~\eqref{le local law for GOE} by Young's and Markov's inequality that
\begin{align}
 \big|G^W_{ij}(z)\big| \prec  \sqrt{\frac{\im \msc(z)}{N \eta}}+\frac{1}{N\eta}\,,\qquad\qquad(i\not=j)\,,
 \end{align}
for fixed $z\in\mathcal{D}$. Using continuity and monotonicity of $G_{ij}(z)$ the bound can be made uniform on the domain $\mathcal{E}$. Together with~\eqref{diagonal entrywise law for GOE}, this shows the entry-wise local law for the GOE.

\end{remark}

\subsection{Local law for sparse matrices}
When applying the strategy of Subsection~\ref{Example: Local law for the GOE} to sparse matrices we face two difficulties. First, since the matrix entries are not Gaussian random variables, the simple integration by parts formula~\eqref{le SL} needs to replaced by a full-fletched cumulant expansion. Second, since the higher order cumulants are not small (in the sense that the $(\ell+2)$-nd cumulant is only $O(N^{-1} q^{-\ell})$) we need to retain higher orders in the cumulant expansion. The following result generalizes~\eqref{le SL}.
  
\begin{lemma}[Cumulant expansion, generalized Stein lemma]\label{le stein lemma}
 Fix $\ell\in \N$ and let $F\in C^{\ell+1}(\R;\C^+)$. Let $Y$ be a  centered random variable with finite moments to order $\ell+2$. Then, 
 \begin{align}\label{le stein}
  \E[Y F(Y)] = \sum_{r=1}^\ell \frac{\kappa^{(r+1)}(Y)}{r!} \E \big[ F^{(r)}(Y) \big]+\E\big[\Omega_\ell(YF(Y))\big]\,,
 \end{align}
where $\E$ denotes the expectation with respect to $Y$, $\kappa^{(r+1)}(Y)$ denotes the $(r+1)$-th cumulant of $Y$ and~$F^{(r)}$ denotes the $r$-th derivative of the function $F$. The error term~$\Omega_\ell(YF(Y))$ in~\eqref{le stein} satisfies
\begin{align}
 \big|\E\big[\Omega_\ell(YF(Y))\big]\big|\le C_\ell \E[ |Y|^{\ell+2}]\sup_{|t|\le Q}|F^{(\ell+1)}(t)|+ C_\ell \E [|Y|^{\ell+2} \lone(|Y|>Q)]\sup_{t\in \R} |F^{(\ell+1)}(t)| \,,
\end{align}
where $Q\ge 0$ is an arbitrary fixed cutoff and $C_\ell$ satisfies $C_\ell\le \frac{(C\ell)^\ell}{\ell!}$ for some numerical constant~$C$.
\end{lemma}
For proofs we refer to Proposition~3.1 in~\cite{LP} and Section~II of~\cite{KKP}. In case $Y$ is a standard Gaussian we recover~\eqref{le stein} and thus we sometimes refer to Lemma~\ref{le stein lemma} as generalized Stein lemma.

 Let $H$ be a sparse matrix satisfying Assumption~\ref{assumption H} with $\phi>0$. Recall the polynomial $P\equiv P_z$ and the function~$\wt m\equiv \wt m(z)$ of Theorem~\ref{theorem:local} that satisfy$P(\wt m)=0$. Let $m\equiv m^{H}(z)$ be given by~\eqref{se green functions}. Following the ideas of Subsection~\ref{Example: Local law for the GOE}, we derive in Section~\ref{sec:stein} a recursive estimate for $\E |P(m)|^{2D}$, for large~$D$ with $z\in\mathcal{D}$; see Lemma~\ref{lem:claim} for the precise statement and Section~\ref{sec:stein} for its proof. We start~with
  \begin{align}\label{the source}
 \E [ |P|^{2D} ] = \E \Big[\Big(1+zm+m^2+\frac{\nm^{(4)}}{q_t^{2}}m^4\Big )P^{{D-1}}\ol{P^D}\Big ]\,,
 \end{align}
 for $D\geq2 $, and expand $zm$ using the identity 
\begin{align}\label{expansion indentity}
zG_{ij} = \sum_{k=1}^N H_{ik} G_{kj}-\delta_{ij}\,,
\end{align}
which follows from the definition of the Green function. We then obtain the identity
\begin{align}\label{expansion}
 \E\big[(1+z m)P^{{D-1}}\ol{P^D}\big]= \E \bigg[ \bigg( \frac{1}{N} \sum_{i\not= k} H_{ik} G_{ki} \bigg) P^{D-1} \ol{P^D} \bigg]\,.
\end{align}
Using the generalized Stein lemma, Lemma~\ref{le stein lemma}, we get
 \begin{align}\begin{split}\label{ziwui eins}
  \E\big[(1+ zm)P^{{D-1}}\ol{P^D}\big]&= \frac{1}{N} \sum_{r=1}^\ell \frac{\kappa^{(r+1)}}{r!} \E \bigg[ \sum_{i \neq k} \partial_{ik}^r \Big( G_{ki} P^{D-1} \ol{P^D} \Big) \bigg]+\E\Omega_{\ell}\Big((1+zm) P^{D-1} \ol{P^D} \Big)\,,
 \end{split}\end{align}
where $\partial_{ik} = \partial/(\partial H_{ik})$ and $\kappa^{(k)}$ are the cumulants of $H_{ij}$, $i\not=j$. The detailed form of the error $\E\Omega_\ell(\cdot)$ is discussed in Subsection~\ref{subsection truncation of the cumulant expansion}. Anticipating the outcome, we mention that we can truncate the expansion at order $\ell\ge 8D$ so that the error term becomes sufficiently small for our purposes. 
 
 From the discussion in Subsection~\ref{Example: Local law for the GOE}, we see that the leading term on the right side of~\eqref{ziwui eins} is $\E [m^2 P^{D-1}\ol{P^D}]$ coming from the $r=1$ term. For the other terms with $2\le r\le \ell$, we need to separate relevant from negligible contributions (see beginning of~Section~\ref{sec:stein} for quantitative statement of negligible contributions). Some of the negligible contributions can be identified by power counting, while others require further expansions using cumulant series and ideas inspired by the GOE computation in Remark~\ref{remark on entrywise local law of GOE} above. In Lemma~\ref{summary expansions}, we will show that the remaining relevant terms stem from the term $r=3$ and are, after further expansions, eventually identified to be $s^{(4)}q^{-2}\E[m^4 P^{D-1}\ol{P^D}]$. As {\it a priori} estimates for this analysis we rely on Proposition~\ref{cor: delocalization}, stating that the eigenvectors of $H$ are completely delocalized, as well as on the rough bounds $|G_{ij}(z)|\prec 1$, $z\in\mathcal{D}$. Returning to~\eqref{ziwui eins}, we then observe that the relevant terms in~\eqref{ziwui eins} cancels with the third and fourth term on the right side of~\eqref{the source}. This yields the recursive moment estimate for $P(m)$, respectively $m$.

As we will see in Section~\ref{sec:rho}, the inclusion of the fourth moment $\nmf/q^2$ in $P(m)$ enables us to compute the deterministic shift of edge which is of order $q^{-2}$. While it is possible to include a higher order correction term involving the sixth moment, $\nm^{(6)}/q^4$, this does not improve the local law in our proof since the largest among the negligible contributions originates from the $r=3$ term in~\eqref{ziwui eins}. (More precisely, it is $I_{3, 2}$ of~\eqref{e1}.)

In Section~\ref{sec:proof main}, we then prove Theorem~\ref{theorem:local} and Theorem~\ref{prop:norm bound} using the recursive moment estimate for~$m$ and a local stability analysis. The local stability analysis relies on some properties of the Stieltjes transform~$\wt m$ of the deterministic distribution~$\wt\rho$ obtained in Section~\ref{sec:rho}.

\subsection{Tracy--Widom limit and Green function comparison} \label{Tracy-Widom limit and Green function comparison}

To establish the edge universality (for $\phi>1/6$), we first show in Subsection~\ref{preliminaries of edge universality} that the distribution of the largest eigenvalue of $H$ may be obtained as the expectation (of smooth functions) of the imaginary part of $m(z)$, for appropriately chosen spectral parameters $z$. Such a relation was the basic structure for proving the edge universality in~\cite{EYY,EKYY2}, and the main ingredients in the argument are the local law, the square-root decay at the edge of the limiting density, and an upper bound on the largest eigenvalue, which are Theorem~\ref{theorem:local}, Lemma~\ref{lem:w}, and Theorem~\ref{prop:norm bound} for the case at hand.
For the sake of self-containment, we redo some parts of these estimates in Subsection~\ref{preliminaries of edge universality}.

In Subsection~\ref{subsection green function comparison}, we then use the {\it Green function comparison} method~\cite{EYY2,EYY} to compare the edge statistics of $H$ with the edge statistics of a GOE matrix. Together with the argument of Subsection~\ref{preliminaries of edge universality}, this will yield the Tracy--Widom limit of the largest eigenvalue. However, the conventional discrete Lindeberg type replacement approach to the Green function comparison does not work due to the slow decaying moments of the sparse matrix. We therefore use a continuous flow that interpolates between the sparse matrix ensemble and the GOE. Such an approach has shown to be effective in proving edge universality for deformed Wigner matrices~\cite{LS14} and for sample covariance matrices~\cite{LS14b}.

More concretely, we consider the {\it Dyson matrix flow} with initial condition $H_0$ defined by
\begin{align} \label{eq:A(t)}
H_t \deq \e{-t/2} H_0 + \sqrt{1-\e{-t}} W^{\GOE}\,, \qquad\qquad( t\ge 0)\,,
\end{align}
where $W^{\GOE}$ is a GOE matrix independent of $H_0$. In fact, since we will choose $H_0$ to be a sparse matrix~$H$ with vanishing diagonal entries, we assume with some abuse of terminology that $W^{\GOE}=(W^{\GOE})^*$ has vanishing diagonal, \ie we assume that $W^{\GOE}_{ii}=0$ and that $(W^{\GOE}_{ij},i<j)$ are independent centered Gaussian random variables of variance $1/N$. It was shown in Lemma~3.5 of~\cite{LY} that the local edge statistics of $W^{\GOE}$ is described by the GOE Tracy--Widom statistics.

Let $\kappa_t^{(k)}$ be the $k$-th cumulant of $(H_t)_{ij}$, $i\not=j$. Then, by the linearity of the cumulants under the addition of independent random variables, we have $\kappa_t^{(1)} =0$, $\kappa_t^{(2)}=1/N$ and $\kappa_t^{(k)}=\e{-kt/2}\kappa^{(k)}$, $k\ge 3$. In particular, we have the bound
\begin{align}\label{eq:cumulant}
 |\kappa_t^{(k)}|\le \e{-t}\frac{(Ck)^{ck}}{N q_t^{k-2}}\,,\qquad\qquad (k\ge 3)\,,
\end{align}
where we introduced the {\it time-dependent sparsity parameter}
\begin{align}
 q_t \deq q\,\e{t/2}\,.
\end{align}

Choosing $t=6\log N$, a straightforward perturbation argument shows that the local statistics, at the edges and in the bulk, of $H_t$ and $W^{\GOE}$ agree up to negligible error. It thus suffices to consider $t\in[0,6\log N]$. 

We first establish the local law for the normalized trace of the Green function of~$H_t$.~Let
\begin{align}\label{le timedependent green functions}
 G_t(z)\equiv G^{H_t}(z)=\frac{1}{H_t-z\id}\,,\qquad\quad m_t(z)\equiv m^{H_t}(z)=\frac{1}{N}\sum_{i=1}^N (G_t)_{ii}(z)\,,\qquad\qquad (z\in\C^+)\,.
\end{align}
\begin{proposition} \label{prop:local}
Let $H_0$ satisfy Assumption~\ref{assumption H} with $\phi>0$. Then, for any $t\ge0$,  there exists a algebraic function $\wt m_t : \C^+ \to \C^+$ and $2\le L_t<3$ such that the following holds:
\begin{enumerate}
\item $\wt m_t$ is the Stieltjes transform of a deterministic symmetric probability measure $\wt \rho_t$, \ie $\wt m_t(z)=m_{\wt\rho_t}(z)$. Moreover, $\supp \wt\rho_t=[-L_t,L_t]$ and $\wt\rho_t$ is absolutely continuous with respect to Lebesgue measure with a strictly positive density on $(-L_t,L_t)$.
\item $\wt m_t \equiv \wt m_t (z)$ is a solution to the polynomial equation
\begin{align}\begin{split} \label{eq:poly}
P_{t, z}(\wt m_t)& \deq 1 + z \wt m_t + \wt m_t^2 + \e{-t}  q_t^{-2} \nmf\wt m_t^4 \\
&=1 + z \wt m_t + \wt m_t^2 + \e{-2t} q^{-2}\nmf \wt m_t^4 = 0\,.
\end{split}\end{align}
\item
The normalized trace of the Green function
satisfies
\begin{align}
|m_t (z) - \wt m_t (z)| \prec \frac{1}{q_t^2}+\frac{1}{N\eta}\,,
\end{align}
uniformly on the domain $\caE$ and uniformly in $t\in[0,6\log N]$.
\end{enumerate}
\end{proposition}
Note that Theorem~\ref{theorem:local} is a special case of Proposition~\ref{prop:local}. Given Proposition~\ref{prop:local}, Corollary~\ref{corollary density of states} extends in the obvious way from $H$ to $H_t$. Proposition~\ref{prop:local} is proved in Subsection~\ref{subsection proof of prposition prop:local}. 

The endpoints $\pm L_t$ of the support of $\wt \rho_t$ are given by  $L_t = 2 + \e{-t} \nmf q_t^{-2} + O(\e{-2t}q_t^{-4})$ and satisfy
\begin{align}\label{le L dot}
\dot L_t = -2 \e{-t} \nmf q_t^{-2} + O(\e{-2t}q_t^{-4})\,,
\end{align}
where $\dot L_t$ denotes the derivative with respect to $t$ of $L_t$; \cf Remark~\ref{rem:L_+ estimate} below.

Choose now $q\ge N^{\phi}$ with $\phi>1/6$. In our proof of the Green function comparison theorem, Proposition~\ref{prop:green}, we estimate the rate of change of~$m_t$ along the Dyson matrix flow over the time interval $[0,6\log N]$, where it undergoes a change of $o(1)$. The continuous changes in $m_t$ can be compensated by letting evolve the spectral parameter $z\equiv z(t)$ according to~\eqref{le L dot}. This type of cancellation argument appeared first in~\cite{LS14} in the context of deformed Wigner matrices. However, one cannot prove the Green function comparison theorem for sparse random matrices by directly applying the cancellation argument since the error bound for the entry-wise local law in Proposition~\ref{local semiclrcle law} is not sufficiently small. Thus the proof of the Green function comparison theorem requires some non-trivial estimates on functions of Green functions as is explained in Subsection~\ref{subsection green function comparison}.

\section{The measure $\wt \rho$ and its Stieltjes transform} \label{sec:rho}
In this section, we prove important properties of $\wt m_t\equiv \wt m_t(z) $ in Proposition \ref{prop:local}. Recall that $\wt m_t$ is a solution to the polynomial equation $P_{t,z}(\wt m_t) = 0$ in \eqref{eq:poly} and that $q_t=\e{t/2}q$. 
\begin{lemma} \label{lem:w}
For any fixed $z = E + \ii \eta \in \caE$ and any $t\ge0$, the polynomial equation $P_{t,z}(w_t) = 0$ has a unique solution $w_t \equiv w_t(z)$ satisfying $\im w_t > 0$ and $|w_t| \leq 5$. Moreover, $w_t$ has the following properties:
\begin{enumerate}
\item There exists a probability measure $\wt \rho_t$ such that the analytic continuation of $w_t$ coincides with the Stieltjes transform of $\wt \rho_t$.
\item The probability measure $\wt \rho_t$ is supported on $[-L_t, L_t]$, for some $L_t \ge 2$, has a strictly positive density inside its support and vanishes as a square-root at the edges, \ie letting
\begin{align}\label{le kappa}
\kappa_t\equiv\kappa_t(E) \deq \min \{ |E+L_t|, |E-L_t| \}\,,
\end{align}
we have
\begin{align}
\wt \rho_t(E) \sim \kappa_t^{1/2}(E)\,,\qquad\qquad (E\in(-L_t,L_t))\,.
\end{align}
Moreover, $L_t = 2 + \e{-t} q_t^{-2}\nmf + O(\e{-2t}q_t^{-4})$.
\item The solution $w_t$ satisfies that
\begin{align} \begin{split}
\im w_t(E+\ii\eta)& \sim \sqrt{\kappa_t + \eta} \qquad \text{ if }\; E \in [-L_t, L_t]\,, \\
\im w_t(E+\ii\eta) &\sim \frac{\eta}{\sqrt{\kappa_t + \eta}} \qquad\! \text{ if }\; E \notin [-L_t, L_t]\,.
\end{split} \end{align}
\end{enumerate}
\end{lemma}

\begin{proof}
For simplicity, we abbreviate $P \equiv P_{t, z}$. Let
\begin{align}
Q(w)\equiv Q_{t,z}(w_t) \deq -\frac{1}{w_t} - w_t - \e{-t} q_t^{-2}\nmf w_t^3.
\end{align}
By definition, $P(w) = 0$ if and only if $z = Q(w)$. It is easily checked that the derivative
\begin{align}
Q'(w) = \frac{1}{w^2} - 1 - 3 \e{-t} q_t^{-2}\nmf w^2
\end{align}
is monotone increasing on $(-\infty, 0)$. Furthermore, we have
$$
Q'(-1) =-3\e{-t}q_t^{-2}\nmf<0\,, \quad Q'(-1 + 2 q_t^{-2}\nmf) = (4-3\e{-t}) q_t^{-2}\nmf + O(q_t^{-4}) > 0\,.
$$
Hence, $Q'(w) = 0$ has a unique solution on $(-1, -1+ 2 q_t^{-2}\nmf)$, which we will denote by $\tau_t$, and $Q(w) \equiv Q_t(w_t)$ attains its minimum on $(-\infty, 0)$ at $w_t = \tau_t$. We let $L_t\deq Q_t(\tau_t)$, or equivalently, $w_t=\tau_t$ if $z=L_t$. We remark that there is no $w_t \in (-\infty, 0)$ satisfying $Q_t(w_t) < L_t$. 
A direct calculation shows that
\begin{align} \label{eq:L_+ estimate}
\tau_t = -1 + \frac{3 \e{-t}}{2} q_t^{-2}\nmf + O(\e{-2t}q_t^{-4})\,, \qquad L_t = 2 + \e{-t} q_t^{-2}\nmf + O(\e{-2t}q_t^{-4})\,.
\end{align}
For simplicity we let $L \equiv L_t$ and $\tau\equiv\tau_t$. Choosing now $w = Q^{-1}(z)$ we have the expansion
\begin{align} \begin{split}
z &= Q(\tau) + Q'(\tau)(w-\tau) + \frac{Q''(\tau)}{2} (w-\tau)^2 + O(|w-\tau|^3) \\
&= L + \frac{Q''(\tau)}{2} (w-\tau)^2 + O(|w-\tau|^3)\,,
\end{split} \end{align}
in a $q_t^{-1/2}$-neighborhood of $\tau_t$. We hence find that
\begin{align} \label{L+ square root}
w = \tau + \left( \frac{2}{Q''(\tau)} \right)^{1/2} \sqrt{z-L} + O(|z-L|)\,,
\end{align}
in that neighborhood. In particular, choosing the branch of the square root so that $\sqrt{z-L} \in \C^+$, we find that $\im w > 0$ since $Q''(\tau)>0$.

We can apply the same argument with a solution of the equation $Q'(w) = 0$ on $(1-2q_t^{-2}\nmf, 1)$, which will lead us to the relation
\begin{align} \label{L- square root}
w = -\tau + \left( \frac{2}{Q''(\tau)} \right)^{1/2} \sqrt{z+L} + O(|z+L|)\,,
\end{align}
in a $q_t^{1/2}$-neighborhood of $-\tau$. We note that there exists another solution with negative imaginary part, which corresponds to the different branch of the square root.

For uniqueness, we consider the disk $B_5 = \{ w \in \C : |w| < 5 \}$. On its boundary $\partial B_5$
\begin{align}
|w^2 + zw + 1| \geq |w|^2 - |z| \cdot |w| - 1 > 1 > \big| q_t^{-2} w^4\nmf \big|\,,
\end{align}
for $z \in \caE$. Hence, by Rouch\'e's theorem, the equation $P(w) = 0$ has the same number of roots as the quadratic equation $w^2 + zw + 1 = 0$ in $B_5$. Since $w^2 + zw + 1 = 0$ has two solutions on $B_5$, we find that $P(w) = 0$ has two solutions on it. For $z = L + \ii q_t^{-1}$, we can easily check that the one solution of $P(w) = 0$ has positive imaginary part (from choosing the branch of the square root as in~\eqref{L+ square root}) and the other solution has negative imaginary part. If both solutions of $P(w) = 0$ are in~$\C^+ \cup \R$ (or in $\C^- \cup \R$) for some $z = \wt z \in \C^+$, then by continuity, there exists $z'$ on the line segment joining $L + \ii q_t^{-1}$ and $\wt z$ such that $P_{t, z'}(w') = 0$ for some $w' \in \R$. By the definition of $P$ this cannot happen, hence one solution of $P(w)=0$ is in~$\C^+$ and the other in~$\C^-$, for any $z \in \caE$. This shows the uniqueness statement of the lemma.

Next, we extend $w \equiv w(z)$ to cover $z \notin \caE$. (With slight abuse of notation, the extension of $w$ will also be denoted by $w \equiv w(z)$.) Repeating the argument of the previous paragraph, we find that $P(w) = 0$ has two solutions for $z \in (-L, L)$. Furthermore, we can also check that exactly one of them is in~$\C^+$ by considering $z = \pm L \mp q_t^{-1}$ and using continuity. Thus, $w(z)$ forms a curve on $\C^+$, joining $-\tau$ and $\tau$, which we will denote by $\Gamma$. We remark that, by the inverse function theorem, $w(z)$ is analytic for $z \in (-L, L)$ since $Q'(w) \neq 0$ for such $z$.

By symmetry, $\Gamma$ intersects the imaginary axis at $w(0)$. On the imaginary axis, we find that
\begin{align}
Q'(w) = \frac{1}{w^2} - 1 + O(\e{-t}q_t^{-2}) < 0 \qquad \text{if } \, |w| < 5\,.
\end{align}
Thus, we get from $Q(w)=z$ that
\begin{align}
\frac{\dd w}{\dd \eta} = \frac{1}{Q'(w)} \frac{\dd z}{\dd \eta} = \frac{\ii}{Q'(w)}\,,
\end{align}
which shows in particular that $w(\ii)$ is pure imaginary and $\im w(\ii) < \im w(0)$. By continuity, this shows that the analytic continuation of $w(z)$ for $z \in \C^+$ is contained in the domain $D_{\Gamma}$ enclosed by $\Gamma$ and the interval $[-L, L]$. We also find that $|w(z)| < 5$, for all $z \in \C^+$.

To prove that $w(z)$ is analytic in $\C^+$, it suffices to show that $Q'(w) \neq 0$ for $w \in D_{\Gamma}$. If $Q'(w) = 0$ for $w \in D_{\Gamma}$, we have
\begin{align}
w^2 Q'(w) = 1 - w^2 - 3 \e{-t} q_t^{-2}\nmf w^4 = 0\,.
\end{align}
On the circle $\{ w \in \C : |w| = 5 \}$,
\begin{align}
|1 - w^2| \geq 24 > \big| 3 \e{-t} q_t^{-2} w^4\nmf \big|\,.
\end{align}
Hence, again by Rouch\'e's theorem, $w^2 Q'(w) = 0$ has two solutions in the disk $\{ w \in \C : |w| < 5 \}$. We already know that those two solutions are $\pm \tau$. Thus, $Q'(w) \neq 0$ for $w \in D_{\Gamma}$ and $w(z)$ is analytic.

Let $\wt \rho$ be the measure obtained by the Stieltjes inversion of $w \equiv w(z)$. To show that $\wt \rho$ is a probability measure, it suffices to show that $\lim_{y \to \infty} \ii y \, w(\ii y) = -1$. Since $w$ is bounded, one can easily check from the definition of $w$ that $|w| \to 0$ as $|z| \to \infty$. Thus,
\begin{align}
0 = \lim_{z \to \ii \infty} P(w(z)) = \lim_{z \to \ii \infty} (1 + zw)\,,
\end{align}
which implies that $\lim_{y \to \infty} \ii y \, w(\ii y) = -1$. This proves the first property of $\wt \rho$. Other properties can be easily proved from the first property and Equations~\eqref{L+ square root} and~\eqref{L- square root}.
\end{proof}

\begin{remark} \label{rem:L_+ estimate}
Recall that $q_t=\e{t/2}q$. As we have seen in \eqref{eq:L_+ estimate},
\begin{align}
L_t = 2 + \e{-t} q_t^{-2}\nmf + O(\e{-2t}q_t^{-4}) = 2 + \e{-2t} q^{-2}\nmf + O(\e{-4t}q^{-4})\,.\nonumber
\end{align}
Moreover, the time derivative of $L_t$ satisfies
\begin{align}
\dot L_t = \frac{\dd}{\dd t} Q(\tau) = \frac{\partial Q}{\partial t}(\tau) + Q'(\tau) \dot \tau = \frac{\partial Q}{\partial t}(\tau) = 2\e{-2t} q^{-2}\nmf \tau^3\,,\nonumber
\end{align}
hence, referring once more to~\eqref{eq:L_+ estimate},
\begin{align}
\dot L_t =  -2 \e{-t} q_t^{-2}\nmf + O(\e{-2t}q_t^{-4}) =-2 \e{-2t} q^{-2}\nmf + O(\e{-4t}q^{-4})\,.\nonumber
\end{align}
\end{remark}

\begin{remark} \label{rem:stability of wt m}
It can be easily seen from the definition of $P_{t, z}$ that $w_t \to \msc$ as $N \to \infty$ or $t \to \infty$. For $z \in \caE$, we can also check the {\it stability condition }
$|z+w_t| > 1/6$ since
\begin{align}
|z+ w_t| = \frac{|1+ \e{-2t} q^{-2}\nmf |w_t|^4|}{|w_t|}
\end{align}
and $|w_t|<5$, as we have seen in the proof of Lemma \ref{lem:w}.
\end{remark}

\section{Proof of Proposition \ref{prop:local} and Theorem~\ref{prop:norm bound}} \label{sec:proof main}

\subsection{Proof of Proposition \ref{prop:local}}\label{subsection proof of prposition prop:local}

In this section, we prove Proposition \ref{prop:local}. The main ingredient of the proof is the recursive moment estimate for $P(m_t)$. Recall the subdomain $\caD$ of $\caE$ defined in~\eqref{domain} and the matrix $H_t$, $t\ge 0$, defined in~\eqref{eq:A(t)}. We have the following result.
\begin{lemma}[Recursive moment estimate]\label{lem:claim}
Fix $\phi>0$ and suppose that $H_0$ satisfies Assumption~\ref{assumption H}. Fix any $t\ge 0$. Recall the definition of the polynomial $P\equiv P_{t,z}$ in \eqref{eq:poly}. Then, for any $D > 10$ and (small) $\epsilon > 0$, the normalized trace of the Green function, $m_t \equiv m_t (z)$, of the matrix $H_t$ satisfies
\begin{align} \label{eq:claim}
\E \left|P(m_t) \right|^{2D}& \leq N^{\epsilon} \, \E \bigg[ \bigg( \frac{\im m_t}{N\eta} + q_t^{-4} \bigg) \big| P(m_t) \big|^{2D-1} \bigg] + N^{-\epsilon/8} q_t^{-1} \, \E \bigg[ |m_t - \wt m_t|^2 \big| P(m_t) \big|^{2D-1} \bigg]\nonumber \\
&\qquad+ N^{\epsilon} q_t^{-1} \, \sum_{s=2}^{2D} \sum_{s'=0}^{s-2} \E \bigg[ \bigg( \frac{\im m_t}{N\eta} \bigg)^{2s-s'-2} \big| P'(m_t) \big|^{s'} \big| P(m_t) \big|^{2D-s} \bigg] + N^\epsilon q_t^{-8D}\\
&\qquad+ N^{\epsilon} \, \sum_{s=2}^{2D} \E \bigg[ \bigg( { \frac{1}{N\eta} } + \frac{1}{q_t}\bigg({\frac{\im m_t}{N\eta}}\bigg)^{1/2} + \frac{1}{q_t^{2}} \bigg) \bigg( \frac{\im m_t}{N\eta} \bigg)^{s-1} \big| P'(m_t) \big|^{s-1} \big| P(m_t) \big|^{2D-s} \bigg]\,,\nonumber
\end{align}
uniformly on the domain $\caD$, for $N$ sufficiently large.
\end{lemma}

The proof of Lemma~\ref{lem:claim} is postponed to Section~\ref{sec:stein}. We are now ready to prove Proposition~\ref{prop:local}.

\begin{proof}[Proof of Proposition \ref{prop:local} and Theorem~\ref{theorem:local}]
Fix $t\in[0,6\log N]$. Let $\wt m_t$ be the solution $w_t$ in Lemma~\ref{lem:w}. The first two parts were already proved in Lemma~\ref{lem:w}, so it suffices to prove the third part of the proposition. For simplicity, we omit the $z$-dependence. Let
\begin{align}
\Lambda_t \deq |m_t - \wt m_t|\,.
\end{align}
We remark that from the local law in Lemma \ref{local semiclrcle law}, we have $\Lambda_t\prec 1$, $z\in\caD$. We also define the following $z$-dependent deterministic parameters
\begin{align}
\alpha_1 \deq \im \wt m_t\, ,\qquad \alpha_2 \deq P'(\wt m_t)\,, \qquad \beta \deq \frac{1}{N\eta} +\frac{1}{ q_t^2}\,,
\end{align}
with $z=E+\ii\eta$. We note that 
\begin{align*}
|\alpha_2| \geq \im P'(\wt m_t) = \im z + 2 \im m_t + 4 \e{-t}\nmf q_t^{-2} \left( 3 (\re \wt m_t)^2 \im \wt m_t - (\im \wt m_t)^3 \right) \geq \im \wt m_t = \alpha_1\,,  
\end{align*}
since $|\wt m_t| \leq 5$ as proved in Lemma \ref{lem:w}. Recall that $\wt m_t(L) = \tau$ in the proof of Lemma \ref{lem:w}. Recalling the definition of $\kappa_t\equiv \kappa_t(E)$ in~\eqref{le kappa} and using~\eqref{L+ square root} we have
\begin{align*}
|\wt m_t - \tau| \sim \sqrt{z - L} \sim \sqrt{\kappa_t + \eta}\,.
\end{align*}
By the definitions of $\tau$ and $L$ in the proof of Lemma \ref{lem:w}, we also have that
\begin{align*} \begin{split}
L + 2\tau + 4\e{-t} q_t^{-2}\nmf \tau^3 &= -\frac{1}{\tau} - \tau - \e{-t} q_t^{-2} \tau^3\nmf + 2\tau + 4\e{-t} q_t^{-2} \tau^3 \nmf\\ &= -\tau \big( \frac{1}{\tau^2} - 1 - 3\e{-t} q_t^{-2} \tau^2\nmf \big)=0\,,
\end{split} \end{align*}
hence
\begin{align} \label{eq:P' expansion}
P'(\wt m_t) = z + 2\wt m_t + 4\e{-t} q_t^{-2} {\wt m_t}^3\nmf = (z-L) + 2(\wt m_t -\tau) + 4 \e{-t} q_t^{-2}\nmf ({\wt m_t}^3 - \tau^3)\,,
\end{align}
and we find from~\eqref{L+ square root} that
$$
|\alpha_2|=|P'(\wt m_t)| \sim \sqrt{\kappa_t + \eta}\,.
$$

We remark that the parameter $\alpha_1$ is needed only for the proof of Theorem~\ref{prop:norm bound}; the proof of Proposition \ref{prop:local} can be done simply by substituting every $\alpha_1$ below with $|\alpha_2|$.

Recall that, for any $a, b \geq 0$ and ${\textsf p},{\textsf q} > 1$ with ${\textsf p}^{-1} + {\textsf q}^{-1} = 1$, Young's inequality states
\begin{align}\label{Young inequality}
ab \leq \frac{a^{\textsf p}}{{\textsf p}} + \frac{b^{\textsf q}}{{\textsf q}}\,.
\end{align}

Let $D\geq 10$ and choose any (small) $\epsilon>0$. All estimates below hold for $N$ sufficiently large (depending on $D$ and $\epsilon$). For brevity, $N$ is henceforth implicitly assumed to be sufficiently large. 
Using first that $\im m_t \leq \im \wt m_t + |m_t -\wt m_t| = \alpha_1 + \Lambda_t $ and then applying~\eqref{Young inequality} with ${\textsf p}=2D$ and ${\textsf q} =2D/(2D-1)$, we get for the first term on the right side of \eqref{eq:claim} that
\begin{align} \begin{split} \label{young1}
&N^{\epsilon} \left( \frac{\im m_t}{N\eta} + q_t^{-4} \right) |P(m_t)|^{2D-1} \leq N^{\epsilon} \frac{\alpha_1 + \Lambda_t}{N\eta} |P(m_t)|^{2D-1} + N^{\epsilon} q_t^{-4} |P(m_t)|^{2D-1} \\
&\qquad\leq \frac{N^{(2D+1)\epsilon}}{2D} \left( \frac{\alpha_1 + \Lambda_t}{N\eta} \right)^{2D} + \frac{N^{(2D+1)\epsilon}}{2D} q_t^{-8D} + \frac{2(2D-1)}{2D} \cdot N^{-\frac{\epsilon}{2D-1}} |P(m_t)|^{2D}\,.
\end{split} \end{align}
Similarly, for the second term on the right side of \eqref{eq:claim}, we have
\begin{align} \label{young2}
N^{-\epsilon/8} q_t^{-1} \Lambda_t^2 \left| P(m_t) \right|^{2D-1} \leq \frac{N^{-(D/4 -1)\epsilon}}{2D} q_t^{-2D} \Lambda_t^{4D} + \frac{2D-1}{2D}  N^{-\frac{\epsilon}{2D-1}} |P(m_t)|^{2D}\,.
\end{align}
From the Taylor expansion of $P'(m_t)$ around $\wt m_t$, we have
\begin{align}
|P'(m_t) - P'(\wt m_t) - P''(\wt m_t) (m_t - \wt m_t)| \le C q_t^{-2} \Lambda_t^2\,,
\end{align}
and $|P'(m_t)| \leq |\alpha_2| + 3\Lambda_t$, for all $z\in\mathcal{D}$, with high probability since $P''(\wt m_t) = 2 + O(q_t^{-2})$ and $\Lambda_t\prec 1$ by assumption. We note that, for any fixed $s\ge 2$,
\begin{align*}\begin{split}
(\alpha_1 + \Lambda_t)^{2s-s'-2} (|\alpha_2| + 3\Lambda_t)^{s'} &\leq N^{\epsilon/2} (\alpha_1 + \Lambda_t)^{s-1} (|\alpha_2| + 3\Lambda_t)^{s-1}\\
&\leq N^{\epsilon} (\alpha_1 + \Lambda_t)^{s/2} (|\alpha_2| + 3\Lambda_t)^{s/2}
\end{split}\end{align*}
with high probability, uniformly on $\caD$, since $\alpha_1 \leq |\alpha_2| \leq C$ and $\Lambda_t \prec 1$. In the third term of \eqref{eq:claim}, note that $2s-s'-2 \geq s$ since $s' \leq s-2$. Hence, for $2 \leq s \leq 2D$,
\begin{align} \label{young3} \begin{split}
& {N^{\epsilon}}q_t^{-1} \left( \frac{\im m_t}{N\eta} \right)^{2s-s'-2} \left| P'(m_t) \right|^{s'} \left| P(m_t) \right|^{2D-s}\\
&\qquad\qquad\leq {N^{\epsilon} }q_t^{-1} \beta^s (\alpha_1 + \Lambda_t)^{2s-s'-2} (|\alpha_2| + 3\Lambda_t)^{s'} \left| P(m_t) \right|^{2D-s} \\
&\qquad\qquad\leq N^{2\epsilon} q_t^{-1} \beta^s (\alpha_1 + \Lambda_t)^{s/2} (|\alpha_2| + 3\Lambda_t)^{s/2} \left| P(m_t) \right|^{2D-s} \\
&\qquad\qquad\leq N^{2\epsilon} q_t^{-1} \frac{s}{2D} \beta^{2D} (\alpha_1 + \Lambda_t)^D (|\alpha_2| + 3\Lambda_t)^{D} + N^{2\epsilon} q_t^{-1} \frac{2D-s}{2D} \left| P(m_t) \right|^{2D}
\end{split} \end{align}
uniformly on $\mathcal{D}$ with high probability. For the last term in~\eqref{eq:claim}, we note that
\begin{align}
{ \frac{1}{N\eta} } + q_t^{-1}\bigg({\frac{\im m_t}{N\eta}}\bigg)^{1/2} + q_t^{-2}  \prec \beta\,,
\end{align}
uniformly on $\caD$. Thus, similar to \eqref{young3} we find that, for $2 \leq s \leq 2D$,
\begin{align} \label{young4}
& N^{\epsilon} \bigg( { \frac{1}{N\eta} } + q_t^{-1}\bigg({\frac{\im m_t}{N\eta}}\bigg)^{1/2} + q_t^{-2} \bigg) \bigg( \frac{\im m_t}{N\eta} \bigg)^{s-1} \left| P'(m_t) \right|^{s-1}  \left| P(m_t) \right|^{2D-s}\nonumber \\
&\qquad\leq N^{2\epsilon} \beta \cdot \beta^{s-1} (\alpha_1 + \Lambda_t)^{s/2} (|\alpha_2| + 3\Lambda_t)^{s/2} \left| P(m_t) \right|^{2D-s} \nonumber\\
&\qquad\leq \frac{s}{2D} \left( N^{2\epsilon} N^{\frac{(2D-s)\epsilon}{4D^2}} \right)^{\frac{2D}{s}} \beta^{2D} (\alpha_1 + \Lambda_t)^D (|\alpha_2| + 3\Lambda_t)^{D} + \frac{2D-s}{2D} \left( N^{-\frac{(2D-s)\epsilon}{4D^2}} \right)^{\frac{2D}{2D-s}} \left| P(m_t) \right|^{2D} \nonumber\\
&\qquad\leq N^{(2D+1)\epsilon} \beta^{2D} (\alpha_1 + \Lambda_t)^D (|\alpha_2| + 3\Lambda_t)^{D} + N^{-\frac{\epsilon}{2D}} \left| P(m_t) \right|^{2D}\,,
 \end{align}
for all $z\in\caD$, with high probability. We hence have from \eqref{eq:claim}, \eqref{young1}, \eqref{young2}, \eqref{young3} and \eqref{young4} that
\begin{align} \begin{split}
\E \left[|P(m_t)|^{2D}\right] &\leq N^{(2D+1)\epsilon} \E \left[ \beta^{2D} (\alpha_1 + \Lambda_t)^D (|\alpha_2| + 3\Lambda_t)^D \right] + \frac{N^{(2D+1)\epsilon}}{2D} q_t^{-8D} \\
&\qquad + \frac{N^{-(D/4 -1)\epsilon}}{2D} q_t^{-2D} \E \left[\Lambda_t^{4D}\right] + C N^{-\frac{\epsilon}{2D}} \E \left[ |P(m_t)|^{2D} \right]\,,
\end{split} \end{align}
for all $z\in\mathcal{D}$. Note that the last term on the right side can be absorbed into the left side. Thence
\begin{align} \begin{split} \label{eq:claim'}
&\E \left[|P(m_t)|^{2D}\right] \\
&\leq C N^{(2D+1)\epsilon} \E \left[ \beta^{2D} (\alpha_1 + \Lambda_t)^D (|\alpha_2| + 3\Lambda_t)^D \right] + \frac{N^{(2D+1)\epsilon}}{D} q_t^{-8D} + \frac{N^{-(D/4 -1)\epsilon}}{D} q_t^{-2D} \E \left[\Lambda_t^{4D}\right] \\
&\leq N^{3D\epsilon} \beta^{2D} |\alpha_2|^{2D} + N^{3D\epsilon} \beta^{2D} \E \left[ \Lambda_t^{2D} \right] + N^{3D \epsilon} q_t^{-8D} + N^{-D \epsilon/8} q_t^{-2D} \E \left[\Lambda_t^{4D}\right],
\end{split} \end{align}
uniformly on $\mathcal{D}$, where we used that $D> 10$ and the inequality 
\begin{align}\label{convex inequality}
(a+b)^{\textsf p} \leq 2^{{\textsf p}-1}(a^{\textsf p} + b^{\textsf p})\,,
\end{align} 
for any $a, b \geq 0$ and ${\textsf p} \geq 1$, to get the second line.

Next, from the third order Taylor expansion of $P(m_t)$ around $\wt m_t$, we have
\begin{align}\label{le taylor of P}
\left| P(m_t) - \alpha_2 (m_t - \wt m_t) - \frac{P''(\wt m_t)}{2} (m_t - \wt m_t)^2 \right| \le C q_t^{-2} \Lambda_t^3
\end{align}
since $P(\wt m_t) = 0$ and $P'''(\wt m_t)=4!\e{-t} q_t^{-2}\nmf\wt m_t$. Thus, using $\Lambda_t\prec 1$ and $P''(\wt m_t)=2+O(q_t^{-2}) $ we~get
\begin{align}
 \Lambda_t^2\prec 2|\alpha_2|\Lambda_t+2|P(m_t)|\,,\qquad\qquad (z\in\mathcal{D})\,.
\end{align}
Taking the $2D$-power of the inequality and using once more~\eqref{convex inequality}, we get after taking the expectation
\begin{align}
\E [\Lambda_t^{4D}]\le 4^{2D}N^{\epsilon/2}|\alpha_2|^{2D}{\E[ \Lambda_t^{2D}]}+4^{2D}N^{\epsilon/2}\E[ |P(m_t)|^{2D}]\,,\qquad\qquad(z\in\mathcal{D})\,.
\end{align}
Replacing from~\eqref{eq:claim'} for $\E[ |P(m_t)|^{2D}]$ we obtain, using that $4^{2D}\le N^{\epsilon/2}$, for $N$ sufficiently large,
\begin{align}\begin{split}\label{annoying}
\E [\Lambda_t^{4D}]&\le N^{\epsilon}|\alpha_2|^{2D}{\E[ \Lambda_t^{2D}]} + N^{(3D+1)\epsilon} \beta^{2D} |\alpha_2|^{2D} +N^{(3D+1)\epsilon} \beta^{2D} \E \left[ \Lambda_t^{2D} \right]+N^{(3D+1) \epsilon} q_t^{-8D}\\ &\qquad+ N^{-D \epsilon/8+\epsilon} q_t^{-2D} \E \big[\Lambda_t^{4D}\big]\,,
\end{split}\end{align}
uniformly on $\mathcal{D}$. Applying Schwarz inequality to the first term and the third term on the right, absorbing the terms $o(1)\E[\Lambda_t^{4D}]$ into the left side and using $q_t^{-2}\le\beta$ in the fourth term, we get
\begin{align}\begin{split}\label{ESCL}
\E [\Lambda_t^{4D}]&\le N^{2\epsilon}|\alpha_2|^{4D} + N^{(3D+2)\epsilon} \beta^{2D} |\alpha_2|^{2D} +N^{(3D+2)\epsilon} \beta^{4D}\,,
\end{split}\end{align} 
uniformly on $\mathcal{D}$. Feeding~\eqref{ESCL} back into~\eqref{eq:claim'} we get, for any $D\ge 10$ and (small) $\epsilon>0$,
\begin{align} \begin{split} \label{eq:claim''}
\E \left[|P(m_t)|^{2D}\right]
&\leq N^{3D\epsilon} \beta^{2D} |\alpha_2|^{2D} + N^{3D\epsilon} \beta^{2D} \E \left[ \Lambda_t^{2D} \right] + N^{(3D+1) \epsilon} \beta^{4D} + q_t^{-2D} |\alpha_2|^{4D} \\
&\le  N^{5D\epsilon} \beta^{2D} |\alpha_2|^{2D} + N^{5D \epsilon} \beta^{4D} + q_t^{-2D} |\alpha_2|^{4D} \,,
\end{split} \end{align}
uniformly on $\mathcal{D}$, for $N$ sufficiently large, where we used Schwarz inequality and once more~\eqref{ESCL} to get the second line. 

By Markov's inequality, we therefore obtain from~\eqref{eq:claim''} that for fixed $z\in\mathcal{D}$,  $|P(m_t)|\prec |\alpha_2|\beta+\beta^2 + q_t^{-1} |\alpha_2|^2$. It then follows from the Taylor expansion of $P(m_t)$ around $\wt m_t$ in~\eqref{le taylor of P} that
\begin{align} \begin{split} \label{eq:claim'''}
      |\alpha_2(m_t-\wt m_t)+(m_t-\wt m_t)^2|\prec \beta{\Lambda_t^2}+|\alpha_2|\beta+\beta^2 + q_t^{-1}|\alpha_2|^2,
     \end{split}
\end{align}
for each fixed $z\in\mathcal{D}$, where we used that $q_t^{-2}\le \beta$. To get a uniform bound on $\mathcal{D}$, we choose $18 N^{8}$~lattice points $z_1, z_2,\ldots, z_{18 N^{8}}$ in $\caD$ such that, for any $\wt z \in \caD$, there exists $z_n$ satisfying $|\wt z-z_n| \leq N^{-4}$. Since
$$
|m_t(\wt z) - m_t(z_n)| \leq |\wt z-z_n| \sup_{z \in \caD} \left| \frac{\partial m_t(z)}{\partial z} \right| \leq |\wt z-z_n| \sup_{z \in \caD} \frac{1}{(\im z)^2} \leq N^{-2}
$$
and since a similar estimate holds for $|\wt m_t(\wt z) - \wt m_t(z)|$, a union bound yields that~\eqref{eq:claim'''} holds uniformly on $\mathcal{D}$ with high probability. In particular, for any (small) $\epsilon>0$ and (large) $D$ there is an event $\wt\Xi$ with $\P(\wt\Xi)\ge 1-N^D$ such that, for all $z\in\mathcal{D}$,
\begin{align}\label{SCL}
   |\alpha_2(m_t-\wt m_t)+(m-m_t)^2|\le N^{\epsilon}\beta{\Lambda_t^2}+N^{\epsilon}|\alpha_2|\beta+N^{\epsilon}\beta^2 +N^{\epsilon} q_t^{-1}|\alpha_2|^2,
\end{align}
on $\wt\Xi$, for $N$ sufficiently large.

Recall next that there is a constant $C_0>1$ such that $C_0^{-1}\sqrt{\kappa_t(E)+\eta}\le |\alpha_2|\le C_0 \sqrt{\kappa_t(E)+\eta}$, where we can choose $C_0$ uniform in $z\in\mathcal{E}$. Note further that $\beta=\beta(E+\ii\eta)$ is for fixed $E$ a decreasing function of~$\eta$ while $\sqrt{\kappa_t(E)+\eta}$ is increasing. Thus, there exists $\wt \eta_0 \equiv \wt \eta_0(E)$ such that $\sqrt{\kappa(E)+ \wt\eta_0} = C_0 q_t \beta(E+\ii \wt\eta_0)$. We then consider the subdomain $\wt\caD \subset \caD$ defined by
\begin{align}
\wt\caD \deq \left\{ z=E+\ii\eta\in\caD\,:\, \eta> \wt\eta_0(E) \right\}.
\end{align}
On $\wt\caD$, $\beta \leq q_t^{-1} |\alpha_2|$, hence we obtain from the estimate~\eqref{SCL} that
\begin{align*}
|\alpha_2(m_t-\wt m_t)+(m-m_t)^2|\le N^{\epsilon}\beta{\Lambda_t^2}+ 3N^{\epsilon} q_t^{-1}|\alpha_2|^2
\end{align*}
and thus	
\begin{align*}
|\alpha_2|\Lambda_t\le  (1+N^\epsilon\beta)\Lambda_t^2+ 3N^{\epsilon} q_t^{-1}|\alpha_2|^2\,,
\end{align*}
uniformly on $\wt\caD$ on $\wt\Xi$. Hence, we get on $\wt\Xi$ that either
\begin{align}\label{le first dichotomy}
|\alpha_2|\le 4 \Lambda_t\qquad\qquad\textrm{ or }\qquad\qquad \Lambda_t\le 6N^{\epsilon} q_t^{-1}|\alpha_2|\,,\qquad\qquad (z\in\wt\caD)\,.
\end{align}
Note that any $z \in \caE$ with $\eta=\im z = 3$ is in $\wt\caD$. When $\eta =3$, we easily see that 
$$
|\alpha_2| \geq |z + 2\wt m_t| - Cq_t^{-2} \geq \eta=3 \gg 6N^{\epsilon} q_t^{-1}|\alpha_2|\,,
$$
for sufficiently large $N$. In particular we have that either $3/4\le\Lambda_t$ or $\Lambda_t\le 6N^{\epsilon} q_t^{-1}|\alpha_2|$ on $\wt\Xi$ for $\eta=3$. Moreover, since $m_t$ and $\wt m_t$ are Stieltjes transforms, we have
$$
\Lambda_t \leq \frac{2}{\eta}=\frac{2}{3}\,.
$$
We conclude that, for $\eta=3$, the second possibility, $\Lambda_t\le 6N^{\epsilon} q_t^{-1}|\alpha_2|$ holds on $\wt\Xi$. Since $6N^{\epsilon} q_t^{-1} \ll 1$ on $\wt\caD$, in particular $6N^{\epsilon} q_t^{-1}|\alpha_2| < |\alpha_2|/8$, we find from~\eqref{le first dichotomy} by continuity that 
\begin{align}\label{eq:claim' mod}
\Lambda_t \le 6N^{\epsilon} q_t^{-1}|\alpha_2|\,,\qquad\qquad (z\in\wt\caD)\,,
\end{align}
holds on the event $\wt\Xi$. Putting the estimate~\eqref{eq:claim' mod} back into~\eqref{eq:claim'}, we find that
\begin{align} \begin{split} \label{eq:claim4}
\E\left[|P(m_t)|^{2D}\right] &\leq N^{4D\epsilon} \beta^{2D} |\alpha_2|^{2D} + N^{3D \epsilon} q_t^{-8D} + q_t^{-6D} |\alpha_2|^{4D} \\
&\leq N^{6D\epsilon} \beta^{2D} |\alpha_2|^{2D} + N^{6D \epsilon} \beta^{4D} \,,
\end{split} \end{align}
for any (small) $\epsilon>0$ and (large) $D$, uniformly on $\wt\caD$. Note that, for $z \in \caD \backslash \wt\caD$, the estimate $\E[|P(m_t)|^{2D}] \leq N^{6D\epsilon} \beta^{2D} |\alpha_2|^{2D} + N^{6D \epsilon} \beta^{4D}$ can be directly checked from \eqref{eq:claim''}. Considering lattice points $\{ z_i \} \subset \caD$ again, a union bound yields for any (small) $\epsilon>0$ and (large) $D$ there is an event~$\Xi$ with $\P(\Xi)\ge 1-N^D$ such that
\begin{align}\label{SCL2}
   |\alpha_2(m_t-\wt m_t)+(m-m_t)^2|\le N^{\epsilon}\beta{\Lambda_t^2}+N^{\epsilon}|\alpha_2|\beta+N^{\epsilon}\beta^2,
\end{align}
on $\Xi$, uniformly on $\caD$ for $N$ sufficiently large.

 Next, recall that $\beta=\beta(E+\ii\eta)$ is for fixed $E$ a decreasing function of~$\eta$ while $\sqrt{\kappa_t(E)+\eta}$ is increasing. Thus there is $\eta_0\equiv\eta_0(E)$ such that $\sqrt{\kappa(E)+\eta_0}=10C_0N^\epsilon\beta(E+\ii\eta_0)$. Further notice that~$\eta_0(E)$ is a continuous function. We consider the three subdomains of $\caE$ defined by
\begin{align}\nonumber\begin{split}
 \caE_1&\deq\left\{z=E+\ii\eta\in\caE\,:\, \eta\le \eta_0(E),10N^{\epsilon} \le N\eta\right\},\\
 \caE_2&\deq\left\{z=E+\ii\eta\in\caE\,:\, \eta> \eta_0(E),10N^{\epsilon} \le N\eta\right\},\\
 \caE_3&\deq\left\{z=E+\ii\eta\in\caE\,:10N^{\epsilon}> N\eta\right\}.
\end{split}\end{align}
Note that $\caE_1\cup\caE_2\subset\mathcal{D}$. We split the stability analysis of~\eqref{SCL2} according to whether $z\in\caE_1$, $\caE_2$ or $\caE_3$.

{\it Case 1:} If $z\in \caE_1$, we note that $|\alpha_2|\le C_0 \sqrt{\kappa(E)+\eta}\le 10C_0^2 N^\epsilon\beta(E+\ii\eta)$. We then obtain from~\eqref{SCL2} that
\begin{align}
\Lambda_t^2\le |\alpha_2|\Lambda_t+N^{\epsilon}\beta\Lambda_t^2+N^{\epsilon}|\alpha_2|\beta+ N^{\epsilon} \beta^2\le 10C_0^2N^{\epsilon}\beta\Lambda_t+N^{\epsilon}\beta\Lambda_t^2+(10C_0^2N^\epsilon+1)N^{\epsilon}\beta^2\,,\nonumber
\end{align}
on $\Xi$. Thus, 
\begin{align}\label{jjj}
\Lambda_t\le CN^{\epsilon}\beta\,,\qquad\qquad (z\in\caE_1)\,,
\end{align}
on $\Xi$, for some finite constant $C$.

{\it Case 2:} If $z\in\caE_2$, we obtain from~\eqref{SCL2} that 
\begin{align}
|\alpha_2|\Lambda_t\le  (1+N^\epsilon\beta)\Lambda_t^2+|\alpha_2|N^{\epsilon}\beta+N^{\epsilon}\beta^2\,,
\end{align}
on $\Xi$. We then note that $C_0|\alpha_2|\ge \sqrt{\kappa_t(E)+\eta}\ge 10C_0  N^\epsilon\beta$, \ie $N^\epsilon\beta\le |\alpha_2|/10$, so that
\begin{align}
|\alpha_2|\Lambda_t\le 2 \Lambda_t^2+(1+ N^{-\epsilon})|\alpha_2|N^{\epsilon}\beta\,,
\end{align}
on $\Xi$, where we used that $N^\epsilon\beta\le 1$. Hence, we get on $\Xi$ that either
\begin{align}
|\alpha_2|\le 4 \Lambda_t\qquad\qquad\textrm{ or }\qquad\qquad \Lambda_t\le  3N^\epsilon\beta\,,\qquad\qquad (z\in\caE_2)\,.
\end{align}
We follow the dichotomy argument and the continuity argument that were used to obtain~\eqref{eq:claim' mod}. Since $3N^{\epsilon} \beta\le |\alpha_2|/8$ on $\caE_2$, we find by continuity that 
\begin{align}\label{jjj2}
\Lambda_t \le 3N^{\epsilon} \beta\,,\qquad\qquad (z\in\caE_2)\,,
\end{align}
holds on the event $\Xi$.

 {\it Case 3:} For $z\in\caE_3=\caE\backslash (\caE_1\cup\caE_2)$ we use that $|m'_t(z)|\le \im m_t(z)/\im z$, $z\in\C^+$, since $m_t$ is a Stieltjes transform. Set now $ \wt\eta\deq 10N^{-1+\epsilon}$. By the fundamental theorem of calculus we can estimate
\begin{align}\begin{split}\nonumber
 |m_t(E+\ii\eta)|&\le \int_{\eta}^{\wt\eta}\frac{\im m_t(E+\ii s)}{s}\dd s+|m_t(E+\ii\wt\eta) |\\
 &\le   \int_{\eta}^{\wt\eta}\frac{s \im m_t(E+\ii s)}{s^2}\dd s+\Lambda_t(E+\ii\wt\eta)+|\wt m_t(E+\ii\wt\eta)|\,.
\end{split}\end{align}
Using that $s\rightarrow s \im m_t(E+\ii s)$ is a monotone increasing function as is easily checked from the definition of the Stieltjes transform, we find that
\begin{align}\begin{split}
|m_t(E+\ii\eta)|&\le \frac{2\wt\eta}{\eta} \im m_t(E+\ii\wt\eta)+\Lambda_t(E+\ii\wt\eta)+|\wt m_t(E+\ii\wt\eta)|\\
&\le C\frac{N^{\epsilon}}{N\eta}\big(\im \wt m_t(E+\ii\wt \eta)+\Lambda_t(E+\ii\wt \eta)\big)  +|\wt m_t(E+\ii\wt\eta)|\,,
\end{split}\end{align}
for some $C$ where we used $\wt\eta\deq 10N^{-1+\epsilon}$ to get the second line. Thus noticing that $z=E+\ii\wt \eta\in \caE_1\cup\caE_2$, hence, on the event $\Xi$ introduced above, we have $\Lambda_t(E+\ii\wt\eta)\le CN^\epsilon\beta(E+\ii\wt \eta) \leq C$ by~\eqref{jjj} and~\eqref{jjj2}. Using moreover that $\wt m_t$ is uniformly bounded by a constant on $\caE$, we then get that, on the event $\Xi$,
 \begin{align}\label{jjj3}
 \Lambda_t\le CN^\epsilon\beta\,,\qquad \qquad( z\in\caE_3)\,.
 \end{align}

Combining~\eqref{jjj},~\eqref{jjj2} and~\eqref{jjj3}, and recalling the definition of the event $\Xi$, we get $\Lambda_t\prec \beta$, uniformly on $\caE$ for fixed $t\in[0,6\log N]$. Choosing $t=0$, we have completed the proof of Theorem~\ref{theorem:local}. To extend this bound to all $t\in[0,6\log N]$, we use the continuity of the Dyson matrix flow. Choosing a lattice $\mathcal{L}\subset[0,6\log N]$ with spacings of order $N^{-3}$, we get $\Lambda_t\prec \beta$, uniformly on $\mathcal{E}$ and on $\mathcal{L}$, by a union bound. By continuity we extend the conclusion to all of $[0,6\log N]$. This proves~Proposition~\ref{prop:local}.
\end{proof}

\subsection{Proof of Theorem~\ref{prop:norm bound}}
We start with an upper bound on the largest eigenvalue $\lambda_1^{H_t}$ of $H_t$.
\begin{lemma} \label{le lemma norm bound} Let $H_0$ satisfy Assumption~\eqref{assumption H} with $\phi>0$.
 Let $L_t$ be deterministic number defined in Lemma \ref{lem:w}. Then, 
\begin{align} \label{eq:norm bound}
 \lambda_1^{H_t} - L_t\prec \frac{1}{q_t^4} +\frac{1}{N^{2/3}}\,,
\end{align}
uniformly in $t\in[0,6\log N]$.
\end{lemma}

\begin{proof}
To prove Lemma~\ref{le lemma norm bound} we follow the strategy of the proof of Lemma~4.4 in~\cite{EKYY1}.
Fix $t\in[0,6\log N]$. Recall first the deterministic $z$-dependent parameters
\begin{align}
\alpha_1 \deq \im \wt m_t\,, \qquad \alpha_2 \deq P'(\wt m_t)\,, \qquad \beta \deq\frac{1}{q_t^2}+ \frac{1}{N\eta} \,.
\end{align}
We mostly drop the $z$-dependence for brevity. We further introduce the $z$-independent quantity
\begin{align}
\wt\beta \deq \left( \frac{1}{q_t^4}+\frac{1}{N^{2/3}} \right)^{1/2}.
\end{align}
Fix a small $\epsilon > 0$ and define the domain $\caD_{\epsilon}$ by
\begin{align}
\caD_{\epsilon} \deq \bigg \{ z = E + \ii \eta : N^{4\epsilon}\wt\beta^2 \leq \kappa_t \leq q_t^{-1/3}\,, \; \eta = \frac{N^{\epsilon}}{N \sqrt{\kappa_t}} \bigg\}\,,
\end{align}
where $\kappa_t\equiv \kappa_t(E)=E-L_t$. Note that on $\caD_\epsilon$,
\begin{align*}
N^{-1+\epsilon}\ll \eta \leq \frac{N^{ -{\epsilon}}}{N\wt\beta}\,, \qquad\qquad \kappa \geq N^{5\epsilon} \eta\,.
\end{align*}
In particular we have $N^\epsilon\wt\beta\le (N\eta)^{-1}$, hence $N^\epsilon q_t^{-2}\le C (N\eta)^{-1}$ so that $q_t^{-2}$ is negligible when compared to $(N\eta)^{-1}$ and $\beta$ on $\caD_\epsilon$. Note moreover that
 \begin{align}\label{the alphas}\begin{split}
 |\alpha_2| &\sim \sqrt{\kappa_t+\eta}\sim\sqrt{\kappa_t}=  \frac{N^{\epsilon} }{N\eta}\sim N^\epsilon\beta\,,\\
 \alpha_1 &= \im \wt m_t \sim \frac{\eta}{\sqrt{\kappa_t + \eta}}\sim\frac{\eta}{\sqrt{\kappa_t}} \leq N^{-5\epsilon}\sqrt{\kappa_t} \sim N^{-5\epsilon} |\alpha_2|  \sim N^{-4\epsilon} \beta\,.
 \end{split}\end{align}
In particular we have $\alpha_1\ll |\alpha_2|$ on $\caD_\epsilon$.

We next claim that
\begin{align*}
\Lambda_t\deq |m_t - \wt m_t| \ll \frac{1}{N\eta}
\end{align*}
with high probability on the domain $\caD_{\epsilon}$. 

Since $\caD_{\epsilon} \subset \caE$, we find from Proposition~\ref{prop:local} that $\Lambda_t \leq N^{\epsilon'}\beta$ for any $\epsilon' > 0$ with high probability. Fix $0 < \epsilon' < \epsilon/7 $. From \eqref{eq:claim'}, we get
\begin{align}
\E \left[|P(m_t)|^{2D}\right]&\leq C N^{(4D-1)\epsilon'} \E \left[ \beta^{2D} (\alpha_1 + \Lambda_t)^D (|\alpha_2| + 3\Lambda_t)^D \right] \nonumber\\ &\qquad\quad+ \frac{N^{(2D+1)\epsilon'}}{D} q_t^{-8D} + \frac{N^{-(D/4 -1)\epsilon'}}{D} q_t^{-2D} \E \left[\Lambda_t^{4D}\right]\nonumber \\
&\leq C^{2D} N^{6D \epsilon'} \beta^{4D} + \frac{N^{(2D+1)\epsilon'}}{D} q_t^{-8D} + \frac{N^{4D\epsilon'}}{D} q_t^{-2D} \beta^{4D}\nonumber\\
&\leq C^{2D} N^{6D \epsilon'} \beta^{4D} \,,\nonumber
\end{align}
for $N$ sufficiently large, where we used that $\Lambda_t\le N^{\epsilon'}\beta\ll N^\epsilon\beta$ with high probability and, by~\eqref{the alphas}, $\alpha_1\ll|\alpha_2|$, $|\alpha_2|\le CN^\epsilon\beta$ on $\caD_\epsilon$.
 Applying $(2D)$-th order Markov inequality and a simple lattice argument combined with a union bound, we get $|P(m_t)| \leq CN^{3\epsilon'} \beta^{2}$ uniformly on $\caD_\epsilon$ with high probability. From the Taylor expansion of $P(m_t)$ around $\wt m_t$ in~\eqref{le taylor of P}, we then get that
\begin{align}\label{choco}
|\alpha_2| \Lambda_t \leq 2 \Lambda_t^2 + CN^{3\epsilon'} \beta^{2}\,,
\end{align}
uniformly on $\caD_\epsilon$ with high probability, where we also used that $\Lambda_t\ll 1$ on $\caD_\epsilon$ with high probability.

Since $\Lambda_t\le N^{\epsilon'} \beta\le C N^{\epsilon'-\epsilon}|\alpha_2|$ with high probability on $\caD_\epsilon$, we have $|\alpha_2|\Lambda_t\ge  CN^{\epsilon-\epsilon'}\Lambda_t^2\gg 2\Lambda_t^2$. Thus the first term on the right side of~\eqref{choco} can be absorbed into the left side and we conclude that
\begin{align}\nonumber
 \Lambda_t  \leq C N^{3\epsilon'}\frac{\beta}{|\alpha_2|}\beta\le CN^{3\epsilon'-\epsilon}\beta\,,
\end{align}
must hold with high probability on $\caD_\epsilon$. Hence, using that $0<\epsilon'<\epsilon/7$, we obtain that
\begin{align}\nonumber
\Lambda_t \leq N^{-\epsilon/2}\beta\le 2\frac{N^{-\epsilon/2}}{N\eta}\,,
\end{align}
with high probability on $\caD_\epsilon$. This proves the claim that $\Lambda_t \ll (N\eta)^{-1}$ on $\caD_{\epsilon}$ with high probability. Moreover, this also shows that
\begin{align}\label{the neck}
\im m_t \leq \im \wt m_t + \Lambda_t=\alpha_1+\Lambda_t \ll\frac{1}{N\eta}\,,
\end{align}
on $\caD_{\epsilon}$ with high probability, where we used~\eqref{the alphas}.

Now we prove the estimate \eqref{eq:norm bound}. If $\lambda_1^{H_t} \in [E-\eta, E+\eta]$ for some $E \in [L_t +N^\epsilon(q_t^{-4}+ N^{-2/3}), L_t +{q}^{-1/3}]$ with $z = E + \ii \eta \in \caD_{\epsilon}$,
\begin{align}
\im m_t (z) \geq \frac{1}{N} \im \frac{1}{\lambda_1^{H_t} -E - \ii \eta} = \frac{1}{N} \frac{\eta}{(\lambda_1^{H_t}-E)^2 + \eta^2} \geq \frac{1}{5N\eta}\,,
\end{align}
which contradicts the high probability bound $\im m_t \ll (N\eta)^{-1}$ in~\eqref{the neck}. The size of each interval $[E-\eta, E+\eta]$ is at least $N^{{ -1+\epsilon}} q_t^{1/6}$. Thus, considering $O({ N})$ such intervals, we can conclude that $\lambda_1 \notin [L + N^\epsilon(q_t^{-4}+N^{-2/3}), L +q_t^{-1/3}]$ with high probability. From Proposition~\ref{a priori norm bound}, we find that $\lambda_1^{H_t}-L_t\prec q_t^{-1/3}$ with high probability, hence we conclude that \eqref{eq:norm bound} holds, for fixed $t\in[0,6\log N]$. Using a lattice argument and the continuity of the Dyson matrix flow, we easily obtain~\eqref{eq:norm bound} uniformly in $t\in[0,6\log N]$.
\end{proof}

We are now well-prepared for the proof of Theorem~\ref{prop:norm bound}. It follows immediately from the next result.
\begin{lemma}\label{lemma friday evenig}
 Let $H_0$ satisfy Assumption~\eqref{assumption H} with $\phi>0$. Then, uniformly in $t\in[0,6\log N]$,
\begin{align} \label{eq:norm bound t}
 \left|\| H_t\| - L_t\right|\prec \frac{1}{q_t^4} +\frac{1}{N^{2/3}}\,.
\end{align}

\end{lemma}

\begin{proof}[Proof of Lemma~\ref{lemma friday evenig} and Theorem~\ref{prop:norm bound}]
 Fix $t\in[0,\log N]$. Consider the largest eigenvalue $\lambda_1^{H_t}$. In Proposition~\ref{le lemma norm bound}, we already showed that $(L_t-\lambda_1^{H_t})_-\prec q_t^{-4}+N^{-2/3}$. It thus suffices to consider $(L_t-\lambda_1^{H_t})_+$. By Lemma~\ref{lem:w} there is $c>0$ such that $ c(L_t-\lambda_1^{H_t})_+^{3/2}\le n_{\wt\rho_t}(\lambda_1^{H_t},L_t)$. Hence, by Corollary~\ref{corollary density of states} (and its obvious generalization to $H_t$), we have the estimate
 \begin{align}
 (L_t-\lambda_1^{H_t})_+^{3/2}\prec \frac{(L_t-\lambda_1^{H_t})_+}{q_t^2}+\frac{1}{N}\,,
  \end{align}
 so that $(L_t-\lambda_1^{H_t})_+\prec q_t^{-4}+N^{-2/3}$. Thus $|\lambda_1^{H_t}-L_t|\prec q_t^{-4}+N^{-2/3}$. Similarly, one shows the estimate $|\lambda_N^{H_t}+L_t|\prec  q_t^{-4}+N^{-2/3}$ for the smallest eigenvalue $\lambda_N^{H_t}$. This proves~\eqref{eq:norm bound t} for fixed $t\in[0,6\log N]$. Uniformity follows easily from the continuity of the Dyson matrix flow.
\end{proof}

\section{Recursive moment estimate: Proof of Lemma \ref{lem:claim}} \label{sec:stein}

In this section, we prove Lemma \ref{lem:claim}. Recall the definitions of the Green functions $G_t$ and $m_t$ in~\eqref{le timedependent green functions}. We fix $t\in[0,6\log N]$ throughout this section, and we will omit $t$ from the notation in the matrix~$H_t$, its matrix elements and its Green functions. Given a (small) $\epsilon>0$, we introduce the $z$-dependent control parameter $\Phiepsi\equiv\Phiepsi(z)$ by setting
\begin{align} \begin{split} \label{eq:phi} 
\Phiepsi(z) &\deq N^{\epsilon} \, \E \bigg[ \bigg(\frac{1}{q_t^4}+ \frac{\im m_t}{N\eta}   \bigg) \big| P(m_t) \big|^{2D-1} \bigg] + N^{-\epsilon/4} q_t^{-1} \, \E \bigg[ |m_t - \wt m_t|^2 \big| P(m_t) \big|^{2D-1} \bigg] \\
& \qquad+ {N^{\epsilon}}{ q_t}^{-1}\,  \sum_{s=2}^{2D} \sum_{s'=0}^{s-2} \E \bigg[ \bigg( \frac{\im m_t}{N\eta} \bigg)^{2s-s'-2} \big| P'(m_t) \big|^{s'} \big| P(m_t) \big|^{2D-s} \bigg]  +N^{\epsilon} q_t^{-8D}\\
&\qquad+ N^{\epsilon} \, \sum_{s=2}^{2D} \E \bigg[ \bigg( { \frac{1}{N\eta} } + \frac{1}{q_t}\bigg({\frac{\im m_t}{N\eta}}\bigg)^{1/2} + \frac{1}{q_t^{2}} \bigg) \bigg( \frac{\im m_t}{N\eta} \bigg)^{s-1} \big| P'(m_t) \big|^{s-1} \big| P(m_t) \big|^{2D-s} \bigg] \,.
\end{split} \end{align}

Recall the domain $\mathcal{D}$ defined in~\eqref{domain}. Lemma~\ref{lem:claim} then states that, for any (small) $\epsilon>0$,
\begin{align}
 \E [ |P|^{2D}(z) ] \le\Phiepsi(z)\,,\qquad\qquad( z\in\mathcal{D})\,,
\end{align}
for $N$ sufficiently large. We say that a {\it random variable $Z$ is negligible} if $|\E[Z]| \leq C \Phiepsi$ for some $N$-independent constant $C$.

 To prove the recursive moment estimate of Lemma~\ref{lem:claim}, we return to~\eqref{ziwui eins} which reads
 \begin{align}\begin{split}\label{ziwui}
  \E\big[(1+ zm)P^{{D-1}}\ol{P^D}\big]&= \frac{1}{N} \sum_{r=1}^\ell \frac{\kappa_t^{(r+1)}}{r!} \E \bigg[ \sum_{i \neq k} \partial_{ik}^r \Big( G_{ki} P^{D-1} \ol{P^D} \Big) \bigg]+\E\Omega_{\ell}\Big((1+zm) P^{D-1} \ol{P^D} \Big)\,,
 \end{split}\end{align}
where $\partial_{ik} = \partial/(\partial H_{ik})$ and $\kappa_t^{(k)}\equiv\kappa_t^{(k)}$ are the cumulants of $(H_t)_{ij}$, $i\not=j$. The detailed form of the error $\E\Omega_\ell(\cdot)$ is discussed in Subsection~\ref{subsection truncation of the cumulant expansion}.

It is convenient to condense the notation a bit. Abbreviate
\begin{align}\label{first time Ir}
\Ir\equiv  \Ir(z,m,D)\deq (1+zm) P(m)^{D-1}\ol{P(m)^D}\,.
\end{align}
We rewrite the cumulant expansion~\eqref{ziwui} as
\begin{align}\label{the IR}
 \E\Ir=\sum_{r=1}^\ell\sum_{s=0}^r w_{\Ir_{r,s}}\E\Ir_{r,s}+\E\Omega_\ell(I)\,,
\end{align}
where we set
\begin{align}\begin{split}\label{e1}
 \Ir_{r,s}\deq {N\kappa_t^{(r+1)}} \frac{1}{N^2} \sum_{i \neq k} \big( \partial_{ik}^{r-s} G_{ki} \big) \big( \partial_{ik}^s \big( P^{D-1} \ol{P^D} \big) \big)\,.
 \end{split}\end{align}
 (By convention $\partial_{ik}^0G_{ki}=G_{ki} $.) The weights $ w_{\Ir_{r,s}}$ are combinatoric coefficient given by
 \begin{align}
  w_{\Ir_{r,s}}\deq\frac{1}{r!}\binom{r}{s}=\frac{1}{(r-s)!s!	} \,.
 \end{align}
 
Returning to~\eqref{the source}, we have in this condensed form the expansion
\begin{align}\begin{split}\label{the short guy}
\E [ |P|^{2D} ] &= \sum_{r=1}^\ell\sum_{s=0}^r w_{\Ir_{r,s}}\E\Ir_{r,s}+ \E \bigg[ \Big( m^2 + \frac{\nm^{(4)}}{q_t^{2}} m^4 \Big) P^{D-1} \ol{P^D} \bigg] + \E\Omega_{\ell}(I)\,.
\end{split}\end{align}

\subsection{Truncation of the cumulant expansion}\label{subsection truncation of the cumulant expansion}
 In this subsection, we bound the error term $\E\Omega_\ell(I)$ in~\eqref{the short guy} for large $\ell$.  We need some more notation. Let $E^{[ik]}$ denote the $N\times N$ matrix determined~by
\begin{align}
 (E^{[ik]})_{ab}=\begin{cases}\delta_{ia}\delta_{kb}+\delta_{ib}\delta_{ka} \qquad &\textrm{if}\; i\not=k\,,\\
 \delta_{ia}\delta_{ib} &\textrm{if}\; i=k\,,
 \end{cases}
\qquad\qquad (i,k,a,b\in\llbracket1,N\rrbracket)\,.
\end{align}
For each pair of indices $(i,k)$, we define the matrix $H^{(ik)}$ from $H$ through the decomposition
\begin{align}\label{s notation}
 H=H^{(ik)}+H_{ik} E^{[ik]}\,.
\end{align}
With this notation we have the following estimate.

\begin{lemma}\label{le lemma first bound for error terms}
Suppose that $H$ satisfies Assumption~\ref{assumption H} with $\phi>0$. Let $i,k\in\llbracket1,N\rrbracket$, $D\in\N$ and $z\in\caD$. Define the function $F_{ki}$ by 
\begin{align}\label{le F function}
  F_{ki}(H)\deq G_{ki} P^{D-1}\ol P^D\,,
 \end{align}
where $G\equiv G^H(z)$ and $P\equiv P(m(z))$.  Choose an arbitrary $\ell\in\N$. Then, for any (small) $\epsilon>0$,
\begin{align}\label{le first bound for error terms}
 \E\bigg[\sup_{x\in\R, \,|x|\le q_t^{-1/2}}|\partial_{ik}^\ell F_{ki}(H^{(ik)}+xE^{[ik]})|\bigg]\le N^\epsilon\,,
\end{align}
uniformly $z\in\mathcal{D}$, for $N$ sufficiently large. Here $\partial_{ik}^\ell$ denotes the partial derivative $\frac{\partial^\ell}{\partial H_{ik}^\ell}$.
\end{lemma}

\begin{proof}
Fix two pairs of indices $(a,b)$ and $(i,k)$. From the definition of the Green function and~\eqref{s notation} we easily get
\begin{align*}
 G^{H^{(ik)}}_{ab}=G^{H}_{ab}+H_{ik}\big(G^{H^{(ik)}}E^{[ik]}G^{H}\big)_{ab}=G^{H}_{ab}+H_{ik}G_{ai}^{H^{(ik)}}G_{kb}^{H}+H_{ik} G_{ak}^{H^{(ik)}}G_{ib}^{H}\,,
\end{align*}
where we omit the $z$-dependence. Letting $\Lambda_o^{H^{(ik)}}\deq\displaystyle{\max_{a,b}}| G^{H^{(ik)}}_{ab}|$ and $\Lambda_o^{H}\deq\displaystyle{\max_{a,b}}| G^{H}_{ab}|$, we get
\begin{align*}
 \Lambda_o^{H^{(ik)}}\prec \Lambda_o^{H}+\frac{1}{q_t}\Lambda_o^{H}\Lambda_o^{H^{(ik)}}\,.	
\end{align*}
 By~\eqref{le moment condition} we have $|H_{ik}|\prec q_t^{-1}$ and by~\eqref{le EYYY1 2} we have $\Lambda_o^{H}\prec 1$, uniformly in $z\in\mathcal{D}$. It follows that $\Lambda_o^{H^{(ik)}}\prec \Lambda_o^{H}\prec 1$, uniformly in $z\in\mathcal{D}$, where we used~\eqref{le EYYY1 2}. Similarly, for $x\in\R$, we have
\begin{align*}
 G^{H^{(ik)}+xE^{[ik]}}_{ab}=G^{H^{(ik)}}_{ab}-x\big(G^{H^{(ik)}}E^{[ik]}G^{H^{(ik)}+xE^{[ik]}}\big)_{ab}\,,
\end{align*}
and we get
\begin{align}\label{le nilpferd}
 \sup_{|x|\le q_t^{-1/2}}\max_{a,b}|G^{H^{(ik)}+xE^{[ik]}}_{ab}|\prec \Lambda_o^{H^{(ik)}}\prec 1\,,
\end{align}
uniformly in $z\in\mathcal{D}$, where we used once more~\eqref{le EYYY1 2}.

Recall that $P$ is a polynomial of degree $4$ in $m$. Then $F_{ki}$ is a multivariate polynomial of degree $4(2D-1)+1$ in the Green function entries and the normalized trace $m$ whose number of member terms is bounded by $4^{2D-1}$. Hence $\partial_{{ik}}^{\ell}F_{ki}$ is a multivariate polynomial of degree $4(2D-1)+1+\ell$ whose number of member terms is roughly bounded by $4^{2D-1}\times(4(2D-1)+1+2l)^l$. Next, to control the individual monomials in~$\partial_{{ik}}^{\ell}F_{ki}$, we apply~\eqref{le nilpferd} to each factor of Green function entries (at most $4(2D-1)+1+\ell$ times). Thus, altogether we obtain
\begin{align}\label{le biranimi}
 \E\bigg[\sup_{|x|\le q_t^{-1/2}}|(\partial_{{ik}}^{\ell}F_{ki})(H^{(ik)}+xE^{[ik]})|\bigg]\le 4^{2D}(8D+\ell) N^{(8D+\ell)\epsilon'}\,,
\end{align}
for any small $\epsilon'>0$ and sufficiently large $N$. Choosing $\epsilon'=\epsilon/(2(8D+\ell))$ with get~\eqref{le first bound for error terms}.
\end{proof}
Recall that we set $\Ir= {(1+zm)} P(m)^{D-1}\overline{P(m)^D}$ in~\eqref{first time Ir}. To control the error term $\E\Omega_{\ell}(I)$ in~\eqref{the short guy}, we use the following result.
 
\begin{corollary}\label{le corollary for error terms omega}
Let $\E\Omega_\ell(I)$ be as in~\eqref{the short guy}. With the assumptions and notation of Lemma~\ref{le lemma first bound for error terms}, we have, for any (small) $\epsilon>0$,
 \begin{align}\label{final bound on the error term in the Stein}
 |\E \Omega_{\ell}\big( I\big)\big|\le N^\epsilon \left(\frac{1}{q_t}\right)^\ell,
 \end{align}
 uniformly in $z\in\mathcal{D}$, for $N$ sufficiently large.
In particular, the error $\E\Omega_{\ell}(I)$ is negligible for $\ell\ge 8D$.

\end{corollary}

\begin{proof}
First, fix a pair of indices $(k,i)$, $k\not=i$. Recall the definition of $F_{ik}$ in~\eqref{le F function}. Denoting $\E_{ik}$ the partial expectation with respect to $H_{ik}$, we have from Lemma~\ref{le stein lemma}, with $Q=q_t^{-1/2}$,
\begin{align}\label{le error term in general stein}\begin{split}
 |\E_{ik}\Omega_\ell(H_{ik}F_{ki})|&\le C_\ell \E_{ik}[ |H_{ik}|^{\ell+2}]\sup_{|x|\le q_t^{-1/2}}|\partial_{ik}^{\ell+1}F_{ki}(H^{(ik)}+xE^{[ik]})|\\ &\qquad+ C_\ell \E_{ik} [|H_{ik}|^{\ell+2} \lone(|H_{ik}|>q_t^{-1/2})]\sup_{x\in \R} |\partial_{ik}^{\ell+1}F_{ki}(H^{(ik)}+xE^{[ik]})| \,,
\end{split}\end{align}
with $C_\ell\le (C\ell)^\ell/\ell!$, for some numeral constant $C$. To control the full expectation of the first term on the right side, we use the moment assumption~\eqref{le moment condition} and Lemma~\ref{le lemma first bound for error terms} to conclude that, for any $\epsilon>0$,
\begin{align*}\begin{split}
  C_\ell\E \bigg[\E_{ik}[ |H_{ki}|^{\ell+2}]\sup_{|x|\le q_t^{-1/2}}\big|\partial_{ik}^{\ell+1}F_{ki}\big(H^{(ik)}+xE^{[ik]}\big)\big|\bigg]&\le C_\ell\frac{(C(\ell+2))^{c(\ell+2)}}{Nq_t^{\ell}}N^\epsilon\le \frac{N^{2\epsilon}}{Nq_t^{\ell}}\,,
\end{split}\end{align*}
for $N$ sufficiently large. To control the second term on the right side of~\eqref{le error term in general stein}, we use the deterministic bound $\|G(z)\|\le \eta^{-1}$ to conclude that
\begin{align*}
 \sup_{x\in\R}\big|\partial_{ik}^{\ell+1}F_{ki}\big(H^{(ik)}+xE^{[ik]}\big)\big|\le  4^{2D}(8D+\ell)\left
 (\frac{C}{\eta}\right)^{(8D+\ell)}, \qquad\qquad( z\in\C^+)\,;
\end{align*}
\cf the paragraph above~\eqref{le biranimi}. On the other hand, we have from H\"older's inequality and the moment assumptions in~\eqref{le moment condition} that, for any $D'\in\N$,
\begin{align*}
 \E_{ik} [|H_{ik}|^{\ell+2} \lone(|H_{ik}|>q_t^{-1/2})]\le \left(\frac{C}{q}\right)^{D'}\,,
\end{align*}
for $N$ sufficiently large. Using that $q\ge N^\phi$ by~\eqref{assumption H}, we hence obtain, for any $D'\in\N$,
\begin{align}\label{le biranimi 2}
 C_\ell \E_{ik} [|H_{ik}|^{\ell+2} \lone(|H_{ik}|>q_t^{-1/2})]\sup_{x\in \R} |\partial_{ik}^{\ell+1}F_{ki}(H^{(ik)}+xE^{[ik]})| \le\left(\frac{C}{q}\right)^{D'}\,,
\end{align}
uniformly on $\C^+$, for $N$ sufficiently large. 

Next, summing over $i$, $k$ and choosing $D'\ge \ell $ sufficiently large in~\eqref{le biranimi 2} we obtain, for any~$\epsilon>0$
\begin{align*}
 \bigg|\E\bigg[\Omega_\ell\Big( (1+ zm)P^{{D-1}}\ol{P^D}\Big)\bigg]\bigg|=\bigg|\E\bigg[\Omega_\ell\Big( \frac{1}{N}\sum_{i\not=k}H_{ik}F_{ki}\Big)\bigg]\bigg|\le \frac{N^{\epsilon}}{q_t^{\ell}}\,,
\end{align*}
uniformly on $\mathcal{D}$, for $N$ sufficiently large. This proves~\eqref{final bound on the error term in the Stein}.
\end{proof}

\begin{remark}\label{remark on other expansions}
 We will also consider slight generalizations of the cumulant expansion in~\eqref{ziwui}. Let $i,j,k\in\llbracket1,N\rrbracket$. Let $n\in \N_0$ and choose indices $a_1,\dots, a_{n},$ $ b_1,\ldots, b_n\in\llbracket1,N\rrbracket$. Let $D\in\N$ and choose $s_1,s_2,s_3,s_4\in\llbracket0,D\rrbracket$. Fix $z\in\mathcal{D}$. Define the function $F_{ki}$ by setting
 \begin{align}\label{le F function bis}
   F_{ki}(H)\deq  G_{ki}\prod_{l=1}^n G_{a_lb_l}P^{D-s_1} \ol{P^{D-s_2}} \left( P' \right)^{s_3} \left( \ol{P'} \right)^{s_4} \,.
 \end{align} 
 It is then straightforward to check that we have cumulant expansion 
 \begin{align}\label{le general general stein}
   \E \bigg[\frac{1}{N} \sum_{i \neq k} H_{ik} F_{ki}  \bigg] = \sum_{r=1}^\ell \frac{\kappa_t^{(r+1)}}{r!} \E \bigg[\frac{1}{N} \sum_{i \neq k} \partial_{ik}^r F_{ki} \bigg] +\E\Omega_{\ell}\bigg( \frac{1}{N} \sum_{i\not=k} H_{ik}F_{ki} \bigg)\,,
 \end{align} 
 where the error $\E\Omega_\ell(\cdot)$ satisfies the same bound as in~\eqref{final bound on the error term in the Stein}. This follows easily by extending Lemma~\ref{le lemma first bound for error terms} and Corollary~\ref{le corollary for error terms omega}.
\end{remark}

\subsection{Truncated cumulant expansion}
Armed with the estimates on $\E\Omega_\ell(\cdot)$ of the previous subsection, we now turn to the main terms on the right side of~\eqref{the short guy}. 
In the remainder of this section we derive the following result from which Lemma~\ref{lem:claim} follows directly. Recall the definition of $\Phiepsi$ in~\eqref{eq:phi}.
\begin{lemma}\label{summary expansions}
 Fix $D\ge 2$ and $\ell\ge 8D$. Let $\Ir_{r,s}$ be given by~\eqref{e1}. Then we have, for any (small) $\epsilon>0$,
\begin{gather}\label{le summary equation 1} \begin{aligned}
 w_{\Ir_{1,0}} \E[\Ir_{1,0}]&=-\E\big[m^2 P(m)^{D-1}\ol{P(m)^D}\big]+O(\Phiepsi)\,, & w_{\Ir_{2,0}}\E[\Ir_{2,0}]&=O(\Phiepsi)\,,\\
w_{\Ir_{3,0}}\E[\Ir_{3,0}]&=-\frac{s^{(4)}}{q_t^2}\E\big[ m^4 P(m)^{D-1}\ol{P(m)^D}\big]+O(\Phiepsi)\,, &w_{\Ir_{r,0}}\E[\Ir_{r,0}]&=O(\Phiepsi)\,,\quad\;\; (4\le r\le \ell)\,,
\end{aligned}\end{gather}
uniformly in $z\in\mathcal{D}$, for $N$ sufficiently large. Moreover, we have, for any (small) $\epsilon>0$,
\begin{align}\label{le summary equation 2}
w_{\Ir_{r,s}}|\E[\Ir_{r,s}]|\le \Phiepsi\,, \qquad (1\le s\le r\le \ell)\,,
\end{align}
uniformly in $z\in\mathcal{D}$, for $N$ sufficiently large. 
\end{lemma}
\begin{proof}[Proof of Lemma~\ref{lem:claim}]
 By the definition of $\Phiepsi$ in~\eqref{eq:phi} (with a sufficiently large (small) $\epsilon>0$), it suffices to show that $\E [ |P|^{2D}(z) ] \le\Phiepsi(z)$, for all $z\in\mathcal{D}$, for $N$ sufficiently large. Choosing~$\ell\ge 8D$, Corollary~\ref{le corollary for error terms omega} asserts that $\E \Omega_\ell(I)$ in~\eqref{the short guy} is negligible. By Lemma~\ref{summary expansions} the only non-negligible terms in the expansion of the first term on the right side of~\eqref{the short guy} are $w_{\Ir_{1,0}}\E\Ir_{1,0}$ and $w_{\Ir_{3,0}}\E\Ir_{3,0}$, yet these two terms cancel with the middle term on the right side of~\eqref{the short guy}, up to negligible terms. Thus the whole right side of~\eqref{the short guy} is negligible. This proves Lemma~\ref{lem:claim}. 
\end{proof}

We now choose an initial (small) $\epsilon>0$. Below we use the factor $N^\epsilon$ to absorb numerical constants in the estimates by allowing $\epsilon$ to increase by a tiny amount from line to line. We often drop $z$ from the notation; it is always understood that $z\in\mathcal{D}$ and all estimates are uniform on $\mathcal{D}$. The proof of Lemma~\ref{summary expansions} is done in the remaining Subsections~\ref{sub:e111}--\ref{sub:e113} where $\E\Ir_{r,s}$ are controlled.

\subsection{Estimate on $\Ir_{1, s}$} \label{sub:e111}

Starting from the definition of $\Ir_{1,0}$ in~\eqref{the IR}, a direct computation yields
\begin{align} \label{ea111}
\E\Ir_{1,0}&=\frac{\kappa_t^{(2)}}{N} \E \bigg[ \sum_{i_1 \neq i_2} \big( \partial_{i_1i_2} G_{i_2i_1} \big) P^{D-1} \ol{P^D} \bigg] = -\E \bigg[\frac{1}{N^2}  \sum_{i_1 \neq i_2} (G_{i_2i_2} G_{i_1i_1}+G_{i_1i_2}G_{i_2i_1}) P^{D-1} \ol{P^D} \bigg] \nonumber \\
&= -\E \bigg[ m^2 P^{D-1} \ol{P^D} \bigg] +  \E \bigg[\frac{1}{N^2} \sum_{i_1=1}^N (G_{i_1i_1})^2 P^{D-1} \ol{P^D} \bigg] - \E \bigg[ \frac{1}{N^2} \sum_{i_1 \neq i_2} (G_{i_2i_1})^2 P^{D-1} \ol{P^D} \bigg]\,.
 \end{align}
The middle term on the last line is negligible since
\begin{align*}
\bigg| \frac{1}{N^2} \E \bigg[ \sum_{i_1=1}^N (G_{i_1i_1})^2 P^{D-1} \ol{P^D} \bigg] \bigg| \leq \frac{N^{\epsilon}}{N} \E \Big[ P^{D-1} \ol{P^D} \Big]\,,
\end{align*}
where we used $|G_{i_1i_1}|\prec 1$, and so is the third term since
\begin{align*}
\frac{1}{N^2} \bigg| \E \bigg[ \sum_{i_1 \neq i_2} (G_{i_2i_1})^2 P^{D-1} \ol{P^D} \bigg] \bigg| \leq N^{\epsilon} \E \bigg[ \frac{\im m}{N\eta} |P|^{2D-1} \bigg]\,,
\end{align*}
where we used Lemma~\ref{lem:2 off-diagonal}. We thus obtain from~\eqref{ea111} that
\begin{align} \begin{split} \label{e111}
\Big|\Ir_{1,0} + \E \big[ m^2 P^{D-1} \ol{P^D}\, \big] \Big| \leq \Phiepsi\,,
\end{split} \end{align}
for $N$ sufficiently large. This proves the first estimate in~\eqref{le summary equation 1}. 

Consider next $\Ir_{1,1}$. Similar to~\eqref{expansion wigner 3}, we have
\begin{align*}\begin{split}
\E\Ir_{1,1}=\frac{1}{N^2}\sum_{i_1\not=i_2} \E \bigg[ G_{i_2i_1} \partial_{i_1i_2} ( P^{D-1} \ol{P^D} )\bigg] &= -\frac{2(D-1)}{N^3}\E\bigg[ \sum_{i_1\not=i_2}\sum_{i_3=1}^N G_{i_2i_1}  G_{i_2i_3}G_{i_3i_2}P'(m) P^{D-2} \ol{P^D}\bigg] \\ &\qquad- \frac{2D}{N^3}   \sum_{i_1\not=i_2}\sum_{i_3=1}^N \E\bigg[G_{i_2i_1} \ol{G_{i_3i_2}G_{i_2i_3}P'(m)}P^{D-1} \ol{P^{D-1}}\bigg]\,.
\end{split}\end{align*} 
Here the {\it fresh summation index} $i_3$ originated from $\partial_{i_1i_2}P(m)=P'(m)\frac{1}{N}{\sum_{i_3=1}^N}\partial_{i_1i_2}G_{i_3i_3}$. Note that we can add the terms with $i_1=i_2$ at the expense of a negligible error, so that
\begin{align}\begin{split}\label{le soul}
\E\Ir_{1,1} &= -2(D-1)\E\bigg[\frac{1}{N^3}\Tr G^3 P'(m) P^{D-2} \ol{P^D}\bigg] \\ &\qquad\qquad - 2D \,\E\bigg[\frac{1}{N^3} G (G^*)^2\ol{P'(m)}P^{D-1} \ol{P^{D-1}}\bigg]+O(\Phiepsi)\,. 
\end{split}\end{align}
In the remainder of this section, we will freely include or exclude negligible terms with coinciding indices. Using~\eqref{sum rule 1} and~\eqref{sum rule 2} we obtain from~\eqref{le soul} the estimate
\begin{align} \label{e1211}
|\E\Ir_{1,1}|=\bigg| \E \bigg[ \frac{1}{N^2} \sum_{i_1\not=i_2} G_{i_2i_1} \partial_{i_1i_2} ( P^{D-1} \ol{P^D} ) \bigg] \bigg| \leq (4D-2) \E \bigg[ \frac{\im m}{(N\eta)^2} |P'| |P|^{2D-2} \bigg]+\Phiepsi\,,
\end{align}
for $N$ sufficiently large. This proves~\eqref{le summary equation 2} for $r=s=1$.

\subsection{Estimate on $\Ir_{2,0}$} \label{sub:e112}

We start with a lemma that is used in the power counting arguments~below.
\begin{lemma} \label{lem:2 off-diagonal}
For any $i,k\in\llbracket1,N\rrbracket$,
\begin{align}\label{generalized Ward}
\frac{1}{N} \sum_{j=1}^N |G_{ij}(z)G_{jk}(z)| \prec \frac{\im m(z)}{N\eta}\,,\qquad\frac{1}{N}\sum_{j=1}^N|G_{ij}(z)|\prec \left(\frac{\im m(z)}{N\eta} \right)^{1/2}\,,\qquad\qquad (z\in\C^+)\,.
\end{align}
Moreover, for fixed $n\in \N$,
\begin{align}\label{generalized Ward 2}
 \frac{1}{N^n}\sum_{j_1,j_2,\ldots ,j_n=1}^N|G_{ij_1}(z)G_{j_1j_2}(z)G_{j_2j_3}(z)\cdots G_{j_nk}(z)|\prec\bigg(\frac{\im m(z)}{N\eta} \bigg)^{n/2}\,,\qquad\qquad (z\in\C^+)\,.
\end{align}
\end{lemma}

\begin{proof}
Let $\lambda_1^{H_t} \geq \lambda_2^{H_t} \geq \dots \geq \lambda_N^{H_t}$ be the eigenvalues of $H_t$, and let $\bsu_1,\ldots,\bsu_N$, $\bsu_{\alpha}\equiv \bsu_{\alpha}^{H_t}$, denoted the associated normalized eigenvectors. Then, by spectral decomposition, we get
\begin{align*} \begin{split}
\sum_{j=1}^N |G_{ij}|^2 = \sum_{j=1}^N \sum_{\alpha, \beta} \frac{\bsu_{\alpha}(i) \ol{\bsu_{\alpha}(j)}}{\lambda_{\alpha} -z} \frac{\ol{\bsu_{\beta}(i)} \bsu_{\beta}(j)}{\lambda_{\beta} - \ol z} = \sum_{\alpha, \beta} \frac{\bsu_{\alpha}(i) \langle \bsu_{\alpha}, \bsu_{\beta} \rangle \ol{\bsu_{\beta}(i)}}{(\lambda_{\alpha} -z)(\lambda_{\beta} - \ol z)} = \sum_{\alpha=1}^N \frac{|\bsu_{\alpha}(i)|^2}{|\lambda_{\alpha} -z|^2}\,.
\end{split} \end{align*}
Since the eigenvectors are delocalized by Proposition~\ref{cor: delocalization}, we find that
\begin{align*}
\frac{1}{N}\sum_{j=1}^N |G_{ij}|^2 \prec \frac{1}{N^2} \sum_{\alpha=1}^N \frac{1}{|\lambda_{\alpha} -z|^2} = \frac{\im m}{N\eta}\,.
\end{align*}
This proves the first inequality in~\eqref{generalized Ward} for $i=k$. The inequality for $i\not=k$, the second inequality in~\eqref{generalized Ward} and~\eqref{generalized Ward 2} then follow directly from Schwarz inequality.
\end{proof}

Recalling the definition of $\Ir_{r,s}$  in~\eqref{e1} we have
\begin{align*}
\Ir_{2,0}\deq N\kappa_t^{(3)}\frac{1}{N^2} \sum_{i_1 \neq i_2} \big( \partial_{i_1i_2}^{2} G_{i_2i_1} \big)  P^{D-1} \ol{P^D}  \,.
 \end{align*}
We then notice that $\Ir_{2,0}$ contains terms with one or three off-diagonal Green function entries $G_{i_1i_2}$. We split accordingly
\begin{align}\label{le split Ir2}
 w_{\Ir_{2,0}}\Ir_{2,0}=w_{\Ir_{2,0}^{(1)}}\Ir_{2,0}^{(1)}+w_{\Ir_{2,0}^{(3)}}\Ir_{2,0}^{(3)}\,,
\end{align}
where $\Ir_{2,0}^{(1)}$ contains all terms with one off-diagonal Green function entries (and, necessarily, two diagonal Green function entries) and where $\Ir_{2,0}^{(3)}$ contains all terms with three off-diagonal Green function entries (and zero diagonal Green function entries), and $w_{\Ir_{2,0}^{(1)}}$, $w_{\Ir_{2,0}^{(3)}} $ denote the respective weights. Explicitly,
\begin{align}\begin{split}\label{definition of the IR21}
 \E\Ir_{2,0}^{(1)}&= {N\kappa_t^{(3)}}\E \bigg[\frac{1}{N^2} \sum_{i_1 \neq i_2} G_{i_2i_1} G_{i_2i_2} G_{i_1i_1} P^{D-1} \ol{P^D} \bigg]\,,\\
 \E\Ir_{2,0}^{(3)}&= {N\kappa_t^{(3)}}\E \bigg[\frac{1}{N^2} \sum_{i_1 \neq i_2} (G_{i_2i_1})^3 P^{D-1} \ol{P^D} \bigg]\,,
\end{split}\end{align}
and $w_{\Ir_{2,0}}=1$, $w_{\Ir_{2,0}^{(1)}}=3$, $w_{\Ir_{2,0}^{(3)}}=1$.

We first note that $\Ir_{2,0}^{(3)}$ satisfies, for $N$ sufficiently large,
\begin{align} \label{e1123}
|\E\Ir_{2,0}^{(3)}|& \leq\frac{N^\epsilon s^{(3)}}{q_t }  \E \bigg[\frac{1}{N^2} \sum_{i_1 \neq i_2} |G_{i_1i_2}|^2 |P|^{2D-1} \bigg] \leq \frac{N^{\epsilon}}{ q_t} \E \bigg[ \frac{\im m}{N\eta} |P|^{2D-1} \bigg] \leq \Phiepsi\,.
\end{align}
\begin{remark}\label{remark power counting}[Power counting I] Consider the terms $I_{r,0}$, $r\ge 1$. For $n\ge 1$, we then split
\begin{align}\label{smino}
 w_{\Ir_{2n}}\Ir_{2n,0}=\sum_{l=0}^n  w_{\Ir_{2n,0}^{(2l+1)}}\Ir_{2n,0}^{(2l+1)}\,,\qquad
 w_{\Ir_{2n-1}}\Ir_{2n-1,0}=\sum_{l=0}^n  w_{\Ir_{2n-1,0}^{(2l)}}\Ir_{2n-1,0}^{(2l)}\,,
\end{align}
according to the parity of $r$. For example, for $r=1$, $\Ir_{1,0}=\Ir_{1,0}^{(0)}+\Ir_{1,0}^{(2)}$, with
\begin{align*}
 \E\Ir_{1,0}^{(0)}=  -\E \bigg[\frac{1}{N^2}  \sum_{i \neq k} G_{kk} G_{ii} P^{D-1} \ol{P^D} \bigg]\,,\qquad \E\Ir_{1,0}^{(2)}=-\E \bigg[\frac{1}{N^2}  \sum_{i \neq k} G_{ik}G_{ki} P^{D-1} \ol{P^D} \bigg]\,;
\end{align*}
\cf~\eqref{ea111}. Now, using a simple power counting, we bound the summands in~\eqref{smino} as follows. First, we note that each term in $\Ir_{r,0}$ contains a factor of $q_t^{(r-2)_+}$. Second, for  $\E\Ir_{2n,0}^{(2l+1)}$ and $\E\Ir_{2n-1,0}^{(2l)}$, with $n\ge 1$, $l\ge 1$, we can by Lemma~\ref{lem:2 off-diagonal} extract one factor of $\frac{\im m}{N\eta}$ (other Green function entries are bounded using $|G_{ik}|\prec 1$). Thus, for $n\ge 1$, $l\ge 1$,
\begin{align}\label{trivially estimated higher order terms}
 |\E\Ir_{2n,0}^{(2l+1)}|\le\frac{ N^{\epsilon}}{q_t^{2n-2}}\E\bigg[\bigg(\frac{\im m}{N\eta}\bigg)|P|^{2D-1}\bigg]\,,\qquad\quad |\E\Ir_{2n-1,0}^{(2l)}|\le \frac{N^{\epsilon}}{q_t^{(2n-3)_+}}\E\bigg[\bigg(\frac{\im m}{N\eta}\bigg)|P|^{2D-1}\bigg]\,, 
\end{align}
for $N$ sufficiently large, and we conclude that all these terms are negligible.

\end{remark}

We next consider $\E\Ir_{2,0}^{(1)}$ that is not covered by~\eqref{trivially estimated higher order terms}. Using $|G_{ii}|\prec 1$ and Lemma~\ref{lem:2 off-diagonal} we~get 
\begin{align} \begin{split}\label{a factor of q to be gained}
|\E\Ir_{2,0}^{(1)}|\leq \frac{N^{\epsilon}s^{(3)}}{ q_t} \E \bigg[\frac{1}{N^2} \sum_{i_1 \neq i_2} |G_{i_1i_2}| |P|^{2D-1} \bigg]\le \frac{N^\epsilon}{ q_t}\E\bigg[\bigg(\frac{\im m}{N\eta}\bigg)^{1/2}|P|^{2D-1}\bigg]\,,
\end{split} \end{align}
for $N$ sufficiently large. Yet, this bound is not negligible. We need to gain an additional factor of $q_t^{-1}$ with which it will become negligible. We have the following result.

\begin{lemma}\label{le first round lemma q}
For any (small) $\epsilon>0$, we have, for all $z\in\mathcal{D}$,
\begin{align}\label{a factor q gained from a round}
 |\E\Ir_{2,0}^{(1)}|\le \frac{N^\epsilon}{q_t^2}\E\bigg[\bigg(\frac{\im m}{N\eta}\bigg)^{1/2}|P|^{2D-1}\bigg]+\Phiepsi\le N^\epsilon\E \bigg[ \bigg(q_t^{-4}+ \frac{\im m}{N\eta}   \bigg) \big| P(m) \big|^{2D-1} \bigg] +\Phiepsi\,,
\end{align}
for $N$ sufficiently large. In particular, the term $\E\Ir_{2,0}$ is negligible.
\end{lemma}

\begin{proof}
Fix a (small) $\epsilon>0$. Recalling~\eqref{definition of the IR21}, we have
\begin{align}\label{preumathched begins}
 \E\Ir_{2,0}^{(1)}={N\kappa_t^{(3)}} \E \bigg[\frac{1}{N^2} \sum_{i_1\not=i_2} G_{i_2i_1}G_{i_1i_1}G_{i_2i_2} P^{D-1} \ol{P^D} \bigg] \,.
\end{align}
The key feature here is that the Green function entries are $ G_{i_2i_1}G_{i_1i_1}G_{i_2i_2}$, where at least one index, say $i_2$, appears an odd number of times. (This index $i_2$ can be considered as ``unmatched''.) Using the resolvent formula~\eqref{expansion indentity} we expand in the unmatched index $i_2$ to get
\begin{align} \label{unmatched begins}\begin{split}
z \E\Ir_{2,0}^{(1)}&={N\kappa_t^{(3)}}\E \bigg[\frac{1}{N^2}\sum_{i_1\not=i_2 \not=i_3} H_{i_2i_3} G_{i_3i_1} G_{i_2i_2} G_{i_1i_1} P^{D-1} \ol{P^D}\bigg]\,.
\end{split}\end{align}
We now proceed in a similar way as in Remark~\ref{remark on entrywise local law of GOE} where we estimated $|G^W_{i_1i_2}|$, $i_1\not=i_2$, for the GOE. Applying the cumulant expansion to the right side of~\eqref{unmatched begins}, we will show that the leading term is $-\E[ m\Ir_{2,0}^{(1)}]$. Then, upon substituting $m(z)$ by the deterministic quantity~$\wt m(z)$ and showing that all other terms in the cumulant expansion of the right side of~\eqref{unmatched begins} are negligible, we will get that
\begin{align}\label{for second round needed}
 |z+\wt m(z)|\,\big| \E\Ir_{2,0}^{(1)}\big|\le  \frac{N^\epsilon}{q_t^2}\E\bigg[\bigg(\frac{\im m}{N\eta}\bigg)^{1/2}|P|^{2D-1}\bigg]+\Phiepsi\le 2\Phiepsi\,,
\end{align}
for $N$ sufficiently large. Since $|z+\wt m(z) |\ge 1/6$ uniformly on $\mathcal{E}\supset\mathcal{D}$, as shown in Remark~\ref{rem:stability of wt m}, the lemma directly follows. The main efforts in the proof go into showing that the sub-leading terms in the cumulant expansion of the right side of~\eqref{unmatched begins} are indeed negligible.

For simplicity we abbreviate $\wh I\equiv \Ir_{2,0}^{(1)}$. Then using Lemma~\ref{le stein lemma} and Remark~\ref{remark on other expansions}, we have, for arbitrary $\ell'\in\N$, the cumulant expansion
\begin{align}\label{le second round cumulant expansion}
  z \E\Ir_{2,0}^{(1)}=z \E\wh \Ir=\sum_{r'=1}^{\ell'}\sum_{s'=0}^{r'} w_{\wh\Ir_{r',s'}}\E \wh\Ir_{r',s'}+O\left(\frac{N^\epsilon}{q_t^{\ell'}}\right)\,,
\end{align}
with
\begin{align}\label{def wtirs}
 {\wh\Ir_{r',s'}\deq {N\kappa_t^{(r'+1)}} {N\kappa_t^{(3)}} \frac{1}{N^3} \sum_{i_1\not=i_2\not=i_3} \big( \partial_{i_2i_3}^{r'-s'}( G_{i_3i_1}G_{i_2i_2}  G_{i_1i_1})\big) \big( \partial_{i_2i_3}^{s'}\big( P^{D-1} \ol{P^D} \big) \big)}
\end{align}
and $w_{\wh\Ir_{r',s'}}=\frac{1}{r'!}\binom{r'}{s'}$. Here, we used Corollary~\ref{le corollary for error terms omega} to truncate the series in~\eqref{le second round cumulant expansion} at order $\ell'$. Choosing $\ell'\ge 8D$ the remainder is indeed negligible.

We first focus on $\wh\Ir_{r',0}$. For $r'=1$, we compute
\begin{align}\begin{split}\label{weissenstein}
\E\wh\Ir_{1,0}&=-\frac{\nm^{(3)}}{q_t} \E \bigg[\frac{1}{N^3} \sum_{i_1\not=i_2\not=i_3}G_{i_2i_1}G_{i_3i_3}  G_{i_2i_2} G_{i_1i_1} P^{D-1} \ol{P^D}\bigg]\\ &\qquad- 3\frac{\nm^{(3)}}{q_t} \E \bigg[\frac{1}{N^3} \sum_{i_1\not=i_2\not=i_3} G_{i_2i_3}G_{i_3i_1} G_{i_2i_2} G_{i_1i_1} P^{D-1} \ol{P^D}\bigg]\\
&\qquad-2\frac{\nm^{(3)}}{q_t} \E \bigg[\frac{1}{N^3} \sum_{i_1\not=i_2\not=i_3} G_{i_1i_2}G_{i_2i_3}G_{i_3i_1}G_{i_2i_2}  P^{D-1} \ol{P^D}\bigg]\\
&=:\E\wh\Ir_{1,0}^{(1)}+3\E\wh\Ir_{1,0}^{(2)}+2\E\wh\Ir_{1,0}^{(3)}\,,
 \end{split}\end{align}
where we organize the terms according to the off-diagonal Green functions entries. By Lemma~\ref{lem:2 off-diagonal},
\begin{align}\label{weissenstein 2}
 |\E\wh\Ir_{1,0}^{(2)}|\le \frac{N^{\epsilon}}{q_t}\E\bigg[\frac{\im m}{N\eta}|P|^{2D-1}\bigg]\le \Phiepsi\,,\qquad |\E\wh\Ir_{1,0}^{(3)}|\le   \frac{N^{\epsilon}}{q_t}\E\bigg[\bigg(\frac{\im m}{N\eta}\bigg)^{3/2}|P|^{2D-1}\bigg]\le \Phiepsi\,.
\end{align}
Recall $\wt m \equiv \wt m_t(z)$ defined in Proposition \ref{prop:local}. We rewrite $\wh\Ir_{1,0}^{(1)}$ with $\wt m$ as
\begin{align} \begin{split} \label{u1}
 z\E\wh\Ir_{1,0}^{(1)}&=- \E \bigg[ \frac{1}{N^2} \sum_{i_1\not=i_2} \wt m G_{i_2i_1} G_{i_2i_2} G_{i_1i_1} P^{D-1} \ol{P^D} \bigg] \\&\qquad\qquad- \E \bigg[ \frac{1}{N^2} \sum_{i_1\not=i_2} (m-\wt m) G_{i_2i_1} G_{i_2i_2} G_{i_1i_1} P^{D-1} \ol{P^D} \bigg]+O(\Phiepsi)\,.
\end{split} \end{align}
By Schwarz inequality and the high probability bounds $|G_{kk}|, |G_{ii}| \leq N^{\epsilon/8}$, for $N$ sufficiently large, the second term in \eqref{u1} is bounded as
\begin{align} \begin{split}\label{u1bis}
& \bigg|\E \bigg[ \frac{1}{N^2} \sum_{i_1\not=i_2} (m-\wt m) G_{i_2i_1} G_{i_2i_2} G_{i_1i_1} P^{D-1} \ol{P^D} \bigg]\bigg| \leq N^{\epsilon/4} \E \bigg[ \frac{1}{N^2} \sum_{i_1\not= i_2} |m-\wt m| |G_{i_1i_2}| |P|^{2D-1} \bigg] \\
&\qquad\leq N^{-\epsilon/4} \E \bigg[ \frac{1}{N^2} \sum_{i_1\not=i_2} |m-\wt m|^2 |P|^{2D-1} \bigg] + N^{3\epsilon/4} \E \bigg[ \frac{1}{N^2} \sum_{i_1\not=i_2} |G_{i_1i_2}|^2 |P|^{2D-1} \bigg] \\
&\qquad= N^{-\epsilon/4} \E \bigg[  |m-\wt m|^2 |P|^{2D-1} \bigg] + N^{3\epsilon/4} \E \bigg[ \frac{\im m}{N\eta} |P|^{2D-1} \bigg]\,.
\end{split} \end{align}
We thus get from~\eqref{weissenstein},~\eqref{weissenstein 2},~\eqref{u1} and~\eqref{u1bis} that
\begin{align}\begin{split}\label{hahama}
z\E\wh \Ir_{1,0}=-\wt m  \E \bigg[ \frac{1}{N^2} \sum_{i_1\not=i_2}G_{i_2i_1} G_{i_2i_2} G_{i_1i_1} P^{D-1} \ol{P^D} \bigg] +O(\Phiepsi)= -\E\wt m\Ir_{2,0}^{(1)}+O(\Phiepsi)\,,
\end{split}\end{align}
where we used~\eqref{preumathched begins}. 
We remark that in the expansion of $\E\wh \Ir=\E \Ir_{2,0}^{(1)}$ the only term with one off-diagonal entry is $\E\wh\Ir_{2,0}^{(1)}$. All the other terms contain at least two off-diagonal entries.

\begin{remark}\label{remark power counting II}[Power counting II]
Comparing~\eqref{the IR} and~\eqref{le second round cumulant expansion}, we have $\wh\Ir_{r',s'}=(\Ir_{2,0}^{(1)})_{r',s'}$. Consider now the terms with $s'=0$. As in~\eqref{smino} we organize the terms according to the number of off-diagonal Green function entries.  For $r'\ge 2$,
 \begin{align}\label{smino2}
 w_{\wh\Ir_{r',0}}\wh\Ir_{r',0}&=\sum_{l=0}^n  w_{\wh\Ir_{r',0}^{(l+1)}}\wh\Ir_{r',0}^{(l+1)}=\sum_{l=0}^n  w_{\wh\Ir_{r',0}^{(l+1)}} (\Ir_{2,0}^{(1)})_{r',0}^{(l+1)} \,.
\end{align}
A simple power counting as in Remark~\ref{remark power counting} then directly yields 
\begin{align}\label{trivially estimated higher order terms 2}
|\E\wh\Ir_{r',0}^{(1)}|\le\frac{ N^{\epsilon}}{q_t^{r'}}\E\bigg[\bigg(\frac{\im m}{N\eta}\bigg)^{1/2}|P|^{2D-1}\bigg]\,,\qquad|\E\wh\Ir_{r',0}^{(l+1)}|\le\frac{ N^{\epsilon}}{q_t^{r'}}\E\bigg[\bigg(\frac{\im m}{N\eta}\bigg)|P|^{2D-1}\bigg]\,, \qquad(l\ge 1)\,,
\end{align}
for $N$ sufficiently large. Here, we used that each term contains a factor $\kappa_t^{(3)}\kappa_t^{(r'+1)}\le CN^{-2}q_t^{-r'}$. We conclude that all terms in~\eqref{trivially estimated higher order terms 2} with $r'\ge 2$ are negligible, yet we remark that $|\E\wh\Ir_{2,0}^{(1)}|$ is the leading error term in $|\E \Ir_{2,0}^{(1)}|$, which is explicitly listed on the right side of~\eqref{a factor q gained from a round}.
\end{remark}

\begin{remark}\label{remark power counting III}[Power counting III]
Consider the terms $\wh\Ir_{r',s'}$, with $1\le s'\le r'$. For $s'=1$, note that $\partial_{i_2i_3} \big(P^{D-1} \ol{P^D}\big)$ contains two off-diagonal Green function entries. Explicitly,
\begin{align*}\begin{split}
 \wh\Ir_{r',1}= & -2(D-1)\frac{N\kappa_t^{(r'+1)}N\kappa_t^{(3)}} {{N^3}} \sum_{i_1\not=i_2\not=i_3} \big(\partial_{i_2i_3}^{r'-1}(G_{i_3i_1}G_{i_2i_2}G_{i_1i_1}) \big)\bigg({\frac{1}{N}}\sum_{i_4=1}^N G_{i_4i_2}G_{i_3i_4}\bigg)P'P^{D-2}\ol{P^D} \\ &-2D\frac{N\kappa_t^{(r'+1)}N\kappa_t^{(3)}} {{N^3}} \sum_{i_1\not=i_2\not=i_3} \big(\partial_{i_2i_3}^{r'-1}(G_{i_3i_1}G_{i_2i_2}G_{i_1i_1}) \big)\bigg({\frac{1}{N}}\sum_{i_4=1}^N \ol{G_{i_4i_2}G_{i_3i_4}}\bigg)\ol{P'} P^{D-1} \ol{P^{D-1}} \,,
\end{split}\end{align*}
where the fresh summation index $i_4$ is generated from $\partial_{i_2i_3} P$. Using Lemma~\ref{lem:2 off-diagonal} we get, for $r'\ge 1$,
\begin{align}\begin{split}\label{siminibo}
|\E\wh\Ir_{r',1}|&\le \frac{N^\epsilon}{q_t^{r'}}\E\bigg[\bigg(\frac{\im m}{N\eta} \bigg)^{3/2}|P'||P|^{2D-2}+\bigg(\frac{\im m}{N\eta} \bigg)^{3/2}|P'||P|^{2D-2} \bigg]\le 2 \Phiepsi\,,
\end{split}\end{align}
for $N$ sufficiently large, where we used that $\partial_{i_2i_3}^{r'-1}(G_{i_3i_1}G_{i_2i_2}G_{i_1i_1})$, $r'\ge 1$, contains at least one off-diagonal Green function entry.

For $2\le s'\le r'$, we first note that, for $N$ sufficiently large,
\begin{align}\begin{split}\label{monday morning}
|\E\wh\Ir_{r',s'}|&\le\frac{N^\epsilon}{q_t^{r'}}\bigg|\E \bigg[\frac{1}{N^3} \sum_{i_1\not=i_2\not=i_3} \big( \partial_{i_2i_3}^{r'-s'} G_{i_3i_1}G_{i_2i_2}G_{i_1i_1} \big) \big( \partial_{i_2i_3}^{s'}\big( P^{D-1} \ol{P^D} \big) \big)\bigg]\bigg|\\
&\le\frac{N^\epsilon}{q_t^{r'}}\E \bigg[\bigg(\frac{\im m}{N\eta}\bigg)^{1/2}\frac{1}{N^2} \sum_{i_2\not=i_3}  \big| \partial_{i_2i_3}^{s'}\big( P^{D-1} \ol{P^D} \big) \big|\bigg]\,.
\end{split}\end{align}
Next, since $s'\ge 2$, the partial derivative in $\partial_{i_2i_3}^{s'} \big( P^{D-1} \ol{P^D}\big)$ acts on $P$ and $\ol{P}$ (and on their derivatives) more than once. For example, for $s'=2$,
\begin{align}\begin{split}\nonumber
\partial_{i_1i_2}^2 P^{D-1} &=\frac{4(D-1)(D-2)}{N} \bigg(\sum_{i_3=1}^N  G_{i_3i_2}G_{i_1i_3} \bigg)^2(P')^2 P^{D-3}\\ &\qquad + \frac{2(D-1)}{N} \bigg( \sum_{i_3=1}^N  G_{i_2i_3}G_{i_3i_1} \bigg)^2 P'' P^{D-2}-\frac{2(D-1)}{N} \sum_{i_3=1}^N \partial_{i_1i_2} \big(  G_{i_3i_2}G_{i_1i_3} \big) P' P^{D-2}\,,
\end{split}\end{align}
where $\partial_{i_1i_2}$ acted twice on $P$, respectively $P'$, to produce the first two terms. More generally, for $s'\ge 2$, consider a resulting term containing
\begin{align}\label{fresh indeces term}
P^{D-s_1'} \ol{P^{D-s_2'}} \left( P' \right)^{s_3'} \left( \ol{P'} \right)^{s_4'} \left( P'' \right)^{s_5'}\left( \ol{P''} \right)^{s_6'}\left( P''' \right)^{s_7'}\left( \ol{P'''} \right)^{s_8'}\,, 
\end{align}
with  $1\le s_1' \le D$,  $0\le s_2' \le D$ and $\sum_{n=1}^8s_n'\le s'$. Since $P^{(4)}$ is constant we did not list it. We see that such a term above was generated from $P^{D-1}\overline{P^D}$ by letting the partial derivative $\partial_{i_2i_3}$ act $s_1'-1$-times on~$P$ and~$s_2'$-times on~$\ol P$, which implies that $ s_1'-1\ge s_3'$ and $s_2'\ge s_4'$. If $ s_1'-1>s_3'$, then $\partial_{i_2i_3}$ acted on the derivatives of $P, \ol{P}$ directly $(s_1'-1-s_3')$-times, and a similar argument holds for~$\ol{P'}$. Whenever $\partial_{i_2i_3}$ acted on $P$, $\ol P$ or their derivatives, it generated a term $ 2N^{-1} \sum_{i_l} G_{i_2i_l}G_{i_li_3}$, with $i_l$, $l\ge 3$, a fresh summation index. For each fresh summation index we apply Lemma~\ref{lem:2 off-diagonal} to gain a factor $\frac{\im m}{N\eta}$. The total number of fresh summation indices in a term corresponding to~\eqref{fresh indeces term} is
\begin{align*}
s_1'- + s_2' +(s_1 '- 1- s_3') + (s_2' - s_4')& = 2s_1' + 2s_2' - s_3' - s_4'= 2\tilde s_0-\tilde s-2\,,
\end{align*}
with $\tilde s_0\deq s_1'+s_2'$ and $\tilde s\deq s_3+s_4$ we note this number does not decrease when $\partial_{i_2i_3}$ acts on off-diagonal Green functions entries later. Thus, from~\eqref{monday morning} we conclude, upon using $|G_{i_1i_2}|,|P''(m)|, |P'''(m)|$, $|P^{(4)}(m)| \prec 1$  that, for $2\le s'\le r'$,
\begin{align}\label{monday morning 2}\begin{split}
 |\E\wh\Ir_{r',s'}|&\le\frac{N^\epsilon}{q_t^{r'}}\E \bigg[\bigg(\frac{\im m}{N\eta}\bigg)^{1/2}\frac{1}{N^2} \sum_{i_2\not=i_3}  \big| \partial_{i_2i_3}^{s'}\big( P^{D-1} \ol{P^D} \big) \big|\bigg]\\
 &\le\frac{N^{2\epsilon}}{q_t^{r'}}\sum_{\tilde s_0=2}^{2D}\sum_{\tilde s=1}^{\tilde s_0-2}\E\ \bigg[\bigg(\frac{\im m}{N\eta}\bigg)^{1/2+2\tilde s_0-\tilde s-2}|P' |^{\tilde s }|P|^{2D-\tilde s_0}\bigg]\\
 &\qquad\qquad+\frac{N^{2\epsilon}}{q_t^{r'}}\sum_{\tilde s_0=2}^{2D}\E\ \bigg[\bigg(\frac{\im m}{N\eta}\bigg)^{1/2+2\tilde s_0-1}|P' |^{\tilde s_0-1 }|P|^{2D-\tilde s_0}\bigg]\,,
\end{split}\end{align}
for $N$ sufficiently large. Here the last term on the right corresponds to $\wt s=\wt s_0-1$. Thus, we conclude form~\eqref{monday morning 2} and the definition of $\Phiepsi$ in~\eqref{eq:phi} that $\E[\wh\Ir_{r',s'}]$, $2\le s'\le r'$, is negligible.

To sum up, we have established that all terms $\E[\wt\Ir_{r',s'}]$ with $1\le s'\le r'$ are negligible.
\end{remark}

From~\eqref{preumathched begins},~\eqref{le second round cumulant expansion},~\eqref{hahama},~\eqref{trivially estimated higher order terms 2},~\eqref{siminibo} and~\eqref{monday morning 2} we find that
\begin{align*} \begin{split}
&|z+\wt m|\,| \E \Ir_{2,0}^{(1)}|\le\frac{N^\epsilon}{q_t^2}\E\bigg[\bigg(\frac{\im m}{N\eta}\bigg)^{1/2}|P|^{2D-1}\bigg]+\Phiepsi\,,
\end{split} \end{align*}
for $N$ sufficiently large. Since $(z+ \wt m)$ is deterministic and $|z+\wt m| > 1/6$, as we showed in Remark~\ref{rem:stability of wt m}, we obtain that $|\E\Ir_{2,0}^{(1)}|\le \Phiepsi$. This concludes the proof of~\eqref{a factor q gained from a round}. 
\end{proof}
Summarizing, we showed in~\eqref{e1123} and~\eqref{a factor q gained from a round} that
\begin{align}\label{final bound Ir_20}
|\E\Ir_{2,0}|\le \Phiepsi\,, 
\end{align}
for $N$ sufficiently large, \ie all terms in $\E\Ir_{2,0}$ are negligible and the second estimate in~\eqref{le summary equation 1} is proved.

\subsection{Estimate on $\Ir_{r,0}$, $r \geq 4$} \label{sub:e114}
For $r \geq 5$ we use the bounds $|G_{i_1i_1}|,|G_{i_1i_2}|\prec 1$ to get
\begin{align} \begin{split}\label{e115}
|\E\Ir_{r,0}|&=\bigg|  N\kappa_t^{(r+1)} \E \bigg[ \frac{1}{N^2}\sum_{i_1 \neq i_2} \big( \partial_{i_1i_2}^r G_{i_2i_1} \big) P^{D-1} \ol{P^D} \bigg] \bigg|\\&\ \leq \frac{N^{\epsilon}}{q_t^4}\E \bigg[ \frac{1}{N^2}\sum_{i_1 \neq i_2} |P|^{2D-1} \bigg] \leq \frac{N^{\epsilon}}{q_t^4} \E [|P|^{2D-1}] \leq \Phiepsi\,,
\end{split}\end{align}
for $N$ sufficiently large. For $r=4$, $\partial_{i_1i_2}^r G_{i_2i_1}$ contains at least one off-diagonal term $G_{i_1i_2}$. Thus
\begin{align} \begin{split} \label{e114}
\bigg|  N\kappa_t^{(5)} \E \bigg[\frac{1}{N^2} \sum_{i_1 \neq i_2} \big( \partial_{i_1i_2}^4 G_{ki} \big) P^{D-1} \ol{P^D} \bigg] \bigg| &\leq \frac{N^{\epsilon}}{ q_t^{3}} \E \bigg[ \frac{1}{N^2}\sum_{i_1 \neq i_2} |G_{ik}| |P|^{2D-1} \bigg] \\
&\le  \frac{N^\epsilon}{ q_t^{3}}\E\bigg[ \bigg(\frac{\im m}{N\eta}\bigg)^{1/2} |P|^{2D-1} \bigg]\le \Phiepsi\,,
\end{split} \end{align}
for $N$ sufficiently large, where we used Lemma~\ref{lem:2 off-diagonal} to get the last line. We conclude that all terms $\E\Ir_{r,0}$ with $r\geq 4$ are negligible. This proves the fourth estimate in~\eqref{le summary equation 1}.

\subsection{Estimate on $\Ir_{r,s}$, $r\ge 2$, $s\ge 1$} \label{sub:e12}
For $r\ge 2$ and $s=1$, we have
\begin{align*}
 \E\Ir_{r,1}= N\kappa_t^{(r+1)}\E\bigg[ \frac{1}{N^2} \sum_{i_1\not=i_2}(\partial_{i_1i_2}^{r-1}G_{i_2i_1}) \partial_{i_1i_2} ( P^{D-1} \ol{P^D} ) \bigg] \,.
\end{align*}
Note that each term in $\E\Ir_{r,1}$, $r\ge 2$, contains at least two off-diagonal Green function entries. For the terms with at least three off-diagonal Green function entries, we use the bound $| G_{i_1i_2}|, |G_{i_1i_1}| \prec 1$ and
\begin{align} \begin{split} \label{e12r1o}
&N\kappa_t^{(r+1)}\E \bigg[ \frac{1}{N^3} \sum_{i_1, i_2,i_3}|G_{i_2i_1}G_{i_1i_3} G_{i_3i_2}| |P'| |P|^{2D-2} \bigg] 
\leq N^{\epsilon}\frac{ s^{(r+1)}}{q_t}\E \bigg[ \bigg( \frac{\im m}{N\eta} \bigg)^{3/2} |P'| |P|^{2D-2} \bigg] \\
&\qquad\qquad \leq N^{\epsilon}s^{(r+1)} \E \bigg[ \sqrt{\im m} \bigg( \frac{\im m}{N\eta} \bigg) \bigg( \frac{1}{N\eta} + q_t^{-2} \bigg) |P'| |P|^{2D-2} \bigg]\,,
\end{split} \end{align}
for $N$ sufficiently large, where we used Lemma~\ref{lem:2 off-diagonal}. Note that the right side is negligible since $\im m \prec 1$. 

Denoting the terms with two off-diagonal Green function entries in $\E\Ir_{r,1}$ by $\E\Ir_{r,1}^{(2)}$, we have
\begin{align}\label{le Ir 112}\begin{split}
\E\Ir_{r,1}^{(2)}&=N\kappa^{(r+1)}\E\bigg[\frac{2(D-1)}{N^2} \sum_{i_1\neq i_2} G_{i_2i_2}^{r/2} G_{i_1i_1}^{r/2} \Big(\frac{1}{N}\sum_{i_3=1}^N  G_{i_2i_3}G_{i_3i_1} \Big)P' P^{D-2} \ol{P^D}\bigg]\\&\qquad +N\kappa^{(r+1)}\E\bigg[\frac{2D}{N^2} \sum_{i_1\neq i_2} G_{i_2i_2}^{r/2} G_{i_1i_1}^{r/2}\Big(\frac{1}{N} \sum_{i_3=1}^N  \ol{ G_{i_2i_3}G_{i_3i_1}}\Big) \ol{P'} P^{D-1} \ol{P^{D-1}} \bigg]\,,
\end{split}\end{align}
where $i_3$ is a fresh summation index and where we noted that $r$ is necessarily even in this case.
Lemma~\ref{lem:2 off-diagonal} then give us the upper bound
$$
\big|\E\Ir_{r,1}^{(2)}\big|\le \frac{N^{\epsilon}}{ q_t^{r-1}} \E \bigg[ \frac{\im m}{N\eta} |P'| |P|^{2D-2} \bigg]\,,
$$ 
which is negligible for $r> 2$. However, for $r=2$, we need to gain an additional factor~$q_t^{-1}$. This can be done as in the proof of Lemma~\ref{le first round lemma q} by considering the off-diagonal entries $G_{i_2i_3}G_{i_3i_1} $, generated from $\partial_{i_1i_2} P(m)$, since the index $i_2$ appears an odd number of times.
\begin{lemma}
For any (small) $\epsilon>0$, we have
 \begin{align}\label{the sunday lemmal}
 | \E\Ir_{2,1}^{(2)}|\le \frac{N^{\epsilon}}{ q_t^{2}} \E \bigg[ \frac{\im m}{N\eta} |P'| |P|^{2D-2} \bigg]+\Phiepsi\,,
 \end{align}
uniformly on $\caD$, for $N$ sufficiently large. In particular, the term $\E\Ir_{2,1}$ is negligible.
\end{lemma}
\begin{proof}
We start with the first term on the right side of~\eqref{le Ir 112}. Using~\eqref{expansion indentity}, we write
\begin{align*}\begin{split}
 z N\kappa_t^{(3)}&\E\bigg[\frac{1}{N^3} \sum_{i_1\neq i_2\neq i_3} G_{i_2i_3}G_{i_2i_2}G_{i_1i_1} G_{i_1i_3}P' P^{D-2} \ol{P^D}\bigg]\\ &= N\kappa_t^{(3)}\E\bigg[\frac{1}{N^3} \sum_{i_1\neq i_2\neq i_3\neq i_4} H_{i_2i_4}G_{i_4i_3}G_{i_2i_2}G_{i_1i_1} G_{i_1i_3}P' P^{D-2} \ol{P^D}\bigg]\,.
\end{split}\end{align*}
As in the proof of Lemma~\ref{le first round lemma q}, we now apply the cumulant expansion to the right side. The leading terms of the expansion is
\begin{align}\label{le leading for second q round}
  N\kappa_t^{(3)}&\E\bigg[\frac{1}{N^3} \sum_{i_1\neq i_2\neq i_3}m G_{i_2i_3}G_{i_2i_2}G_{i_1i_1} G_{i_1i_3}P' P^{D-2} \ol{P^D}\bigg]\,,
\end{align}
and, thanks to the additional factor of $q_t^{-1}$ from the cumulant $\kappa_t^{(3)}$, all other terms in the cumulant expansion are negligible, as can be checked by power counting as in the proof of Lemma~\ref{le first round lemma q}. Replacing in~\eqref{le leading for second q round}~$m$ by $\widetilde m$, we then get
\begin{align*}
|z+\wt m|\, |N\kappa_t^{(3)}|&\bigg|\E\bigg[\frac{1}{N^3} \sum_{i_1\neq i_2\neq i_3} G_{i_2i_3}G_{i_2i_2}G_{i_1i_1} G_{i_1i_3}P' P^{D-2} \ol{P^D}\bigg]\bigg|\le C\Phiepsi\,,
\end{align*}
for $N$ sufficiently large; \cf~\eqref{for second round needed}. Since $|z+\wt m(z)|\ge 1/6$, $z\in\mathcal{D}$ by Remark~\ref{rem:stability of wt m}, we conclude that the first term on the right side of~\eqref{le Ir 112} is negligible. In the same way one shows that the second term is negligible, too. We leave the details to the reader.
\end{proof}
We hence conclude from~\eqref{e12r1o} and~\eqref{the sunday lemmal} that $\E\Ir_{r,1}$ is negligible for all $r\ge 2$.

Consider next the terms
\begin{align*}
  \E\Ir_{r,s}= N\kappa_t^{(r+1)}\E\bigg[ \frac{1}{N^2} \sum_{i_1\not=i_2}(\partial_{i_1i_2}^{r-s}G_{i_2i_1}) \partial_{i_1i_2}^{s} ( P^{D-1} \ol{P^D} ) \bigg] \,,
\end{align*}
with $2\le s\le  r$. We proceed in a similar way as in Remark~\ref{remark power counting II}. We note that each term in $\partial_{i_1i_2}^{r-s}G_{i_2i_1}$ contains at least one off-diagonal Green function when $r-s$ is even, yet when $r-s$ is odd there is a term with no off-diagonal Green function entries. Since $s\ge 2$, the partial derivative $\partial_{i_1i_2}^s$ acts on $P$ or~$\ol P$ (or their derivatives) more than once in total; \cf Remark~\ref{remark power counting II}.  Consider such a term with
$$
P^{D-s_1} \ol{P^{D-s_2}} \left( P' \right)^{s_3} \left( \ol{P'} \right)^{s_4} \,,
$$
for $ 1\le s_1 \le D$ and $0\le s_2\le D$. Since  $P''(m), P'''(m), P^{(4)}(m) \prec 1$ and $P^{(5)}=0$, we do not include derivatives of order two and higher here. We see that such a term was generated from $P^{D-1}\overline{P^D}$ by letting the partial derivative $\partial_{i_1i_2}$ act $(s_1 -1)$-times on~$P$ and $s_2$-times on~$\ol P$, which implies that $s_3 \leq s_1 -1$ and $s_4 \leq s_2$. If $s_3 < s_1 -1$, then $\partial_{i_1i_2}$ acted on $P'$ as well $[(s_1 -1)-s_3]$-times, and a similar argument holds for $\ol{P'}$. Whenever $\partial_{i_1i_2}$ acts on $P$ or~$\ol P$ (or their derivatives), it generates a fresh summation index $i_l$, $l\ge 3$, with a term $ 2N^{-1} \sum_{i_l} G_{i_2i_l}G_{i_li_1}$. The total number of fresh summation indices in this case is
$$
 (s_1 -1) + s_2 + [(s_1 -1) - s_3] + [s_2 - s_4] = 2s_1 + 2s_2 - s_3 - s_4 -2\,.
$$
Assume first that $r=s$ so that $ \partial_{i_1i_2}^{r-s}G_{i_2i_1}=G_{i_2i_1}$. Then applying Lemma~\ref{lem:2 off-diagonal} $(2s_1 + 2s_2 - s_3 - s_4 -2)$-times and letting $s_0=s_1+s_2$ and $s'=s_3+s_4$, we obtain an upper bound, $r=s\ge 2$,
\begin{align}\begin{split} \label{e12s-2Kevin}
|\E\Ir_{r,r}|&\le \frac{N^\epsilon}{q_t^{r-1}}\sum_{s_0=2}^{2D}\sum_{s'=1}^{s_0-1}\E\bigg[\bigg(\frac{\im m}{N\eta} \bigg)^{1/2}\bigg(\frac{\im m}{N\eta} \bigg)^{2s_0-s'-2}|P'|^{s'}|P|^{2D-s_0} \bigg] \le \Phiepsi\,,
\end{split}\end{align}
for $N$ sufficiently large, \ie $\E\Ir_{r,r}$, $r\ge 2$, is negligible.

Second, assume that $2\le s<r$. Then applying Lemma~\ref{lem:2 off-diagonal} $(2s_1 + 2s_2 - s_3 - s_4 -2)$-times, we get
\begin{align}\begin{split} \label{e12s-2}
|\E\Ir_{r,s}|&\le \frac{N^\epsilon}{q_t^{r-1}}\sum_{s_0=2}^{2D}\sum_{s'=1}^{s_0-2}\E\bigg[\bigg(\frac{\im m}{N\eta} \bigg)^{2s_0-s'-2}|P'|^{s'}|P|^{2D-s_0} \bigg] \\
&\qquad\qquad +\frac{N^\epsilon}{q_t^{r-1}}\sum_{s_0=2}^{2D}\E\bigg[\bigg(\frac{\im m}{N\eta} \bigg)^{s_0-1}|P'|^{s_0-1}|P|^{2D-s_0} \bigg] \,,
\end{split}\end{align}
for $N$ sufficiently large with $2\le s<r$. In particular, $|\E\Ir_{r,s}|\le \Phiepsi$, $2\le s<r$. In~\eqref{e12s-2} the second term bounds the terms corresponding to $s_0-1=s'$ obtained by acting on $\partial_{i_1i_2}$ exactly $(s_1 -1)$-times on $P$ and $s_2$-times on $\ol P$ but never on their derivatives.

To sum up, we showed that $\E\Ir_{r,s}$ is negligible, for $1\le s<r$. This proves~\eqref{le summary equation 2} for $1\le s<r$.

\subsection{Estimate on $\Ir_{3,0}$} \label{sub:e113}
We first notice that $\Ir_{3,0}$ contains terms with zero, two or four off-diagonal Green function entries and we split accordingly
\begin{align*}
 w_{\Ir_{3,0}}\Ir_{3,0}=w_{\Ir_{3,0}^{(0)}}\Ir_{3,0}^{(0)}+w_{\Ir_{3,0}^{(2)}}\Ir_{3,0}^{(2)}+w_{\Ir_{3,0}^{(4)}}\Ir_{3,0}^{(4)}\,.
\end{align*}
When there are two off-diagonal entries, we can use Lemma~\ref{lem:2 off-diagonal} to get the bound
\begin{align} \begin{split} \label{e1132}
|\E\Ir_{3,0}^{(2)}|=\bigg| {N\kappa_t^{(4)}} \E \bigg[\frac{1}{N^2} \sum_{i_1 \neq i_2} G_{i_2i_2} G_{i_1i_1} (G_{i_2i_1})^2 P^{D-1} \ol{P^D} \bigg] \bigg|& \leq \frac{N^{\epsilon}}{ q_t^{2}} \E \bigg[ \frac{\im m}{N\eta} |P|^{2D-1} \bigg]\le \Phiepsi\,,
\end{split} \end{align}
for $N$ sufficiently large. A similar estimate holds for $|\E\Ir_{3,0}^{(2)}|$. The only non-negligible term is $\Ir_{3,0}^{(0)}$. 

For $n\in\N$, set
\begin{align}\label{definition of SN}
S_n\equiv S_n(z) \deq \frac{1}{N} \sum_{i=1}^N (G_{ii}(z))^n\,.
\end{align}
By definition $S_1=m$. We remark that $|S_n| \prec 1$, for any fixed $n$, by Proposition~\ref{local semiclrcle law}.

\begin{lemma} \label{le combinatorics lemma} We have
\begin{align} \label{e1130 expansion}
w_{\Ir_{3,0}^{(0)}}\E\Ir_{3,0}^{(0)} =- N\kappa_t^{(4)} \E \Big[S_2^2 P^{D-1} \ol{P^D} \Big] \,.
\end{align}
\end{lemma}

\begin{proof}
Recalling the definition of $\Ir_{r,s}$ in~\eqref{the IR}, we have
\begin{align*}
w_{\Ir_{3,0}}\Ir_{3,0}=\frac{N\kappa_t^{(4)}}{3!} \E \bigg[\frac{1}{N^2} \sum_{i_1 \neq i_2} \big( \partial_{i_1i_2}^3 G_{i_2i_1} \big) P^{D-1} \ol{P^D} \bigg]\,.
\end{align*}
We then easily see that the terms with no off-diagonal entries in $\partial_{i_1i_2}^3 G_{i_2i}$ are of the form
$$
-G_{i_2i_2} G_{i_1i_1} G_{i_2i_2} G_{i_1i_1}\,.
$$
We only need to determine the weight $w_{\Ir_{3,0}^{(0)}}$. With regard to the indices, taking the third derivative corresponds to putting the indices $i_2i_1$ or $i_1i_2$ three times. In that sense, the very first $i_2$ and the very last $i_1$ are from the original $G_{i_2i_1}$. The choice of $i_2i_1$ or $i_1i_2$ must be exact in the sense that the connected indices in the following diagram must have been put at the same time:	
$$
i_2\underbrace{i_2 \quad i_1} \underbrace{i_1 \quad i_2} \underbrace{i_2 \quad i_1}i_1\,.
$$ 
Thus, the only combinatorial factor we have to count is the order of putting the indices. In this case, we have three connected indices, so the number of terms must be $3!=6$. Thus, $w_{\Ir_{3,0}^{(0)}}=1$ and \eqref{e1130 expansion} indeed holds.
\end{proof}

\begin{lemma}\label{le S22 lemma} For any (small) $\epsilon>0$, we have, for all $z\in\mathcal{D}$,
 \begin{align}\label{le S22 lemma equation}
  N\kappa_t^{(4)}\E  \Big[ S_2^2 P^{D-1} \ol{P^D} \Big] =N\kappa_t^{(4)} \E \Big[ m^4 P^{D-1} \ol{P^D} \Big]+ O( \Phiepsi)\,.
 \end{align}
\end{lemma}

\begin{proof}
Fix $\epsilon>0$. We first claim that
\begin{align} \label{s21=s22}
N\kappa_t^{(4)}\E \Big[S_2^2 P^{D-1} \ol{P^D} \Big]=N\kappa_t^{(4)} \E \Big[ m^2 S_2 P^{D-1} \ol{P^D} \Big]+O(\Phiepsi)\,.
\end{align}
The idea is to expand the term $\E[zm S_2^2 P^{D-1} \ol{P^D}]$ in two different ways and compare the results. Using the resolvent identity \eqref{expansion indentity} and Lemma~\ref{le stein lemma}, we get
\begin{align} \begin{split} \label{s21}
\E\big[zm S_2^2 P^{D-1} \ol{P^D}\big] &= -\E\big[S_2^2 P^{D-1} \ol{P^D}\big] +  \E \bigg[\frac{1}{N} \sum_{i_1\not=i_2} H_{i_1i_2} G_{i_2i_1} S_2^2 P^{D-1} \ol{P^D} \bigg] \\
&= -\E\big[S_2^2 P^{D-1} \ol{P^D}\big] +  \sum_{r=1}^{\ell'} \frac{N\kappa_t^{(r+1)}}{r!} \E \bigg[\frac{1}{N^2} \sum_{i_1\not= i_2} \partial_{i_1i_2}^r \Big( G_{i_2i_1} S_2^2 P^{D-1} \ol{P^D} \Big) \bigg]\\ &\qquad\qquad + \E\bigg[\Omega_{\ell'}\bigg(\frac{1}{N}\sum_{i_1\not= i_2}  H_{i_1i_2} G_{i_2i_1} S_2^2 P^{D-1} \ol{P^D}\bigg) \bigg]\,,
\end{split} \end{align}
for arbitrary $\ell'\in\N$. Using the resolvent identity~\eqref{expansion indentity} once more, we write
\begin{align*}
z S_2 = \frac{1}{N} \sum_{i=1}^N z G_{ii} G_{ii} = -\frac{1}{N} \sum_{i=1}^N G_{ii} + \frac{1}{N} \sum_{i_1\not= i_2} H_{i_1i_2} G_{i_2i_1} G_{i_1i_1}\,.
\end{align*}
Thus, using Lemma~\ref{le stein lemma}, we also have
\begin{align}  \label{s22}
\E[zm S_2^2 P^{D-1} \ol{P^D}] &= -\E[m^2 S_2 P^{D-1} \ol{P^D}] +\E \bigg[ \frac{1}{N}  \sum_{i_1 \not=i_2} H_{i_1i_2} G_{i_2i_1} G_{i_1i_1} m S_2 P^{D-1} \ol{P^D} \bigg]\nonumber \\
&= -\E[m^2 S_2 P^{D-1} \ol{P^D}] +  \sum_{r=1}^{\ell'} \frac{N\kappa_t^{(r+1)}}{r!} \E \bigg[ \frac{1}{N^2}\sum_{i_1\not=i_2} \partial_{i_1i_2}^r \Big( G_{i_2i_1} G_{i_1i_1} m S_2 P^{D-1} \ol{P^D} \Big) \bigg]\nonumber\\
&\qquad\qquad+  \E\bigg[\Omega_{\ell'}\bigg(\frac{1}{N}\sum_{i_1\not=i_2} H_{i_1i_2} G_{i_2i_1} G_{i_1i_1} m S_2 P^{D-1} \ol{P^D}\bigg) \bigg]\,,
\end{align}
for arbitrary $\ell'\in\N$. By Corollary~\ref{le corollary for error terms omega} and Remark~\ref{remark on other expansions}, the two error terms $\E[\Omega_{\ell'}(\cdot)]$ in~\eqref{s21} and~\eqref{s22} are negligible for $\ell'\ge 8D$. 

With the extra factor $N\kappa_t^{(4)}$, we then write
\begin{align}\label{le with the extra factor}\begin{split}
 N\kappa_t^{(4)}\E[zm S_2^2 P^{D-1} \ol{P^D}] &= -N\kappa_t^{(4)}\E[S_2^2 P^{D-1} \ol{P^D}] +  \sum_{r=1}^{\ell'}\sum_{s=0}^r w_{\scriptsize{\ttilde \Ir_{r,s}}}\E\ttilde \Ir_{r,s}+O(\Phiepsi)\,,\\
 N\kappa_t^{(4)} \E[zm S_2^2 P^{D-1} \ol{P^D}] &=- N\kappa_t^{(4)}\E[m^2 S_2 P^{D-1} \ol{P^D}] + \sum_{r=1}^{\ell'}\sum_{s=0}^r w_{\scriptsize{\ttilde[2] \Ir_{r,s}}}\E{\ttilde[2]{ \Ir}}_{r,s}+O(\Phiepsi)\,,
\end{split}\end{align}
with
\begin{align}\begin{split}
 \ttilde \Ir_{r,s}&\deq N\kappa_t^{(4)} N\kappa_t^{(r+1)}\frac{1}{N^2}\sum_{i_1\not=i_2}\big(\partial_{i_1i_2}^{r-s}\big( G_{i_2i_1}S_2^2\big)\big)\big(\partial_{i_1i_2}^s \big(P^{D-1}\ol{P^D}\big)\big)\,,\\
 {\ttilde[2]{ \Ir}}_{r,s}&\deq N\kappa_t^{(4)} N\kappa_t^{(r+1)}\frac{1}{N^2}\sum_{i_1\not=i_2}\big(\partial_{i_1i_2}^{r-s}\big(G_{i_2i_1}G_{i_1i_1} mS_2\big)\big)\big(\partial_{i_1i_2}^s \big(P^{D-1}\ol{P^D}\big)\big)
\end{split}\end{align}
and $ w_{\scriptsize{\ttilde \Ir_{r,s}}}= w_{\scriptsize{\ttilde[2] \Ir_{r,s}}}=(r!)^{-1}((r-s)!)^{-1}$.

 For $r=1$, $s=0$, we find that
\begin{align}\label{le they agree 1}\begin{split}
 \E\ttilde \Ir_{1,0}&= -N\kappa_t^{(4)} \E\bigg[\frac{1}{N^2}\sum_{i_1\not=i_2}G_{i_1i_1}G_{i_2i_2}S_2^2 P^{D-1}\ol{P^D}\bigg]+O(\Phiepsi)\\
 &=-N\kappa_t^{(4)} \E\Big[m^2S_2^2 P^{D-1}\ol{P^D}\Big]+O(\Phiepsi)\,,
 \end{split}\end{align}
 and similarly
 \begin{align}\label{le they agree 2}\begin{split}
 \E\ttilde[2]\Ir_{1,0}&= -N\kappa_t^{(4)} \E\bigg[\frac{1}{N^2}\sum_{i_1\not=i_2} G_{i_1i_1}^2G_{i_2i_2}mS_2P^{D-1}\ol{P^D}\bigg]+O(\Phiepsi)\\
 &= -N\kappa_t^{(4)} \E\bigg[m^2S_2^2P^{D-1}\ol{P^D}\bigg]+O(\Phiepsi)\,,
\end{split}\end{align}
where we used~\eqref{definition of SN}. We hence conclude that $\E\ttilde\Ir_{1,0}$ equals $\E\ttilde[2]\Ir_{1,0}$ up to negligible error.

Following the ideas in Subsection~\ref{sub:e112}, we can bound
\begin{align*}
 \big|\E\ttilde \Ir_{1,1}\big|&\le\frac{N^\epsilon}{q_t^{2}}\bigg|\E\bigg[\frac{1}{N^2}\sum_{i_1\not=i_2}G_{i_1i_2}S_2^2(\partial_{i_1i_2} P^{D-1}\ol{P^D})\bigg]\bigg|\le \frac{N^\epsilon}{q_t^{2}}\E\bigg[\bigg(\frac{\im m}{N\eta}\bigg)^{3/2}|P'||P|^{2D-2}\bigg]\le \Phiepsi\,,
\end{align*}
and similarly $|\E\ttilde[2] \Ir_{1,1}|\le \Phiepsi$, for $N$ sufficiently large. In fact, for $r\ge 2$, $s\ge 0$ we can use, with small notational modifications the power counting outlined in Remark~\ref{remark power counting II} and Remark~\ref{remark power counting III} to conclude that 
\begin{align*}
 \E\ttilde\Ir_{r,s}=O(\Phiepsi)\,,\qquad\qquad \E\ttilde[2]\Ir_{r,s}=O( \Phiepsi)\,, \qquad\qquad (r\ge 2\,,s\ge 0)\,.
\end{align*}
Therefore the only non-negligible terms on the right hand side of~\eqref{le with the extra factor} are $N\kappa_t^{(4)}\E[S_2^2 P^{D-1} \ol{P^D}] $, $N\kappa_t^{(4)}\E[m^2 S_2 P^{D-1} \ol{P^D}]$ as well as $\ttilde\Ir_{1,0}$ and $\ttilde[2]\Ir_{1,0}$. Since, by~\eqref{le they agree 1} and~\eqref{le they agree 2}, the latter agree up to negligible error terms, we conclude that the former two must be equal up do negligible error terms. Thus~\eqref{s21=s22} holds.

Next, expanding the term $\E[zm^3 S_2 P^{D-1} \ol{P^D}]$ in two different ways similar to above, we further~get
\begin{align} \label{s22=s23}
N\kappa_t^{(4)}\E \Big[m^2 S_2 P^{D-1} \ol{P^D} \Big]= N\kappa_t^{(4)} \E \Big[ m^4 P^{D-1} \ol{P^D} \Big]+O(\Phiepsi)\,.
\end{align}
Together with~\eqref{s21=s22} this shows~\eqref{le S22 lemma equation} and concludes the proof of the lemma.
\end{proof}

Finally, from~Lemma~\ref{le combinatorics lemma} and Lemma~\ref{le S22 lemma}, we conclude that
\begin{align}\label{le final Ir30}
  w_{\Ir_{3,0}}\E[\Ir_{3,0}]=-\frac{\nm_t^{(4)}}{q_t^2} \E \Big[ m^4 P^{D-1} \ol{P^D} \Big]+O(\Phiepsi)\,.
\end{align}
This proves the third estimate in~\eqref{le summary equation 1}.
\begin{proof}[Proof of Lemma~\ref{summary expansions}]
The estimates in~\eqref{le summary equation 1} were obtained in~\eqref{e111},~\eqref{final bound Ir_20},~\eqref{e115},~\eqref{e114} and~\eqref{le final Ir30}. Estimate~\eqref{le summary equation 2} follows from~\eqref{e1211},~\eqref{e12r1o},~\eqref{e12s-2Kevin} and~\eqref{e12s-2}.
\end{proof}

\section{Tracy--Widom limit: Proof of Theorem~\ref{thm tw}} \label{sec:tw}

In this section, we prove the Tracy--Widom limit of the extremal eigenvalues, Theorem~\ref{thm tw}. As we explained in Section~\ref{Tracy-Widom limit and Green function comparison}, the distribution of the largest eigenvalue of $H$ can be obtained by considering the imaginary part of the normalized trace of the Green function $m$ of $H$. For $\eta > 0$, we introduce 
\begin{align}\label{le poisson kernel}
\theta_{\eta}(y) = \frac{\eta}{\pi(y^2 + \eta^2)}\,,\qquad\qquad (y\in\R)\,.
\end{align}
Using the functional calculus and the definition of the Green function, we have
\begin{align}
\im m(E+\ii\eta)=\frac{\pi}{ N} \Tr \theta_{\eta}(H-E)\,,\qquad\qquad (z=E+\ii\eta\in\C^+)\,.
\end{align}
We have the following proposition, which corresponds to Corollary~6.2 in \cite{EYY} or Lemma~6.5 of~\cite{EKYY2}.

\begin{proposition} \label{prop:cutoff}
Suppose that $H$ satisfies Assumption~\ref{assumption H} with $\phi > 1/6$. Denote by $\lambda_1^H$ the largest eigenvalue of $H$. 
Fix $\epsilon > 0$. Let $E \in \R$ such that $|E-L| \leq N^{-2/3 + \epsilon}$. Let $E_+ \deq L + {2}N^{-2/3 + \epsilon}$ and define $\chi_E \deq \lone_{[E, E_+]}$.  Set $\eta_1 \deq N^{-2/3 - 3\epsilon}$ and $\eta_2 \deq N^{-2/3 - 9\epsilon}$. Let $K: \R \to [0, \infty)$ be a smooth function satisfying
\begin{align}
K(x) =
	\begin{cases}
	1 & \text{ if } |x| < 1/3 \\
	0 & \text{ if } |x| > 2/3
	\end{cases},
\end{align}
which is a monotone decreasing on $[0, \infty)$. Then, for any~$D > 0$,
\begin{align}
\E \left[ K \left( \Tr (\chi_E \ast \theta_{\eta_2})(H) \right) \right] > \p (\lambda_1^H \leq E-\eta_1) - N^{-D}
\end{align}
and
\begin{align}
\E \left[ K \left( \Tr (\chi_E \ast \theta_{\eta_2})(H) \right) \right] < \p (\lambda_1^H \leq E+\eta_1) + N^{-D}\,,
\end{align}
for $N$ sufficiently large, with $\theta_{\eta_2}$ as in~\eqref{le poisson kernel}.
\end{proposition}

We prove Proposition~\ref{prop:cutoff} in Section~\ref{preliminaries of edge universality}. We move on to the Green function comparison theorem. Let $W^{\GOE}$ be a GOE matrix independent of $H$ with vanishing diagonal entries as introduced in Subsection~\ref{Tracy-Widom limit and Green function comparison} and denote by $m^{\GOE}\equiv m^{W^{\GOE}}$ the normalized trace of its Green function.

\begin{proposition}[Green function comparison] \label{prop:green}
Let $\epsilon>0$ and set $\eta_0 = N^{-2/3 - \epsilon}$. Let $E_1, E_2\in\R$ satisfy $|E_1|, |E_2| \leq N^{-2/3 + \epsilon}$. Let $F : \R \to \R$ be a smooth function satisfying
\begin{align} \label{F bound}
\max_{x\in\R} |F^{(l)}(x)| (|x|+1)^{-C} \leq C\,,\qquad \qquad (l\in\llbracket 1,11\rrbracket)\,.
\end{align}
Then, for any sufficiently small $\epsilon > 0$, there exists $\delta>0$ such that
\begin{align} \label{green_comp}
\bigg| \E F \bigg( N \int_{E_1}^{E_2} \im m(x + L + \ii \eta_0) \,\dd x \bigg) - \E F \bigg( N \int_{E_1}^{E_2} \im m^{\GOE}(x + 2 + \ii \eta_0)\, \dd x \bigg) \bigg| \leq N^{-\delta}\,,
\end{align}
for sufficiently large $N$.
\end{proposition}

Proposition~\ref{prop:green} is proved in Section~\ref{subsection green function comparison}. We are ready to prove Theorem~\ref{thm tw}.

\begin{proof}[Proof of Theorem~\ref{thm tw}]
Fix $\epsilon > 0$ and set $\eta_1 \deq N^{-2/3 - 3\epsilon}$ and $\eta_2 \deq N^{-2/3 - 9\epsilon}$. Consider $E = L + sN^{-2/3}$ with $s \in (-N^{-2/3+\epsilon}, N^{-2/3+\epsilon})$. For any $D > 0$, we find from Proposition~\ref{prop:cutoff} that
\begin{align*}
\p (\lambda_1^H \leq E) < \E K \left( \Tr (\chi_{E+\eta_1} \ast \theta_{\eta_2})(H) \right) + N^{-D}.
\end{align*}
Applying Proposition~\ref{prop:green} with $9\epsilon$ instead of $\epsilon$ and setting $E_1 = E-L +\eta_1$, $E_2 = E_+ -L$, we find that
\begin{align} \begin{split}\nonumber
&\E K \bigg( \Tr (\chi_{E+\eta_1} \ast \theta_{\eta_2})(H) \bigg) = \E K \bigg( \frac{N}{\pi} \int_{E_1}^{E_2} \im m(x + L + \ii \eta_2) \,\dd x \bigg) \\
&\leq \E K \bigg( \frac{N}{\pi} \int_{E_1}^{E_2} \im m^{\GOE}(x + 2 + \ii \eta_2)\, \dd x \bigg) + N^{-\delta} = \E K \bigg( \Tr (\chi_{[E_1+2, E_2+2]} \ast \theta_{\eta_2})(W^{\GOE}) \bigg) + N^{-\delta}
\end{split} \end{align}
for some $\delta > 0$. Hence, applying Proposition~\ref{prop:cutoff} again to the matrix $W^{\GOE}$, we get
\begin{align}  \label{eq:TW upper bound}
&\p \left( N^{2/3} (\lambda_1^H -L) \leq s \right) = \p \left(\lambda_1^H \leq E \right) \\
&< \p \left(\lambda_1^{\GOE} \leq E_1 + 2 + \eta_1 \right) + N^{-D} + N^{-\delta} = \p \left( N^{2/3} (\lambda_1^{\GOE} -2) \leq s + 2N^{-3\epsilon} \right) + N^{-D} + N^{-\delta}\,.\nonumber
 \end{align}
Similarly, we can also check that
\begin{align} \label{eq:TW lower bound}
\p \left( N^{2/3} (\lambda_1^H -L) \leq s \right) > \p \left( N^{2/3} (\lambda_1^{\GOE} -2) \leq s - 2N^{-3\epsilon} \right) - N^{-D} - N^{-\delta}.
\end{align}
Since the right sides of Equations~\eqref{eq:TW upper bound} and \eqref{eq:TW lower bound} converge both in probability to $F_1(s)$, the $\GOE$ Tracy--Widom distribution, as $N$ tends to infinity we conclude that
\begin{align}
\lim_{N \to \infty} \p \Big( N^{2/3} (\lambda^H_1 -L)\leq s \Big) = F_1 (s)\,.
\end{align}
This proves Theorem~\ref{thm tw}.
\end{proof}

In the rest of this section, we prove Propositions~\ref{prop:cutoff} and \ref{prop:green}

\subsection{Proof of Proposition~\ref{prop:cutoff}}\label{preliminaries of edge universality}

For a given $\epsilon > 0$, we chose $E \in \R$ such that $|E-L| \leq N^{-2/3 + \epsilon}$, $E_+ \deq L + {2}N^{-2/3 + \epsilon}$, $\eta_1 = N^{-2/3 - 3\epsilon}$, and $\eta_2 = N^{-2/3 - 9\epsilon}$.  In principle, we could adopt the strategy in the proof of Corollary~6.2 of~\cite{EYY} after proving the optimal rigidity estimate at the edge with the assumption $q \gg N^{1/6}$ and checking that such an optimal bound is required only for the eigenvalues at the edge. However, we introduce a slightly different approach that directly compares $\Tr (\chi_E \ast \theta_{\eta_2})$ and $\Tr \chi_{E-\eta_1}$ by using the local law.

\begin{proof}[Proof of Proposition~\ref{prop:cutoff}]
For an interval $I \subset \R$, let $\caN_I$ be the number of the eigenvalues in $I$, \ie
\begin{align}
\caN_I \deq |\{ \lambda_i : \lambda_i \in I \}|\,.
\end{align}
We compare $(\chi_E \ast \theta_{\eta_2})(\lambda_i)$ and $\chi_{E-\eta_1}(\lambda_i)$ by considering the following cases:

{\it Case 1:} If $x \in [E + \eta_1, E_+ - \eta_1)$, then $\chi_{E-\eta_1}(x) = 1$ and
\begin{align} \begin{split}
(\chi_E \ast \theta_{\eta_2})(x) - 1 &= \frac{1}{\pi} \int_E^{E_+} \frac{\eta_2}{(x-y)^2 + \eta_2^2} \,\dd y -1 = \frac{1}{\pi} \left( \tan^{-1} \frac{E_+ - x}{\eta_2} - \tan^{-1} \frac{E-x}{\eta_2} \right) - 1 \\
&= -\frac{1}{\pi} \left( \tan^{-1} \frac{\eta_2}{E_+ - x} + \tan^{-1} \frac{\eta_2}{x-E} \right) = O \left( \frac{\eta_2}{\eta_1} \right) = O(N^{-6\epsilon})\,.
\end{split} \end{align}
For any $E' \in [E + \eta_1, E_+ - \eta_1)$, with the local law, Proposition~\ref{prop:local}, we can easily see that
\begin{align}\label{the five}
\frac{\caN_{[E'-\eta_1, E'+\eta_1)}}{5N \eta_1} \leq \frac{1}{N} \sum_{i=1}^N \frac{\eta_1}{|\lambda_i -E'|^2 + \eta_1^2} = \im m(E' + \ii \eta_1) \leq \frac{N^{\epsilon/2}}{N\eta_1} + \frac{N^{\epsilon/2}}{q^2}\,,
\end{align}
with high probability, where we used $\im \wt m(E'+\ii\eta_1)\le C\sqrt{\kappa(E')+\eta_1}\ll N^{\epsilon/2}/(N\eta_1)$. Thus, considering at most $[E_+ - (L-N^{-2/3+\epsilon})]/\eta_1 = 3 N^{4\epsilon}$ intervals, we find that
\begin{align}
\caN_{[E + \eta_1, E_+ - \eta_1)} \leq CN^{9\epsilon /2}
\end{align}
and
\begin{align} \label{eq:large x}
\sum_{i: E + \eta_1 < \lambda_i < E_+ - \eta_1} \big( (\chi_E \ast \theta_{\eta_2})(\lambda_i) - \chi_{E-\eta_1} (\lambda_i) \big) \leq N^{-\epsilon}\,,
\end{align}
with high probability.

{\it Case 2:} For $x < E-\eta_1$, choose $k\ge 0$ such that $3^k \eta_1 \leq E-x < 3^{k+1} \eta_1$. Then, $\chi_{E-\eta_1}(x)=0$ and
\begin{align} \begin{split}
(\chi_E \ast \theta_{\eta_2})(x) &= \frac{1}{\pi} \left( \tan^{-1} \frac{E_+ - x}{\eta_2} - \tan^{-1} \frac{E-x}{\eta_2} \right) = \frac{1}{\pi} \left( \tan^{-1} \frac{\eta_2}{E-x} - \tan^{-1} \frac{\eta_2}{E_+ -x} \right) \\
&< \frac{1}{2} \left( \frac{\eta_2}{E-x} -\frac{\eta_2}{E_+ -x} \right) = \frac{1}{2} \frac{\eta_2(E_+ - E)}{(E-x)(E_+ -x)} < 2 N^{-4/3 -8\epsilon} \cdot 3^{-2k} \eta_1^{-2}.
\end{split} \end{align}
Abbreviate $\caN_k = \caN_{(E- 3^{k+1} \eta_1, E - 3^k \eta_1]}$. Consider
\begin{align}
\im m(E- 2 \cdot 3^k \eta_1 + \ii \cdot 3^k \eta_1) = \frac{1}{N} \sum_{i=1}^N \frac{3^k \eta_1}{| \lambda_i - (E- 2 \cdot 3^k \eta_1)|^2 + (3^k \eta_1)^2} > \frac{1}{N} \frac{\caN_k}{2 \cdot 3^k \eta_1}.
\end{align}
With the local law, Proposition~\ref{prop:local}, and the estimate $\im \wt m(x+ \ii y) \sim \sqrt{|x-L| + y}$, we find that 
\begin{align}
\im m(E- 2 \cdot 3^k \eta_1 + \ii \cdot 3^k \eta_1) \leq C \sqrt{3^k \eta_1} + \frac{N^{\epsilon/2}}{N \cdot 3^k \eta_1} \leq N^{5\epsilon} \sqrt{3^k \eta_1}
\end{align}
and hence
\begin{align}
\frac{1}{N} \frac{\caN_k}{2 \cdot 3^k \eta_1} < N^{5\epsilon} \sqrt{3^k \eta_1}\,,
\end{align}
with high probability. Thus, with high probability,
\begin{align} \begin{split} \label{eq:small x}
\sum_{i: \lambda_i < E - \eta_1} \big( (\chi_E \ast \theta_{\eta_2})(\lambda_i) - \chi_{E- \eta_1} (\lambda_i) \big) &\leq 2 \sum_{k=0}^{2 \log N} N^{-4/3 -8\epsilon} \cdot 3^{-2k} \eta_1^{-2} \caN_k \\
&\leq 4 N^{-1/3 -3\epsilon} \eta_1^{-1/2} \sum_{k=0}^{\infty} 3^{-k/2} \leq 10 N^{-3\epsilon /2}\,.
\end{split}\end{align}

{\it Case 3:} By Proposition \ref{prop:norm bound} there are with high probability no eigenvalues in $[E_+ - \eta_1, \infty)$.

{\it Case 4:}  For $x \in [E-\eta_1, E+\eta_1)$, we use the trivial estimate
\begin{align}\label{le trivial estimate}
(\chi_E \ast \theta_{\eta_2})(x) < 1 = \chi_{E-\eta_1}(x)\,.
\end{align}

Considering the above cases, we find that
\begin{align} \label{chi comparison}
\Tr (\chi_E \ast \theta_{\eta_2})(H) \leq \Tr \chi_{E-\eta_1}(H) + N^{-\epsilon}\,,
\end{align}
with high probability. From the definition of the cutoff $K$ and the fact that $\Tr \chi_{E-\eta_1}(H)$ is an integer,
\begin{align*}
K( \Tr \chi_{E-\eta_1}(H) + N^{-\epsilon} ) = K( \mathcal{N}_{[E-\eta_1,E_+]} )\,.
\end{align*}
Thus, since $K$ is monotone decreasing on $[0, \infty)$, \eqref{chi comparison} implies that
\begin{align*}
K ( \Tr (\chi_E \ast \theta_{\eta_2})(H) ) \geq K ( \Tr \chi_{E-\eta_1}(H) )
\end{align*}
with high probability. After taking expectation, we get
\begin{align*}
\E \left[ K \left( \Tr (\chi_E \ast \theta_{\eta_2})(H) \right) \right] > \p (\lambda_1 \leq E-\eta_1) - N^{-D}\,,
\end{align*}
for any $D>0$. This proves the first part of Proposition~\ref{prop:cutoff}. The second part can also be proved in a similar manner by showing that
\begin{align}
\Tr (\chi_E \ast \theta_{\eta_2})(H) \geq \Tr \chi_{E+\eta_1}(H) - N^{-\epsilon}\,,
\end{align}
applying the cutoff $K$ and taking expectation. In this argument~\eqref{le trivial estimate} gets replaced by
\begin{align}
\chi_{E+\eta_1}(x) = 0 < (\chi_E \ast \theta_{\eta_2})(x)\,,
\end{align}
for $x \in  [E-\eta_1, E+\eta_1)$. This proves Proposition~\ref{prop:cutoff}.
\end{proof}

\subsection{Green function comparison: Proof of Proposition~\ref{prop:green}}\label{subsection green function comparison}
We first introduce a lemma that has the analogue role of Lemma~\ref{le stein lemma} in calculations involving $H \equiv H_t$ and $\dot H \equiv \dd H_t / \dd t$.

\begin{lemma} \label{stein interpolation}
Fix $\ell\in \N$ and let $F\in C^{\ell+1}(\R;\C^+)$. Let $Y \equiv Y_0$ be a random variable with finite moments to order $\ell+2$ and let $W$ be a Gaussian random variable independent of $Y$. Assume that $\E[Y] = \E[W] = 0$ and $\E[Y^2] = \E[W^2]$. Define
\begin{align}
Y_t\deq \e{-t/2} Y_0 + \sqrt{1-\e{-t}} W\,,
\end{align}
and let $\dot Y_t\equiv\frac{\dd Y_t}{\dd t}$. Then,
 \begin{align}\label{le stein cor}
  \E \left[\dot Y_t F(Y_t) \right] = -\frac{1}{2} \sum_{r=2}^\ell \frac{\kappa^{(r+1)}(Y_0)}{r!} \e{-\frac{(r+1)t}{2}} \E \big[ F^{(r)}(Y_t) \big]+\E\big[\Omega_\ell(\dot Y_t F(Y_t))\big]\,,
 \end{align}
where $\E$ denotes the expectation with respect to $Y$ and $W$, $\kappa^{(r+1)}(Y)$ denotes the $(r+1)$-th cumulant of $Y$ and~$F^{(r)}$ denotes the $r$-th derivative of the function $F$. The error term~$\Omega_\ell$ in~\eqref{le stein cor} satisfies
\begin{align} \begin{split}
&\big|\E\big[\Omega_\ell(\dot Y F(Y_t))\big]\big| \le C_\ell \E[ |Y_t|^{\ell+2}]\sup_{|x|\le Q}|F^{(\ell+1)}(x)|+ C_\ell \E [|Y_t|^{\ell+2} \lone(|Y_t|>Q)]\sup_{x\in \R} |F^{(\ell+1)}(x)| \,,
\end{split} \end{align}
where $Q\ge 0$ is an arbitrary fixed cutoff and $C_\ell$ satisfies $C_\ell\le \frac{(C\ell)^\ell}{\ell!}$ for some numerical constant~$C$.
\end{lemma}

\begin{proof}
We follow the proof of Corollary 3.1 in~\cite{LP}. First, note that
\begin{align}
\dot Y_t = -\frac{\e{-t/2}}{2} Y + \frac{\e{-t}}{2\sqrt{1-\e{-t}}} W\,.
\end{align}
Thus,
 \begin{align*} \begin{split}
 &\E \left[ \dot Y_tF(Y_t) \right]= -\frac{\e{-t/2}}{2} \E \left[ Y F \big(\e{-t/2} Y + \sqrt{1-\e{-t}} W \big) \right] + \frac{\e{-t}}{2\sqrt{1-\e{-t}}} \E \left[ W F \big(\e{-t/2} Y + \sqrt{1-\e{-t}} W \big) \right].
 \end{split} \end{align*}
Applying Lemma~\ref{le stein lemma} and~\eqref{le SL}, we get~\eqref{le stein cor} since the first two moments of $W$ and $Y$~agree.
\end{proof}

\begin{proof}[Proof of Proposition~\ref{prop:green}]
Fix a (small) $\epsilon>0$. Consider $x \in [E_1, E_2]$. Recall the definition of $H_t$ in~\eqref{eq:A(t)}. For simplicity, let
\begin{align}
G \equiv G_t(x + L_t + \ii \eta_0)\,,\qquad \qquad m \equiv m_t(x + L_t + \ii \eta_0)\,,
\end{align}
with $\eta_0=N^{-2/3-\epsilon}$, 
and define
\begin{align}
X\equiv X_t \deq N \int_{E_1}^{E_2} \im m(x + L_t + \ii \eta_0) \,\dd x\,.
\end{align}
Note that $X \prec N^{\epsilon}$ and $|F^{(l)}(X)| \prec N^{C\epsilon}$ for $l \in\llbracket 1,11\rrbracket$. Recall from~\eqref{le L dot} that
\begin{align}
L = 2 + \e{-t} \nm^{(4)} q_t^{-2} + O(\e{-2t}q_t^{-4})\,, \qquad \dot L = -2 \e{-t} \nm^{(4)} q_t^{-2} + O(\e{-2t}q_t^{-4})\,,
\end{align}
where $q_t=\e{t/2}q_0$. Let $z = x + L_t + \ii \eta_0$ and $G \equiv G(z)$. Differentiating $F(X)$ with respect to $t$, we get
\begin{align} \begin{split} \label{time derivative}
\frac{\dd}{\dd t} \E F(X) &= \E \bigg[ F'(X) \frac{\dd X}{\dd t} \bigg] = \E \bigg[ F'(X) \im \int_{E_1}^{E_2} \sum_{i=1}^N \frac{\dd G_{ii}}{\dd t}\, \dd x \bigg] \\
&= \E \bigg[ F'(X) \im \int_{E_1}^{E_2} \bigg( \sum_{i, j, k} \dot H_{jk} \frac{\partial G_{ii}}{\partial H_{jk}} + \dot L \sum_{i, j} G_{ij} G_{ji} \bigg) \dd x \bigg]\,,
\end{split} \end{align}
where by definition
\begin{align}
\dot H_{jk}\equiv \dot{(H_t)}_{jk} = -\frac{1}{2} \e{-t/2} (H_0)_{jk} + \frac{\e{-t}}{2\sqrt{1-\e{-t}}} W_{jk}^{\GOE}\,.
\end{align}
Thus, from Lemma~\ref{stein interpolation}, we find that
\begin{align} \begin{split} \label{green claim'}
&\sum_{i, j, k} \E \left[ \dot H_{jk} F'(X) \frac{\partial G_{ii}}{\partial H_{jk}} \right] = -2 \sum_{i, j, k} \E \left[ \dot H_{jk} F'(X) G_{ij} G_{ki} \right] \\
&= \frac{\e{-t}}{N} \sum_{r=2}^{\ell} \frac{q_t^{-(r-1)} \nm^{(r+1)}}{r!} \sum_i \sum_{j \neq k} \E \left[ \partial_{jk}^r \left( F'(X) G_{ij} G_{ki} \right) \right] + O(N^{1/3 + C\epsilon})
\end{split} \end{align}
for $\ell = 10$, where we abbreviate $\partial_{jk} = \partial/(\partial H_{jk})$. Note that the $O(N^{1/3 + C\epsilon})$ error term in~\eqref{green claim'} is originated from~$\Omega_\ell$ in~\eqref{le stein cor}, which is $O(N^{C\epsilon} N^2 q_t^{-10})$ with $Y = H_{jk}$.

We claim the following lemma.
\begin{lemma} \label{lem:green claim}
For an integer $r \geq 2$, let
\begin{align}
J_r \deq \frac{\e{-t}}{N} \frac{q_t^{-(r-1)} \nm^{(r+1)}}{r!} \sum_{i=1}^N \sum_{j \neq k} \E \left[ \partial_{jk}^r \left( F'(X) G_{ij} G_{ki} \right) \right]\,.
\end{align}
Then,
\begin{align}
J_3 = 2\e{-t} \nm^{(4)} q_t^{-2} \sum_{i, j} \E \left[ F'(X) G_{ij} G_{ji} \right] + O(N^{2/3 -\epsilon'})
\end{align}
and, for any $r \neq 3$,
\begin{align}
J_r = O(N^{2/3 -\epsilon'})\,.
\end{align}
\end{lemma}

Assuming that Lemma~\ref{lem:green claim} holds, we obtain that there exists $\epsilon' > 2\epsilon$ such that, for all $t\in[0,6\log N]$,
\begin{align} \label{green claim}
\sum_{i, j, k} \E \left[ \dot H_{jk} F'(X) \frac{\partial G_{ii}}{\partial H_{jk}} \right] = - \dot L \sum_{i, j} \E \left[ G_{ij} G_{ji} F'(X) \right] + O(N^{2/3 -\epsilon'})\,,
\end{align}
which implies that the right side of \eqref{time derivative} is $O(N^{-\epsilon'/2})$. Integrating from $t=0$ to $t=6\log N$, we get
\begin{align*} \begin{split}
&\bigg| \E F \bigg( N \int_{E_1}^{E_2} \im m(x + L + \ii \eta_0) \,\dd x \bigg)_{t=0} - \E F \bigg( N \int_{E_1}^{E_2} \im m(x + L + \ii \eta_0) \,\dd x \bigg)_{t=6\log N} \bigg| \le N^{-\epsilon'/4}\,.
\end{split} \end{align*}
Comparing $\im m|_{t=6\log N}$ and $\im m^{\GOE}$ is trivial; if we let $\lambda_i(6\log N)$ be the $i$-th largest eigenvalue of $H_{6\log N}$ and $\lambda_i^{\GOE}$ the $i$-th largest eigenvalue of $W^{\GOE}$, then $|\lambda_i(6\log N) - \lambda_i^{\GOE}| \prec N^{-3}$, hence
\begin{align}
\left| \im m|_{t=6\log N} - \im m^{\GOE} \right| \prec N^{-5/3}\,.
\end{align}
This proves Proposition~\ref{prop:green}.
\end{proof}

It remains to prove Lemma~\ref{lem:green claim}. The proof uses quite similar ideas as in Section~\ref{sec:stein}, we thus sometimes omit some details and refer to the corresponding paragraph in Section~\ref{sec:stein}.
\begin{proof}[Proof of Lemma~\ref{lem:green claim}]
First, we note that in the definition of $J_r$ we may freely include or exclude the cases $i=j$, $i=k$, or $j=k$ in the summation $\sum_{i, j, k}$, since it contains at least one off-diagonal Green function entry, $G_{ij}$ or $G_{ki}$, and the sizes of such terms are at most of order
$$
N^2 N^{-1} q_t^{-1} q_t^{-1} \ll N^{2/3 - \epsilon'} \e{-t}\,,
$$
for any sufficiently small $\epsilon' > 0$. There may not be any off-diagonal Green function entries when $i=j=k$, but then there is only one summation index, hence we can neglect this case as well.

For the case $r \geq 5$, it is easy to see that $J_r = O(N^{2/3 -\epsilon'})$, since it contains at least two off-diagonal entries in $\partial_{jk}^r \left( F'(X) G_{ij} G_{ki} \right)$ and $|J_r|$ is bounded by
$$
N^3 N^{-1} q_t^{-4} N^{-2/3 +2\epsilon} \ll N^{2/3 - \epsilon'} \e{-2t}\,,
$$
which can be checked using Lemma~\ref{lem:2 off-diagonal} and a simple power counting. 

Therefore, we only need to consider the cases $r=2, 3, 4$. In the following subsections, we check each case and complete the proof of Lemma~\ref{lem:green claim}. See Equations~\eqref{g2}, \eqref{g3} and \eqref{g4} below.
\end{proof}

\subsubsection{Proof of Lemma~\ref{lem:green claim} for $r=2$}

We proceed as in Lemma~\ref {le first round lemma q} of Section \ref{sub:e112} and apply the idea of an unmatched index. Observe that
\begin{align} \label{green2}
\partial_{jk}^2 \left( F'(X) G_{ij} G_{ki} \right) = F'(X) \, \partial_{jk}^2 (G_{ij} G_{ki}) + 2 \partial_{jk} F'(X) \, \partial_{jk} (G_{ij} G_{ki}) + (\partial_{jk}^2 F'(X))\,G_{ij} G_{ki} \,.
\end{align}

We first consider the expansion of $\partial_{jk}^2 (G_{ij} G_{ki})$. We can easily estimate the terms with four off-diagonal Green function entries, since, for example,
\begin{align*} \begin{split}
\sum_{i, j, k} \left| \E \left[ F'(X) G_{ij} G_{kj} G_{kj} G_{ki} \right] \right| &\leq N^{C\epsilon} \sum_{i, j, k} |G_{ij} G_{kj} G_{kj} G_{ki}| \leq N^{C\epsilon} \left( \frac{\im m}{N\eta_0} \right)^2 \leq N^{-4/3 +C\epsilon}\,,
\end{split} \end{align*}
where we used Lemma \ref{lem:2 off-diagonal}. Thus, for sufficiently small $\epsilon$ and $\epsilon'$,
\begin{align}\label{le tala a1}
 \frac{\e{-t} q_t^{-1}}{N} \sum_{i, j, k} \big|\E \big[ F'(X) G_{ij} G_{kj} G_{kj} G_{ki} \big]\big| \ll N^{2/3 - \epsilon'}\,.
\end{align}
For the terms with three off-diagonal Green function entries, the bound we get from Lemma~\ref{lem:2 off-diagonal} is
$$
q_t^{-1} N^{-1} N^3 N^{C\epsilon} \left( \frac{\im m}{N\eta_0} \right)^{3/2} \sim q_t^{-1} N^{1+C\epsilon}\,,
$$
which is not sufficient. To gain an additional factor of $q_t^{-1}$, which makes the above bound $q_t^{-2} N^{1+C\epsilon} \ll N^{2/3 -\epsilon'}$, we use Lemma~\ref{le stein lemma} to expand in an unmatched index. For example, such a term is of the form
$$
G_{ij} G_{kj} G_{kk} G_{ji}
$$
and we focus on the unmatched index $k$ in $G_{kj}$. Then, multiplying by $z$ and expanding, we get
\begin{align*} \begin{split}
&\frac{q_t^{-1}}{N} \sum_{i, j, k} \E \left[ z F'(X) G_{ij} G_{kj} G_{kk} G_{ji} \right] = \frac{q_t^{-1}}{N} \sum_{i, j, k, n} \E \left[ F'(X) G_{ij} H_{kn} G_{nj} G_{kk} G_{ji} \right] \\
&\qquad\qquad=\frac{q_t^{-1}}{N} \sum_{r'=1}^{\ell} \frac{\kappa^{(r'+1)}}{r'!} \sum_{i, j, k, n} \E \left[ \partial_{kn}^{r'} \left( F'(X) G_{ij} G_{nj} G_{kk} G_{ji} \right) \right] + O(N^{2/3 -\epsilon'})\,,
\end{split} \end{align*}
for $\ell= 10$.

For $r'=1$, we need to consider $\partial_{kn} (F'(X) G_{ij} G_{nj} G_{kk} G_{ji})$. When $\partial_{kn}$ acts on $F'(X)$ it creates a fresh summation index, say $a$, and we get a term
\begin{align*} \begin{split}
&\frac{q_t^{-1}}{N^2} \sum_{i, j, k, n} \E \left[ \left( { \partial_{kn}} F'(X) \right) G_{ij} G_{nj} G_{kk} G_{ji} \right] \\
&\qquad= -\frac{2q_t^{-1}}{N^2} \int_{E_1}^{E_2} \sum_{i, j, k, n, a} \E \big[ G_{ij} G_{nj} G_{kk} G_{ji} F''(X) \im \left( G_{an}(y + L + \ii \eta_0) G_{ka}(y + L + \ii \eta_0) \right) \big] \dd y \\
&\qquad= -\frac{2q_t^{-1}}{N^2} \int_{E_1}^{E_2} \sum_{i, j, k, n, a} \E \big[ G_{ij} G_{nj} G_{kk} G_{ji} F''(X) \im \big( \wt G_{an} \wt G_{ka} \big) \big] \,\dd y\,,
\end{split} \end{align*} 
where we abbreviate $\wt G \equiv G(y + L + \ii \eta_0)$. Applying Lemma \ref{lem:2 off-diagonal} to the index $a$ and $\wt G$, we get
\begin{align*}
\sum_{a=1}^N \big| \wt G_{na} \wt G_{ak} \big| \prec N^{-2/3 + 2\epsilon}\,,
\end{align*}
which also shows that
\begin{align} \label{F derivative bound}
|\partial_{kn} F'(X)| \prec N^{-1/3 + C\epsilon}\,.
\end{align}
Applying Lemma~\ref{lem:2 off-diagonal} to the remaining off-diagonal Green function entries, we obtain that
\begin{align} \label{green2 high order}
 \frac{1}{q_tN^2} \sum_{i, j, k, n} \big|\E \left[ \left( { \partial_{kn}} F'(X) \right) G_{ij} G_{nj} G_{kk} G_{ji} \right] \big| \leq { q_t^{-1} N^{-2} N^{-1/3 +C\epsilon} N^4 N^{-1 + 3\epsilon} = q_t^{-1} N^{2/3 + C\epsilon}\,.}
\end{align}

If $\partial_{kn}$ acts on $G_{ij} G_{nj} G_{kk} G_{ji}$, then we always get four or more off-diagonal Green function entries with the only exception being
$$
-G_{ij} G_{nn} G_{kj} G_{kk} G_{ji}\,.
$$
To the terms with four or more off-diagonal Green function entries, we apply Lemma \ref{lem:2 off-diagonal} and obtain a bound similar to~\eqref{green2 high order} by power counting. For the term of the exception, we rewrite it as
\begin{align} \begin{split}
&-\frac{q_t^{-1}}{N^2} \sum_{i, j, k, n} \E \left[ F'(X) G_{ij} G_{nn} G_{kj} G_{kk} G_{ji} \right] = -\frac{q_t^{-1}}{N} \sum_{i, j, k} \E \left[ m F'(X) G_{ij} G_{kj} G_{kk} G_{ji} \right] \\
&= -\wt m \frac{q_t^{-1}}{N} \sum_{i, j, k} \E \left[ F'(X) G_{ij} G_{kj} G_{kk} G_{ji} \right] + \frac{q_t^{-1}}{N} \sum_{i, j, k} \E \left[ (\wt m -m) F'(X) G_{ij} G_{kj} G_{kk} G_{ji} \right].
\end{split} \end{align}
Here, the last term is again bounded by $q_t^{-1} N^{2/3 + C\epsilon}$ as we can easily check with Proposition \ref{prop:local} and Lemma \ref{lem:2 off-diagonal}. We thus arrive at
\begin{align} \begin{split}
&\frac{q_t^{-1}}{N}(z + \wt m)  \sum_{i, j, k} \E \left[ F'(X) G_{ij} G_{kj} G_{kk} G_{ji} \right] \\
&\qquad\qquad=\frac{q_t^{-1}}{N} \sum_{r'=2}^{\ell} \frac{\kappa^{(r'+1)}}{r'!} \sum_{i, j, k, n} \E \left[ \partial_{kn}^{r'} \left( F'(X) G_{ij} G_{nj} G_{kk} G_{ji} \right) \right] + O(N^{2/3 -\epsilon'})\,.
\end{split} \end{align}
On the right side, the summation is from $r'=2$, hence we have gained a factor $N^{-1} q_t^{-1}$ from $\kappa^{(r'+1)}$ and added a fresh summation index $n$, so the net gain is $q_t^{-1}$. Since $|z + \wt m| \sim 1$, this shows that
\begin{align}
\frac{q_t^{-1}}{N} \sum_{i, j, k} \E \left[ F'(X) G_{ij} G_{kj} G_{kk} G_{ji} \right] = O(N^{2/3 -\epsilon'})\,.
\end{align}
Together with~\eqref{le tala a1}, this takes care of the first term on the right side of~\eqref{green2}.

For the second term on the right side of~\eqref{green2}, we focus on
\begin{align}
\partial_{jk} F'(X) = -\int_{E_1}^{E_2} \sum_{a=1}^N \big[ F''(X) \im \big( \wt G_{ja} \wt G_{ak} \big) \big]\, \dd y
\end{align}
and apply the same argument to the unmatched index $k$ in $\wt G_{ka}$. For the third term, we focus on $G_{ij} G_{ki}$ and again apply the same argument with the index $k$ in $G_{ki}$. We omit the detail.

After estimating all terms accordingly, we eventually get the bound
\begin{align} \label{g2}
\frac{q_t^{-1}}{N} \sum_{i, j, k} \left| \E \left[ \partial_{jk}^2 \left( F'(X) \partial_{jk} G_{ii} \right) \right] \right|  = O(N^{2/3 -\epsilon'})\,.
\end{align} 

\subsubsection{Proof of Lemma~\ref{lem:green claim} for $r=3$}

We proceed as in Section \ref{sub:e113}. (Note that there will be no unmatched indices for this case.) If $\partial_{jk}$ acts on $F'(X)$ at least once, then that term is bounded by
$$
N^{\epsilon} N^{-1} q_t^{-2} N^3 N^{-1/3 + C\epsilon} N^{-2/3 + 2\epsilon} = q_t^{-2} N^{1+C\epsilon} \ll N^{2/3 - \epsilon'},
$$
where we used \eqref{F derivative bound} and the fact that $G_{ij} G_{ki}$ or $\partial_{jk} (G_{ij} G_{ki})$ contains at least two off-diagonal entries. Moreover, in the expansion of $\partial_{jk}^2 (G_{ij} G_{ki})$, the terms with three or more off-diagonal Green function entries entries can be bounded by
$$
N^{\epsilon} N^{-1} q_t^{-2} N^3 N^{C\epsilon} N^{-1+3\epsilon} = q_t^{-2} N^{1+C\epsilon} \ll N^{2/3 - \epsilon'}\,.
$$
Thus,
\begin{align} \begin{split} \label{g32}
\frac{\e{-t}\nmf q_t^{-2}}{3!N} \sum_{i, j, k} \E \left[ \partial_{jk}^3 \left( F'(X) G_{ij} G_{ki} \right) \right] &= -\frac{4!}{2}\frac{\e{-t} \nmf q_t^{-2}}{3!} \sum_{i, j} \E \left[ F'(X) G_{ij}  G_{jj}G_{ji}S_2  \right] \\ &\qquad\qquad+ O(N^{2/3 -\epsilon'})\,,
\end{split} \end{align}
where the combinatorial factor $(4!/2)$ is computed as in Lemma~\ref{le combinatorics lemma} and $S_2$ is as in~\eqref{definition of SN}.
As in Lemma~\ref{le S22 lemma} of Section~\ref{sub:e113}, the first term on right side of~\eqref{g32} is computed by expanding
$$
q_t^{-2} \sum_{i, j} \E \left[ z m S_2 F'(X) G_{ij} G_{jj} G_{ji} \right]
$$
in two different ways, respectively. We can then obtain that
\begin{align} \begin{split} \label{g32'}
q_t^{-2} \sum_{i, j} \E \left[ F'(X) G_{ij} G_{jj} G_{ji} S_2 \right] &= q_t^{-2} \sum_{i, j} \E \left[ m^2 F'(X) G_{ij} G_{jj} G_{ji} \right] + O(N^{2/3 -\epsilon'}) \\
&= q_t^{-2} \sum_{i, j} \E \left[ F'(X) G_{ij} G_{jj} G_{ji} \right] + O(N^{2/3 -\epsilon'})\,,
\end{split}\end{align}
where we used that $m \equiv m(z) = -1 + O(N^{-1/3 + \epsilon})$ with high probability.
Indeed, since $\widetilde m(L)$, which was denoted by $\tau$ in the proof of Lemma~\ref{lem:w}, satisfies $\widetilde m(L) = -1 + O(\e{-t}q_t^{-2})$ by~\eqref{eq:L_+ estimate}, and since $|\widetilde m(z) -\widetilde m(L)| \sim \sqrt{\kappa+\eta_0} \leq N^{-1/3 + \epsilon}$ by~\eqref{L+ square root}, we have from Proposition~\ref{prop:local} that $m(z) = -1 + O(N^{-1/3 + \epsilon})$ with high probability.

Finally, we consider
\begin{align} \label{g32' expand}
q_t^{-2} \sum_{i, j} \E \left[ z F'(X) G_{ij} G_{jj} G_{ji} \right] = 2 q_t^{-2} \sum_{i, j} \E \left[ F'(X) G_{ij} G_{jj} G_{ji} \right] + O(N^{2/3 -\epsilon'})\,.
\end{align}
Expanding the left hand side using~\eqref{expansion indentity}, we also obtain
\begin{align*}
q_t^{-2} \sum_{i, j} \E \left[ z F'(X) G_{ij} G_{jj} G_{ji} \right] = -q_t^{-2} \sum_{i, j} \E \left[ F'(X) G_{ij} G_{ji} \right] + q_t^{-2} \sum_{i, j, k} \E \left[ F'(X) H_{jk} G_{ij} G_{kj} G_{ji} \right]\,.
\end{align*}
Applying Lemma~\ref{le stein lemma} to the second term on the right side, we find that most of the terms are $O(N^{2/3 -\epsilon'})$ either due to three (or more) off-diagonal entries, the partial derivative $\partial_{jk}$ acting on $F'(X)$, or higher cumulants. The only term that does not fall into one these categories is
$$
-\frac{q_t^{-2}}{N^2} \sum_{i, j, k} \E \left[ F'(X) G_{ij} G_{kk} G_{jj} G_{ji} \right]\,,
$$
which is generated when $\partial_{jk}$ acts on $G_{kj}$. From this argument, we find that
\begin{align*} \begin{split}
&q_t^{-2} \sum_{i, j} \E \left[ z F'(X) G_{ij} G_{jj} G_{ji} \right] \\
&\qquad= -q_t^{-2} \sum_{i, j} \E \left[ F'(X) G_{ij} G_{ji} \right] - q_t^{-2} \sum_{i, j} \E \left[ m F'(X) G_{ij} G_{jj} G_{ji} \right] + O(N^{2/3 -\epsilon'})\,.
\end{split} \end{align*}
Hence, combining it with \eqref{g32' expand} and the fact that $m = -1 + O(N^{-1/3 + \epsilon})$ with high probability, we get
\begin{align*}
q_t^{-2} \sum_{i, j} \E \left[ F'(X) G_{ij} G_{jj} G_{ji} \right] = -q_t^{-2} \sum_{i, j} \E \left[ F'(X) G_{ij} G_{ji} \right]+ O(N^{2/3 -\epsilon'})\,.
\end{align*}
In combination with~\eqref{g32}~and~\eqref{g32'}, we conclude that
\begin{align} \label{g3}
\frac{\e{-t}}{N} \frac{\nmf q_t^{-2}}{3!} \sum_{i, j, k} \E \left[ \partial_{jk}^3 \left( F'(X) G_{ij} G_{ki} \right) \right] = 2 \e{-t} \nmf q_t^{-2} \sum_{i, j} \E \left[ F'(X) G_{ij} G_{ji} \right] + O(N^{2/3 -\epsilon'})\,.
\end{align}

\subsubsection{Proof of Lemma~\ref{lem:green claim} for $r=4$}

In this case, we estimate the term as in the case $r=2$ and get
\begin{align} \label{g4}
\frac{q_t^{-3}}{N} \sum_{i, j, k} \left| \E \left[ \partial_{jk}^4 \left( F'(X) G_{ij} G_{ki} \right) \right] \right|  = O(N^{2/3 -\epsilon'})\,.
\end{align} 
We leave the details to the interested reader.


\begin{thebibliography}{00}


\bibitem{AEK} Ajanki, O., Erd\H{o}s, L., Kr\"uger, T.: \emph{Universality for General Wigner-type Matrices}, arXiv:1506.05098 (2015).

\bibitem{BES15b} Bao, Z.\ G., Erd\H{o}s, L., Schnelli, K.: \emph{Local Law of Addition of Random Matrices on Optimal Scale}, arXiv:1509.07080 (2015).

\bibitem{BKY} Bauerschmidt, R., Knowles, A., Yau, H.-T.: \emph{Local Semicircle Law for Random Regular Graphs}, arXiv:1503.08702 (2015).

\bibitem{BGM} Benaych-Georges, F., Guionnet, A., Male, C.: \emph{Central Limit Theorems for Linear Statistics of Heavy Tailed Random Matrices}, Commun.\ Math.\ Phys.\ \textbf{329(2)}, 641-686 (2014).
 
\bibitem{BS} Bickel, P.\ J., Sarkar, P.: \emph{Hypothesis Testing for Automated Community Detection in Networks}, J.\ R.\ Statist.\ Soc.~B \textbf{78}, 253-273 (2016).

\bibitem{CMS} Cacciapuoti, C., Maltsev, A., Schlein, B.: \emph{Bounds for the Stieltjes Transform and the Density of States of Wigner Matrices}, Probab.\ Theory Rel.\ Fields \textbf{163(1-2)}, 1-59  (2015).
 
\bibitem{ER} Erd\H{o}s, P., R\'{e}nyi, A.: \emph{On Random Graphs I.}, Publ. Math. \textbf{6}, 290-297 (1959).

\bibitem{ERG2} Erd\H{o}s, P., R\'{e}nyi, A.: \emph{On the Evolution of Random Graphs}, Publ.\ Math.\ Inst.\ Hungar.\ Acad.\ Sci.~\textbf{5}, 17-61 (1960).

\bibitem{E} Erd\H{o}s, L.: \emph{Universality of Wigner Random Matrices: A Survey of Recent Results}, Russian Math.\ Surveys \textbf{66(3)}, 507  (2011).
 
\bibitem{EKYY1} Erd\H{o}s, L., Knowles, A., Yau, H.-T., Yin, J.: \emph{Spectral Statistics of Erd\H{o}s-R\'enyi Graphs I: Local Semicircle Law}, Ann.\ Probab. \textbf{41}, 2279-2375 (2013).

\bibitem{EKYY2} Erd\H{o}s, L., Knowles, A., Yau, H.-T., Yin, J.: \emph{Spectral Statistics of Erd{\H o}s-Rényi Graphs II: Eigenvalue Spacing and the Extreme Eigenvalues}, Commun.\ Math.\ Phys.\ \textbf{314(3)}, 587-640 (2012).

\bibitem{EKYY13} Erd\H{o}s, L., Knowles, A., Yau, H.-T., Yin, J.: \emph{The Local Semicircle Law for a General Class of Random Matrices}. Electron.\ J.\ Probab.\  \textbf{18(59)}, 1-58 (2013). 

\bibitem{EKY} Erd\H{o}s, L., Knowles, A., Yau, H.-T.: \emph{Averaging Fluctuations in Resolvents of Random Band Matrices}, Ann.\ Henri Poincar\'{e} \textbf{14}, 1837-1926 (2013). 

\bibitem{ESY1} Erd\H{o}s, L., Schlein, B., Yau, H.-T.: \emph{Semicircle Law on Short Scales and Delocalization of Eigenvectors for Wigner Random Matrices}, Ann.\ Probab. \textbf{37}, 815-852 (2009).

\bibitem{ESY2} Erd\H{o}s, L., Schlein, B., Yau, H.-T.: \emph{Local Semicircle Law and Complete Delocalization for Wigner Random Matrices}, Commun.\ Math.\ Phys.\ \textbf{287}, 641-655 (2009).

\bibitem{ESY3} Erd\H{o}s, L., Schlein, B., Yau, H.-T.: \emph{Wegner Estimate and Level Repulsion for Wigner Random Matrices}, Int.\ Math.\ Res.\ Notices. \textbf{2010}, no.~3, 436-479 (2010).

\bibitem{EYY2} Erd{\H o}s, L.,  Yau, H.-T., Yin, J.: \emph{Bulk Universality for Generalized Wigner Matrices}, Probab.\ Theory Rel.\ Fields {\bf 154}, no. 1-2, 341-407 (2012).

\bibitem{EYY3} Erd\H{o}s, L., Yau, H.-T., Yin, J.: \emph{Universality for Generalized Wigner Matrices with Bernoulli Distribution}, J.\ Comb.\ \textbf{2}, 15-82 (2012).

\bibitem{EYY} Erd\H{o}s, L., Yau, H.-T., Yin, J.: \emph{Rigidity of Eigenvalues of Generalized Wigner Matrices}, Adv.\ Math.\ \textbf{229}, 1435-1515 (2012).

\bibitem{Gi} Gilbert, E.\ N.: \emph{Random Graphs}, Ann.\ Math.\ Statist.\ \textbf{30(4)}, 1141-1144 (1959).

\bibitem{GNT} G\"otze, F., Naumov, A., Tikhomirov, A.: \emph{Local Semicircle Law under Moment Conditions. Part I: The Stieltjes Transform}, arXiv:1510.07350 (2015).

\bibitem{HW} Hanson, D.\ L., Wright, E.\ T.: \emph{A Bound on Tail Probabilities for Quadratic Forms in Independent Random Variables}, Ann.\ Math.\ Statist.\ \textbf{42}, 1079-1083 (1971).

\bibitem{HLY} Huang, J., Landon, B., Yau, H.-T.: \emph{Bulk Universality of Sparse Random Matrices},  J.\ Math.\ Phys.\ \textbf{56(12)}, 123301 (2015).

\bibitem{KKP} Khorunzhy, A., Khoruzhenko, B., Pastur, L.: \emph{Asymptotic Properties of Large Random Matrices with Independent Entries},  J.\ Math.\  Phys.\ \textbf{37(10)}, 5033-5060 (1996).

\bibitem{Kho1} Khorunzhy, A.: \emph{Sparse Random Matrices: Spectral Edge and Statistics of Rooted Trees}, Adv.\  Appl.\ Probab.\ \textbf{33(1)}, 124-140, (2001).

\bibitem{Kho2} Khorunzhiy, O.: \emph{On High Moments and the Spectral Norm of Large Dilute Wigner Random Matrices}, Zh.\ Mat.\ Fiz.\ Anal.\ Geom.\ \textbf{10(1)}, 64-125 (2014).

\bibitem{LY} Lee, J.\ O., Yin, J.: \emph{A Necessary and Sufficient Condition for Edge Universality of Wigner Matrices},  Duke Math.\ J.\ \textbf{163(1)}, 117-173, (2014). 

\bibitem{LS} Lee, J.\ O., Schnelli, K.: \emph{Local Deformed Semicircle Law and Complete Delocalization for Wigner Matrices with Random Potential}, J.\ Math.\ Phys. \textbf{54} 103504 (2013).

\bibitem{LS14} Lee, J.\ O., Schnelli, K.: \emph{Edge Universality for Deformed Wigner Matrices}, Rev.\ Math.\ Phys.\ \textbf{27(8)}, 1550018 (2015).

\bibitem{LS14b} Lee, J.\ O., Schnelli, K.: \emph{Tracy--Widom Distribution for the Largest Eigenvalue of Real Sample Covariance Matrices with General Population}, arXiv:1409.4979 (2014).

\bibitem{LSSY} Lee, J.\ O., Schnelli, K., Stetler, B., Yau, H.\-T.: \emph{Bulk Universality for Deformed Wigner Matrices}, Ann.\ Probab.\, \textbf{44(3)}, 2349-2425 (2016).

\bibitem{Lei} Lei, J.: \emph{A Goodness-of-fit Test for Stochastic Block Models}, Ann.\ Statist.\ \textbf{44}, 401-424 (2016).
 
\bibitem{LP} Lytova, A., Pastur, L.: \emph{Central Limit Theorem for Linear Eigenvalue Statistics of Random Matrices with Independent Entries}, Ann.\ Probab.\ \textbf{37}, 1778-1840 (2009).
 
\bibitem{PeSo1} P\'{e}ch\'{e}, S., Soshnikov, A.: \emph{On the Lower Bound of the Spectral Norm of Symmetric Random Matrices with Independent Entries}, Electron.\ Commun.\ Probab.\ \textbf{13}, 280–290 (2008).

\bibitem{PeSo2} P\'{e}ch\'{e}, S., Soshnikov, A.: \emph{Wigner Random Matrices with Non-Symmetrically Distributed Entries}, J.\ Stat.\ Phys.\ \textbf{129}, 857–884 (2007).

\bibitem{ShTi} Shcherbina, M., Tirozzi, B.: \emph{Central Limit Theorem for Fluctuations of Linear Eigenvalue Statistics of Large Random Graphs: Diluted Regime}, J.\ Math.\ Phys.\ \textbf{53}, 043501 (2012).
 
\bibitem{SiSo1} Sinai, Y., Soshnikov, A.: \emph{A Refinement of Wigner’s Semicircle Law in a Neighborhood of the Spectrum Edge}, Functional Anal.\ and Appl.\ \textbf{32}, 114–131 (1998).
 
\bibitem{So1} Soshnikov, A.: \emph{Universality at the Edge of the Spectrum in Wigner Random Matrices}, Commun.\ Math.\ Phys.\ \textbf{207}, 697-733 (1999).

\bibitem{TV1} Tao, T., Vu, V.: \emph{Random Matrices: Universality of the Local Eigenvalue Statistics}, Acta Math \textbf{206}, 127-204 (2011).

\bibitem{TV2} Tao, T., Vu, V.: \emph{Random Matrices: Universality of Local Eigenvalue Statistics up to the Edge}, Commun.\ Math.\ Phys.\ \textbf{298}, 549-572 (2010).
 
\bibitem{TW1} Tracy, C., Widom, H.: \emph{Level-Spacing Distributions and the Airy Kernel}, Commun.\ Math.\ Phys.\ \textbf{159}, 151-174 (1994).

\bibitem{TW2} Tracy, C., Widom, H.: \emph{On Orthogonal and Symplectic Matrix Ensembles}, Commun.\ Math.\ Phys.\ \textbf{177}, 727-754 (1996).

 
\end{thebibliography}
\end{document}